\documentclass[11pt,reqno]{amsart}
\usepackage{xcolor}
\usepackage[T1]{fontenc}                          
\usepackage{amsfonts}
\usepackage[utf8]{inputenc}                       
\usepackage{comment}                              
\usepackage{mparhack}                             
\usepackage{amsmath,amssymb,amsthm,mathrsfs,eucal,mathtools}
\usepackage{enumerate}
\usepackage{enumitem}
\usepackage{booktabs}                             
\usepackage{graphicx}
\usepackage{subcaption}
\usepackage{cancel}
\usepackage{wrapfig}                            
\usepackage[bookmarks=true, colorlinks=true]{hyperref}
\hypersetup{urlcolor=blue,citecolor=red,linkcolor=black}
\usepackage{cleveref}
\usepackage{soul}

\usepackage{appendix}

\usepackage{bm}                                   
\usepackage{dsfont}

\newtheorem{defn}{Definition}[section]
\newtheorem{thm}{Theorem}[section]
\newtheorem{prop}{Proposition}[section]
\newtheorem{lem}{Lemma}[section]
\newtheorem{cor}{Corollary}[section]

\newtheorem{rem}{Remark}[section]

\numberwithin{equation}{section}

\DeclareMathOperator*{\argmin}{argmin}
\newcommand{\me}{\mathcal{E}}
\newcommand{\mP}{{\mathcal{P}}}

\newcommand{\mf}{\mathcal{F}}
\newcommand{\JKOstep}{\mathcal{A}}

\newcommand{\mptrd}{{\mathcal{P}_2(\Rd)}}

\newcommand{\mpdtard}{{\mathcal{P}_{2}^a(\Rd)}}
\newcommand{\mw}{\mathcal{W}}
\newcommand{\rhot}{\rho_{\tau}}

\newcommand{\etat}{\eta_{\tau}}
\newcommand{\rhotk}{\rho_{\tau}^{k}}

\newcommand{\etatk}{\eta_{\tau}^{k}}
\newcommand{\rhotkk}{\rho_{\tau}^{k+1}}

\newcommand{\etatkk}{\eta_{\tau}^{k+1}}
\newcommand{\sigmat}{\sigma_\tau}
\newcommand{\sigmatk}{\sigma_\tau^k}
\newcommand{\sigmatkk}{\sigma_\tau^{k+1}}
\newcommand{\R}{\mathbb{R}}
\newcommand{\Rd}{{\mathbb{R}^{d}}}
\newcommand{\Rn}{{\mathbb{R}^{n}}}

\newcommand{\dx}{\mathrm{d}x}
\newcommand{\dt}{\mathrm{d}t}
\newcommand{\ddt}{\frac{\mathrm{d}}{\mathrm{d}t}}
\newcommand{\dW}{d_W}

\newcommand{\supp}{\mathrm{supp}}

\newcommand{\dive}{\mbox{div}}

\makeatletter
\@namedef{subjclassname@2020}{\textup{2020} Mathematics Subject Classification}
\makeatother

\def\XXint#1#2#3{{\setbox0=\hbox{$#1{#2#3}{\int}$}
  \vcenter{\hbox{$#2#3$}}\kern-.5\wd0}}

\definecolor{cadmiumgreen}{rgb}{0.0, 0.42, 0.24}

\title[Competing effects in fourth-order aggregation-diffusion PDEs]{Competing effects in fourth-order aggregation-diffusion equations}
\author[J.A. Carrillo, A. Esposito, C. Falcó, A. Fernández-Jiménez]{José Antonio Carrillo \and Antonio Esposito \and Carles Falcó \and Alejandro Fernández-Jiménez}

\address{Mathematical Institute, University of Oxford, Woodstock Road, Oxford, OX2 6GG, United Kingdom.}

\email{carrillo@maths.ox.ac.uk}
\email{antonio.esposito@maths.ox.ac.uk}
\email{falcoigandia@maths.ox.ac.uk}
\email{alejandro.fernandezjimenez@maths.ox.ac.uk}


\date{}

\begin{document}

\begin{abstract}
We give sharp conditions for global in time existence of gradient flow solutions to a Cahn-Hilliard-type equation, with backwards second order degenerate diffusion, in any dimension and for general initial data. Our equation is the 2-Wasserstein gradient flow of a free energy with two competing effects: the Dirichlet energy and the power-law internal energy. Homogeneity of the functionals reveals critical regimes that we analyse. Sharp conditions for global in time solutions, constructed via the minimising movement scheme, also known as JKO scheme, are obtained. Furthermore, we study a system of two Cahn-Hilliard-type equations exhibiting an analogous gradient flow structure.
\end{abstract}

\maketitle


\section{Introduction}

In this manuscript we are interested in the mathematical analysis of the equation
\begin{equation}\label{eq:thin_film_intro}
    \partial_t \rho=-\, \dive (\rho \nabla(\Delta \rho))-\chi\Delta \rho^{m},
\end{equation}
where $m\geq 1$, and its extension to systems. We look for solutions of \eqref{eq:thin_film_intro} in the set of probabity densities, $\rho\in L^1_+(\mathbb{R}^d):=\{\rho\in L^1(\Rd): \rho\ge0\}$, thus setting the mass to one in the sequel without loss of generality. The parameter $\chi>0$ measures the relative balance between aggregation, modelled by backwards degenerate diffusion, and repulsion, modelled by fourth-order diffusion. 
The case of general masses can be reduced to~\eqref{eq:thin_film_intro} with a suitable parameter $\chi$ upon a standard time rescaling and mass normalisation, cf.~Remark~\ref{rem: critical mass}.

Equation~\eqref{eq:thin_film_intro} is related to the classical thin-film equations from lubrication theory, cf.~\cite{HOCHERMAN_ROSENAU_93,Bertozzi_Pugh_Nonlinearity_94,Otto_CPDE98,Ber98,DalPasso_Giacomelli_Shishkov_CPDE01,Grun_CPDE04} and the references therein. Starting from a conjecture of Hocherman and Rosenau,~\cite{HOCHERMAN_ROSENAU_93},
the authors in~\cite{Bertozzi_Pugh_CPAM98} study well-posedness and finite-time singularities of Cahn-Hiliard-type equations, in one spatial dimension on bounded interval with periodic boundary conditions. More precisely, they analyse the family of equations of the form
\begin{equation}\label{eq:ch_mobilities}
\partial_t\rho=-(\rho^n\rho_{xxx})_x-(\rho^{m-1}\rho_x)_x,
\end{equation}
proving that for nonnegative (weak) solutions, blow-up can only occur for $m\ge n+3$. The results in \cite{HOCHERMAN_ROSENAU_93,Bertozzi_Pugh_CPAM98} hold for general degenerate mobilities, as in~\cite[Conjectures 1,2]{Bertozzi_Pugh_CPAM98}. Afterwards, several contributions to the analysis of the one dimensional problem have been made. Linear ins/stability of steady states for the one-dimensional periodic problem was analysed in \cite{Laugesen_Pugh_Arma00,Slepcev_IntFree_09}. Using the dissipation of a suitable energy functional, the authors of~\cite{Laugesen_Pugh_JDE22} were able to further characterise the energy landscape distinguishing between local minima and saddles among periodic steady states. Stability of droplets steady states with a fixed contact angle for the one-dimensional periodic problem was further studied in \cite{Laugesen_Pugh_EJAM2000}.

The critical case $m=n+3$ in one dimension is analysed in \cite{Witel_Bern_Bert_EJAM04}, where blow-up in finite time can only happen above a certain critical mass identified thanks to a sharp Sz.-Nagy inequality, cf.~\cite{Nagy_1941, Inequalities_book}. Existence of selfsimilar blow-up solutions of \eqref{eq:ch_mobilities} is explored in \cite{Slepcev_Pugh_05} for the critical case $m=n+3$. In particular, for $n=1$, there exists a family of blowing-up symmetric selfsimilar solutions with zero contact angle. Further analysis of one-dimensional self-similar solutions, both expanding and blowing-up, for the critical cases of \eqref{eq:ch_mobilities} has been done in \cite{EGK07a,EGK07b,Slepcev_IntFree_09}.

The nonlinear Cahn-Hilliard-type equations \eqref{eq:thin_film_intro} have also been recently proposed as approximations of nonlocal aggregation-diffusion models of the form
\begin{equation}\label{eq:KS}
    \frac{\partial \rho}{\partial t} = \Delta \rho^s + \dive (\rho \nabla (W\ast\rho)), \quad s \geq 1,
\end{equation}
by truncation of the Fourier expansion of the interaction potential $W$, see~\cite{BernoffTopazCH}. This approximation has been rendered rigorous under certain assumptions on the interaction potential $W$ in~\cite{elbar2022degenerate}. 

The connection between aggregation-diffusion and Cahn-Hilliard equations has also been generalised to systems of aggregation-diffusion equations modelling tissue growth and patterning due to cell-cell adhesion~\cite{carrillo2019population}. The authors in~\cite{falco2022local} show that cell-sorting phenomena are kept for the resulting system of equations:
\begin{subequations}\label{eq:two species}    
\begin{align}
        \partial_t\rho &= -\dive\left(\rho\nabla\left(\kappa\Delta\rho + \alpha\Delta\eta + \beta\rho+\omega\eta\right)\right), \label{eq: two species rho}
        \\
        \partial_t\eta & =-\dive\left(\eta\nabla\left(\alpha\Delta\rho + \Delta\eta + \omega\rho+\eta\right)\right). \label{eq: two species eta}
\end{align}
\end{subequations}
The parameters in the model are such that $\beta,\omega\in\mathbb{R}$ and the matrix 
\begin{equation*}
   A = \begin{pmatrix}
    \kappa & \alpha
    \\
    \alpha & 1
    \end{pmatrix},
\end{equation*}
is positive definite. We extend the theory developed for the one-species case \eqref{eq:thin_film_intro} to construct solutions to the systems of equations \eqref{eq:two species}.
The nonlocal-to-local limit in the context of systems has also been studied rigorously in \cite{carrillo2023degenerate}. We also mention that different multi-species Cahn-Hilliard equations are considered in~\cite{elliott1991generalised,elliott1997diffusional,Ehrla_DiMa_Pietschamm21} and references therein.

Equation \eqref{eq:thin_film_intro} can be interpreted as $2$-Wasserstein gradient flow of the (extended) energy functional 
\begin{equation}\label{eq:functional} 
\mf_m[\rho] = \begin{cases}\frac{1}{2}\displaystyle\int_\Rd|\nabla\rho(x)|^2\,\dx-\chi\me_m[\rho], &\rho\in\mathcal{P}^a(\mathbb{R}^d),\, \nabla\rho \in L^2(\mathbb{R}^d),\\
+\infty, & \mbox{otherwise},
\end{cases}
\end{equation}
as already noted in \cite{Slepcev_IntFree_09}, being
\begin{equation}\label{eq:entropy}
    \me_m[\rho]=
    \begin{cases}
        \displaystyle\int_\Rd\rho(x)\log\rho(x)\,\dx, \quad &m=1,\\
        \frac{1}{m-1}\displaystyle\int_\Rd\rho^m\dx, &m>1.
    \end{cases}
\end{equation}
This gradient flow structure was made rigorous for related Cahn-Hilliard equations in \cite{MMS09, Lisini_Matthes_Savare12}. However, the former does not include the second-order backwards diffusion term in \eqref{eq:thin_film_intro}, while the latter is concerned with more general, density-dependent, mobilities. 

As for the multi-species case, by defining the free energy functional as
\begin{equation*}
    \Tilde{\bm{\mf}}[\rho,\eta] = \int_{\Rd}\left(\frac{\kappa}{2}|\nabla\rho|^2+\frac{1}{2}|\nabla\eta|^2 + \alpha\nabla\rho\cdot\nabla\eta-\frac{\beta}{2}\rho^2-\frac{1}{2}\eta^2-\omega\rho\eta\right)\,\dx,
\end{equation*}
system \eqref{eq:two species} can be written as a 2-Wasserstein gradient flow with respect to the (extended) free energy functional
\begin{equation}\label{eq:functiona_two_species}
    \bm{\mf}[\rho,\eta] = \begin{cases}
      \Tilde{\bm{\mf}}[\rho,\eta] & \mbox{if } (\rho,\eta)\in\mathcal{P}^a(\mathbb{R}^d)^2,\, (\nabla\rho,\nabla\eta) \in L^2(\mathbb{R}^d)^2
        \\
        +\infty & \mbox{otherwise.} 
    \end{cases}
\end{equation}

Our main goal is to show global existence of weak solutions of \eqref{eq:thin_film_intro} for $m < m_c:=2 + \frac{2}{d}$ (subcritical case) and for $m = m_c$ (critical case) for \emph{subcritical 
 mass}, $0<\chi<\chi_c$, by leveraging the aforementioned gradient flow structure. The critical parameter $\chi_c$ is identified by the sharp constant of a suitable functional inequality \cite{liu2017generalized}. The critical exponent $m_c$ is determined by scaling arguments using mass-preserving dilations of densities in the energy functional \eqref{eq:functional}. Moreover, we also obtain global existence of weak solutions for the system \eqref{eq:two species} by an analogous approach. In fact, we employ the (by now) classical variational minimising movement scheme, or JKO scheme,~\cite{JKO98,AGS} to obtain an approximation of a candidate solution. A crucial step will be to use the flow interchange technique, developed in \cite{MMS09, Lisini_Matthes_Savare12}, to gain suitable regularity. Afterwards, we check that limits of the variatonal scheme are indeed weak solutions in any dimension. 

Our main result provides sharp conditions on the exponent of backwards diffusion in \eqref{eq:thin_film_intro} to ensure global existence of solutions in the natural class of initial data for any dimension compared to previous literature \cite{Witel_Bern_Bert_EJAM04,Slepcev_IntFree_09,Lisini_Matthes_Savare12,LW17}.

The key ingredient to take advantage of the gradient flow structure of \eqref{eq:thin_film_intro} and system \eqref{eq:two species} is to have uniform bounds on the competing terms in the free energies \eqref{eq:functional} and \eqref{eq:functiona_two_species}, respectively. Interestingly, this is reminiscent of similar arguments developed for generalisations of the Patlak-Keller-Segel equation for chemotactic cell movement \cite{blanchet2009critical,Carrillo_Craig_Yao19,CarrilloKe21}. Actually, we can draw a nice parallelism with this well studied problem. Generalised Patlak-Keller-Segel equations are of the form \eqref{eq:KS}. In particular, let us focus on the power-law kernel
\begin{equation*}
    W_k (x) = 
    \begin{dcases}
		\frac{|x|^k}{k} & \text{if } k \neq 0, \\
		\log |x| &  \text{if } k = 0.
	\end{dcases}
\end{equation*}
We find an immediate connection with the problem \eqref{eq:thin_film_intro}. Analogously to the case we are studying in this work, there exists a critical exponent, $s_c = 1 - \frac{k}{d}$, also found via mass-preserving dilations on the corresponding free energy functional which characterises the behaviour of \eqref{eq:KS}. 

The case $s > s_c$ is the diffusion dominated regime and global well-posedness for~\eqref{eq:KS} is expected, see for instance \cite{Calvez_Carrillo06,Sugiyama07,blanchet2009critical,Calvez_Carrillo_Hoffmann17a, Calvez_Carrillo_Hoffmann17b,CHMV18}. This is analogous to the case $1 \leq m < m_c$ for \eqref{eq:thin_film_intro}.

As for the range $1 \leq s < s_c$, aggregation-dominated regime for Eq.~\eqref{eq:KS} --- analogous to the case $m > m_c$ for \eqref{eq:thin_film_intro} --- coexistence of blow-up and global existence depending on the initial data is expected, see 
\cite{LS07,Bedrossian_Rodriguez_Bertozzi11,CW14} for instance.

In the fair competition regime $s = s_c$ --- analogous to our critical exponent $m=m_c$ --- there exists a dichotomy between aggregation and diffusion in terms of the \emph{initial mass}: $M$, analogous to our parameter $\chi$. Sharp constants of variants of Hardy-Littlewood-Sobolev type inequalities determine the critical value of the mass $M_c$ for \eqref{eq:KS}, analogously to our critical parameter $\chi_c$. We note that for our fourth-order Cahn-Hilliard type equation, the crucial functional inequality was established in \cite{liu2017generalized}. In the supercritical mass case, $M > M_c$, there exists solutions that blow-up in finite time, see for instance \cite{blanchet2009critical,Bedrossian11, Bedrossian_Rodriguez_Bertozzi11,Calvez_Carrillo_Hoffmann17a, Calvez_Carrillo_Hoffmann17b}. In the subcritical mass case, $M < M_c$, global existence of solutions is shown and spreading self-similar solutions are expected to attract the long time dynamics, see for instance \cite{DP04,BDP06,Bedrossian_Rodriguez_Bertozzi11,Bedrossian11,Calvez_Carrillo_Hoffmann17a, Calvez_Carrillo_Hoffmann17b}.

In the critical case $M = M_c$, there are infinitely many stationary states given by the optimisers of the variants of the HLS inequalities, solutions are globally well-posed blowing-up at infinite time for bounded second moment initial data, and local stability of stationary solutions is expected, see \cite{DP04,blanchet2009critical,BCC12,Yao14,Calvez_Carrillo_Hoffmann17a, Calvez_Carrillo_Hoffmann17b}.

We shall perform a parallel study to nonlinear Keller-Segel equations~\eqref{eq:KS} for our family of Cahn-Hilliard equations~\eqref{eq:thin_film_intro}, depending on the critical exponent case $m_c$ and parameter $\chi$. 

Finally, we want to emphasize that our work sets the path to many other interesting open questions. Uniqueness is widely open being the functionals not convex, even in subsets, in any obvious manner. Existence of minimisers in the subcritical case in the whole space is not clear since we do not know at present how to bound uniformly in time the second moment or any other quantity controlling escape of mass at infinity. Long time asymptotics are, in turn, widely open in all global existence cases. Free boundary problem techniques could help understand if the evolution leads to compactly supported solutions corresponding to compactly supported initial data. This conjecture is corroborated by numerical experiments being this another challenging problem. In the two-species case, we can identify other interesting issues such as sharp segregation for specific parameter values between the two species not only at steady states but along their evolutions. This information is important for the applications in mathematical biology \cite{carrillo2019population,falco2022local}. 

We structure the paper as follows. Section 2 is devoted to the precise statements of the main theorems together with some preliminary material used in the sequel. We will analyse the existence of global minimisers of the energy~\eqref{eq:functional} following the strategies in \cite{DP04,blanchet2009critical,Calvez_Carrillo_Hoffmann17a} in Section 3. In Section 4 we deal with the core main result of global existence of weak solutions to the single equation~\eqref{eq:thin_film_intro} in any dimension for generic initial data. Finally, Section 5 focuses on the generalisation of this approach to the case of systems of the form~\eqref{eq:two species}.


\section{Main results and preliminaries}\label{sec:preliminaries}

We begin by listing the main results covered in this manuscript. First we study some properties of the free energy functional $\mf_m$ and its minimisers. The following theorem summarises the results proven in Section~\ref{sec:Energy minimisers}.
\begin{thm}\label{thm:energy}
   Let $\mf_m$ be as in~\eqref{eq:functional}. The following holds:
   \begin{enumerate}[label=$(\arabic*)$]
       \item\label{thm:subcritical_exp} If $1 \leq m < 2 + 2/d$, then $\mf_m$ is bounded from below.
       \item\label{thm:critical_exp} If $m = 2 +2/d$, then, for the subcritical and critical mass regimes, $\mf_m$ is bounded from below. Furthermore, for the critical mass, the infimum is achieved. In the supercritical mass regime, $\mf_m$ is unbounded from below. 
       \item\label{thm:supercritical_exp} If $m > 2+2/d$, then $\mf_m$ is unbounded from below.
   \end{enumerate}
\end{thm}

Case \ref{thm:subcritical_exp} is proven in \Cref{prop:Some properties}. Case \ref{thm:critical_exp} is a combination of two results. In \Cref{prop:prop_energies_critical} we show that for $\chi \leq \chi_c$ the free energy is bounded from below. In \Cref{prop: infimum free energy}, we prove that the infimum is achieved for critical mass and that the free energy is unbounded if $\chi > \chi_c$. Finally, in \Cref{prop: free energy supercritical} we show case \ref{thm:supercritical_exp}. 

Throughout the manuscript, we denote by $\mP(\Rd)$ the set of probability measures on $\Rd$, for $d\in\mathbb{N}$, and by $\mP_2(\Rd):=\{\rho\in\mP(\Rd):\mathrm{m}_2(\rho)<+\infty\}$, being
$\mathrm{m}_2(\rho):=\int_\Rd|x|^2\,\mathrm{d}\rho(x)$ the $2^{\mathrm{nd}}$-order moment of $\rho$. We shall use $\mP^a(\Rd)$ and $\mP_2^a(\Rd)$ for elements in $\mP(\Rd)$ and $\mP_2(\Rd)$ which are absolutely continuous with respect to the Lebesgue measure.

The second result we prove is the existence of weak solutions to~\eqref{eq:thin_film_intro}, in the following sense:

\begin{defn}[Weak solution]\label{def:weak solution}
    A weak solution to~\eqref{eq:thin_film_intro} on the time interval $[0,T]$, with initial datum $\rho_0\in\mpdtard$ such that $\nabla\rho_0\in L^2(\Rd)$, is a narrowly continuous curve $\rho:[0,T]\to\mptrd$ satisfying the following properties:
    \begin{enumerate}[label=$\roman*$)]
        \item $\rho\in L^\infty([0,T];L^p(\Rd))\cap L^\infty([0,T];H^1(\Rd))\cap L^2([0,T];H^2(\Rd))$, for any $p\in[1,2^*]$, where $2^* =+\infty$ if $d = 1,2$ and $2^* = 2d/(d-2)$ if $d\ge 3$;
        \item for every $\varphi \in  C_c^2(\Rd)$ and every $0 \leq s_1 < s_2 \leq T$ it holds
    \begin{align*}
    \int_{\Rd}  \varphi (x) \rho (s_2, x) \, \dx & = \int_{\Rd} \varphi (x) \rho (s_1, x) \, \dx \\
    & \quad - \int_{s_1}^{s_2}  \int_\Rd\left( \rho \Delta \rho \Delta \varphi + \Delta \rho \nabla \rho \cdot \nabla \varphi \right)\,\dx \, \dt \\
    & \quad -\chi \int_{s_1}^{s_2} \int_\Rd  \rho^m \Delta \varphi \,\dx \, \dt.
\end{align*}
    \end{enumerate}
\end{defn}

\begin{thm}\label{thm:main_result_existence}
    Assume $1\le m<2+2/d$ or $m=2+2/d$ with subcritical mass $\chi<\chi_c$ and let $\rho_0\in\mathcal{P}_2^a(\Rd)$ be an initial datum such that $\mf_m[\rho_0]<+\infty$. Then there exists a weak solution to~\eqref{eq:thin_film_intro}.
\end{thm}
We extend our results from the one species case to construct weak solutions to system~\eqref{eq:two species}, in the following sense. 
\begin{defn}[Weak solution for the system]\label{def:Weak solution two species}
    A weak solution to \eqref{eq:two species} on the time interval $[0,T]$, with initial datum $\sigma_0 = (\rho_0, \eta_0) \in \mP_2^a(\Rd)^2$ such that $\nabla \rho_0, \nabla \eta_0 \in L^2(\Rd)$, consists of a pair of narrowly continuous curves $\rho, \eta: [0,T] \rightarrow \mP_2 (\Rd)$ satisfying the following properties:
    \begin{enumerate}[label=$\roman*$)]
        \item $\rho, \eta \in L^\infty([0,T];L^p(\Rd))\cap L^\infty([0,T];H^1(\Rd))\cap L^2([0,T];H^2(\Rd))$, for any $p\in[1,2^*]$, where $2^* =+\infty$ if $d = 1,2$ and $2^* = 2d/(d-2)$ if $d\ge 3$;
        \item for every $\varphi, \psi \in C^2_c(\Rd)$ and every $0 \leq s_1 < s_2 \leq T$ it holds
        \begin{align*}
        \int_{\Rd} \varphi (x) \rho(s_2 , x) \, \dx & = \int_{\Rd} \varphi (x) \rho(s_1, x) \, \dx \\
        & \quad - \kappa \int_{s_1}^{s_2} \int_{\Rd} \rho \Delta \rho \Delta \varphi + \nabla \rho\cdot\nabla \varphi \Delta \rho\, \dx \, \dt \\
        & \quad - \alpha \int_{s_1}^{s_2} \int_{\Rd} \rho \Delta \eta \Delta \varphi + \nabla \rho \cdot \nabla \varphi \Delta \eta\, \dx \, \dt \\
        & \quad - \frac{\beta}{2} \int_{s_1}^{s_2} \int_{\Rd} \rho^2 \Delta \varphi \, \dx \, \dt + \omega \int_{s_1}^{s_2} \int_{\Rd} \rho \nabla \eta\cdot \nabla \varphi \, \dx \, \dt ,        \end{align*}
        \begin{align*}
        \int_{\Rd} \psi (x) \eta(s_2 , x) \, \dx & = \int_{\Rd} \psi (x) \eta(s_1, x) \, \dx \\
        & \quad -  \int_{s_1}^{s_2} \int_{\Rd} \eta \Delta \eta \Delta \psi + \nabla \eta \cdot \nabla \psi\Delta \eta \, \dx \, \dt \\
        & \quad - \alpha \int_{s_1}^{s_2} \int_{\Rd} \eta \Delta \rho \Delta \psi + \nabla \eta \cdot \nabla \psi\Delta \rho \, \dx \, \dt \\
        & \quad - \frac{1}{2} \int_{s_1}^{s_2} \int_{\Rd} \eta^2 \Delta \psi \, \dx \, \dt + \omega \int_{s_1}^{s_2} \int_{\Rd} \eta \nabla \rho \cdot\nabla \psi \, \dx \, \dt .
    \end{align*}
    \end{enumerate}
\end{defn}

\begin{thm}\label{thm:main_result_existence two species} Let $(\rho_0,\eta_0)\in\mathcal{P}_2^a(\Rd)\times\mpdtard$ be an initial datum such that $\bm{\mf}[\rho_0,\eta_0]<+\infty$. Then there exists a weak solution to \eqref{eq:two species}.
\end{thm}
The last result is generalised to a wider class of systems allowing for nonlinear self-diffusion terms.

\subsection{Preliminaries}
We present the notation and we collect some \textit{a priori} results that we will use  throughout the manuscript. 

A key tool for the analysis is the Wasserstein metric, that is a distance function in the space of probability measures with finite second order moments.
\begin{defn}[$2$-Wasserstein distance]
    For $\mu$, $\nu \in \mathcal{P}_2 (\Rd)$, we define the $2$-Wasserstein distance,  $\mathcal{W}_2 (\mu, \nu )$, between $\mu$ and $\nu$ as
    \begin{equation*}
        \mathcal{W}_2 (\mu, \nu ) := \min_{\gamma\in\Gamma(\mu,\nu)} \left\lbrace \int_{\Rd \times \Rd} |x-y|^2 \, \mathrm{d} \gamma(x,y) \right\rbrace^{\frac{1}{2}},
    \end{equation*}
    where $\Gamma (\mu, \nu)$ is the set of transport plans between $\mu$ and $\nu$,
    \begin{equation*}
        \Gamma (\mu , \nu ) = \left\lbrace \gamma \in \mathcal{P} (\Rd \times \Rd ) \, : \, (\pi_x)_\# \gamma = \mu, \,(\pi_y)_\# \gamma = \nu \right\rbrace,
    \end{equation*}    
    and $\pi_x$ and $\pi_y$ are the projections onto the first and the second variables respectively.
\end{defn}

 In the expression above, marginals are the push-forward of $\gamma$ through $\pi_i$. For a measure $\rho\in\mP(\Rd)$ and a Borel map $T:\Rd\to\Rn$, $n\in\mathbb{N}$, the push-forward of $\rho$ through $T$ is defined by
$$
 \int_{\Rn}f(y)\,\mathrm{d}T_{\#}\rho(y)=\int_{\Rd}f(T(x))\,\mathrm{d}\rho(x) \qquad \mbox{for all $f$ Borel functions on}\ \Rd.
$$
We refer the reader to \cite{AGS,San15,Vil09} for further details on optimal transport theory and Wasserstein spaces. 

In order to obtain strong convergence of $\rho$ in $L^2 ([0,T];L^2(\Rd))$ we take advantage of a refined version of the Aubin-Lions Lemma for compactness in measures, due to Rossi and Savar\'e. It relies on two conditions: tightness and weak integral equi-continuity.

\begin{prop}[\cite{RS03}, Theorem 2]\label{prop:RS03}Let $X$ be a separable Banach space and consider
\begin{itemize}
    \item a lower semicontinuous functional $\mathcal{I}:X\rightarrow [0,+\infty]$ with relatively compact sublevels in $X$,
    \item a pseudo-distance $g:X\times X\rightarrow [0,+\infty]$,  i.e. $g$ is lower semicontinuous and such that $g(\rho,\eta) = 0$ for any $\rho,\eta\in X$ with $\mathcal{I}[\rho]<\infty$ and $\mathcal{I}[\eta]<\infty$ implies $\rho = \eta$.
\end{itemize}
Let $U$ be a set of measurable functions $u:(0,T)\rightarrow X$, with a fixed $T>0$. Assume further that $U$ is tight with respect to $\mathcal{I}$
\begin{equation}\label{tightness}
\sup_{u\in U}\int_0^T\mathcal{I}[u(t)]\,\mathrm{d}t<\infty\,,\end{equation}
and satisfies the weak integral equi-continuity condition \begin{equation}\label{integral equicontinuity}
\lim_{h\downarrow 0}\sup_{u\in U} \int_0^{T-h} g\left(u(t+h),u(t)\right)\mathrm{d}t = 0\,.
\end{equation}
Then $U$ contains an infinite sequence $(u_n)_{n\in\mathbb{N}}$ convergent in measure, with respect to $t\in(0,T)$, to a measurable $\Tilde{u}:(0,T)\to X$, i.e.
\begin{equation}\label{eq:convergence_measure_rs}
    \lim_{n\rightarrow\infty}\left|\{t\in(0,T):\;\Vert u_n(t) - \Tilde{u}(t)\Vert_X\geq\delta\}\right|=0,\quad\forall\delta\geq 0\,.
\end{equation}
\end{prop}

In addition to the strong convergence given by  \Cref{prop:RS03}, we will need an $L^2$ bound on $\Delta\rho$ to obtain suitable compactness in time and space for $\nabla\rho$ and $\Delta\rho$. We employ the \textit{flow interchange} technique, developed by Matthes, McCann and Savaré in~\cite{MMS09} and previously used in~\cite{Otto_CPDE98} --- we also refer the reader to~\cite{BE22,CdFEFS20,dFEF18,dFM14} for further details. The idea of the flow interchange consists in considering the dissipation of the free energy $\mf_m$ along a solution of an auxiliary gradient flow, and using the Evolution Variational Inequality (EVI) afterwards to obtain the desired estimate. For the reader's convenience we recall the definition of $\lambda$-flow for a general functional $\mathcal{G}$, which is connected to the EVI.

\begin{defn}[$\lambda$-flow]\label{def:lambda_flow}
    A semigroup $S_{\mathcal{G}} : [0, + \infty] \times \mathcal{P}_2 (\Rd ) \to \mathcal{P}_2 (\Rd )$ is a $\lambda$-flow for a functional $\mathcal{G}: \mathcal{P}_2 (\Rd ) \to \mathbb{R} \cup \{ \infty \}$ with respect to the distance $\mathcal{W}_2$ if, for an arbitrary $\rho \in \mathcal{P}_2 (\mathbb{R}^d)$, we have that the curve $t \mapsto S_{\mathcal{G}}^t \rho$ is absolutely continuous on $[0, \infty)$ and it satisfies the EVI
    \begin{equation}\label{eq:EVI}
        \tag{EVI}
        \frac{1}{2} \frac{\mathrm{d}^+}{\dt} \mathcal{W}_2^2 (S_{\mathcal{G}}^t \rho , \mu ) + \frac{\lambda}{2} \mathcal{W}_2^2 (S_{\mathcal{G}}^t \rho , \mu ) \leq \mathcal{G} [\mu] - \mathcal{G} [S_{\mathcal{G}}^t \rho],
    \end{equation}
    for all $t>0$, with respect to every reference measure $\mu \in \mathcal{P}_2 (\Rd )$ such that $\mathcal{G}[\mu] < \infty$.
\end{defn}

As shown in the seminal work by Jordan, Kinderlehrer and Otto \cite{JKO98}, the heat equation can be regarded as a $2$-Wasserstein steepest descent of the Boltzmann entropy
\begin{equation}\label{eq:Heat entropy}
    \mathcal{E} [\rho] = \begin{cases} \int_\Rd \rho(x) \log \rho (x) \,\dx, &\rho \log \rho \in L^1(\mathbb{R}^d),\\
+\infty & \mbox{otherwise}
\end{cases}
\end{equation}
We mention~\cite{AGS,San15} and the recent~\cite[Chapter 3.3]{FG21} for further details. The functional $\mathcal{E}$ is $0$-convex along geodesics and it possesses a unique $0$-flow, which we denote $S_{\mathcal{E}}$, given by the heat semigroup, see for example \cite{AGS, DS08, dFM14}. We will use the heat equation as the auxiliary flow and the free energy \eqref{eq:Heat entropy} as the auxiliary functional. 

In order to illustrate the method, let us calculate the dissipation of the Boltzmann entropy along solutions of our equation, \eqref{eq:thin_film_intro}. For simplicity, we consider $m = 2$, although the method generalises to other exponents. In this case, a formal computation yields
\begin{align*}
    \frac{\mathrm{d}}{\dt}\int_\Rd\rho\log\rho\,\dx &=- 
    \int_{\mathbb{R}^d} \log \rho\, \dive \left( \rho \nabla (\Delta \rho ) \right)\dx - 2\chi\int_{\mathbb{R}^d} \log \rho\, \dive \left( \rho \nabla \rho \right)\,\dx.
    \\
    & = \int_{\mathbb{R}^d} \nabla \rho\cdot  \nabla (\Delta \rho ) \,\dx + 2\chi\int_{\mathbb{R}^d} |\nabla\rho|^2 \,\dx. 
    \\
    & \leq - \int_\Rd (\Delta\rho)^2\,\dx + 2C,
\end{align*}
where the constant $C>0$ is given in \Cref{prop:Some properties}. By integrating in time, we obtain
\begin{align*}
    \| \Delta \rho \|_{L^2 ((0,T) \times \mathbb{R}^d)}^2 &\leq \me[\rho_0] - \me[\rho_T] + 2CT
    \\
    &\leq\Vert\rho_0\Vert^2_{L^2(\Rd)}- \me[\rho_T] + 2CT.
\end{align*}
It remains to notice that $\me[\rho]$ can be bounded from below by the second moment of $\rho$, $\mathrm{m}_2(\rho)$, which gives the desired $L^2$ bound on $\Delta\rho$. Although this formal computation requires further regularity, it illustrates how we may use an auxiliary flow to obtain $H^2$ estimates for our equation. In \Cref{H2 bound flow interchange}, we shall make this calculation fully rigorous by considering, instead, the dissipation of our energy functional $\mf_m[\rho]$, \eqref{eq:functional}, along solutions of the heat equation with suitable initial data.
\begin{rem}\label{rem:lower bound rho log rho}
    We remind the reader of the following bound for the Boltzmann entropy functional $\me$,
    \begin{equation*}
        \me[\rho] = \int_{\mathbb{R}^d} \rho \log \rho \geq - C (1 + \mathrm{m}_2(\rho)).
    \end{equation*}
  To prove this, let $M(x) := {(2 \pi )^{-d/2}} \exp \left( - {|x|^2}/{2} \right),$
and consider the relative entropy
    \begin{equation*}
        \mathcal{E} (\rho | M) := \int_{\mathbb{R}^d} \rho \log \frac{\rho}{M} \, \dx\,.
    \end{equation*}
Jensen's inequality implies that
    \begin{align*}
        \mathcal{E} (\rho | M) &  \geq \log \left( \int_{\Rd} \frac{\rho}{M}M \dx \right)\int_{\Rd} \frac{\rho}{M} \, \dx  = 0\,,
    \end{align*}
and thus, we conclude that
    \begin{align*}
        0 & \leq \me (\rho | M) = \int_{\Rd} \rho \log \rho\,\dx + \frac{d}{2}\log (2 \pi ) + \frac{1}{2} \int_{\Rd} |x|^2 \rho (x) \, \dx\,,
    \end{align*}
    which implies
    \begin{equation*}
        \int_{\Rd} \rho \log \rho \geq \frac{d}{2}\log (2 \pi ) - \frac{1}{2} \mathrm{m}_2 (\rho )\,. 
    \end{equation*}
\end{rem}

\section{Properties of the energy functional}\label{sec:Energy minimisers} 

The energy~$\mf_m$ plays a crucial role in the analysis of~\eqref{eq:thin_film_intro}, as it provides uniform bounds we hinge on for the construction of weak solutions. Furthermore, in the theory of gradient flows, the dynamical problem is usually related to energy minimisers via stationary states. This is, indeed, a valuable advantage of studying Eq.~\eqref{eq:thin_film_intro} in the Wasserstein space $(\mP_2(\Rd),\mw_2)$. As we shall see below, the Gagliardo--Nirenberg inequality is essential for a thorough study of our problem as it reveals critical regimes. For the reader's convenience we recall it in the lemma below, cf.~for instance~\cite{ Brezis_Mironescu_GN_2018, Nirenberg_1959}.

\begin{lem}[Gagliardo--Nirenberg interpolation inequality]
Let $\theta\in[0,1]$, $1\le p,q\le +\infty$, and $1\le r<\infty$ such that $\frac{1}{p}=\theta\left(\frac{1}{r}-\frac{1}{d}\right)+\frac{1-\theta}{q}$. Then, it holds
\[
\|f\|_{L^p(\Rd)}\le C\|\nabla f\|_{L^r(\Rd)}^\theta\|f\|_{L^q(\Rd)}^{1-\theta},
\]
where $C$ denotes a positive constant depending on $p,q,r,\theta$, but not on $f$.
\end{lem}
In the following we set
\[
2^{\ast}:=\begin{cases}
    +\infty &\mbox{if } d=1,2,\\
    \frac{2d}{d-2} & \mbox{if } d\ge3.
\end{cases}
\]

In the proposition below we find a range of exponents for which the free energy $\mf_m$ is bounded from below, thus proving~Theorem~\ref{thm:energy}, case~\ref{thm:subcritical_exp}. In turn, this implies further regularity for the density $\rho$ and provides the critical exponent, $m_c=2+2/d$. 

\begin{prop}[Lower bound for the free energy and induced regularity]\label{prop:Some properties}
     Assume $\rho \in L^1_+(\Rd)$ and let $1\le  m<2+\frac{2}{d}$. Set $\alpha:= 1+\frac{\frac{2}{d}(m-1)}{2+ \frac{2}{d}-m}$, for $m>1$, and $\alpha := 2$, for $m =1$. The following properties hold.
    \begin{enumerate}[label=$(\arabic*)$]
        \item \underline{Lower bound for the free energy}: let $\nabla \rho \in L^2 (\mathbb{R}^d)$, then $\mf_m[\rho]$ is bounded from below as
        \begin{equation}\label{eq:Free energy bounded below}
            \mf_m[\rho] \geq - C \| \rho \|_{L^1(\mathbb{R}^d)}^{\alpha},
        \end{equation}
        where $C = C(m,d,\chi)$.\\
        \item \underline{$H^1$-bound}: assume $\mf_m[\rho] < +\infty$, then the following bound holds
        \begin{equation}\label{eq:grad rho is L2}
            \| \nabla \rho \|^2_{L^2 (\Rd)} \leq C \left( \mf_m[\rho] + \| \rho \|_{L^1(\mathbb{R}^d)}^{\alpha} \right),
        \end{equation}
        where $C=C(m,d,\chi)$.  \\
        \item \underline{$L^p$-regularity}: assume $\mf_m[\rho] < +\infty$, then $\rho \in L^{p} (\Rd)$ for $p\in[1,2^{\ast}]$, where $2^{\ast} := +\infty$ for $d=1,2$,  and $2^{\ast} := \frac{2d}{d-2}$ for $d \geq 3$. Note that $1 \leq m < 2+\frac{2}{d}<2^{\ast}$. In particular,  there exists a constant $C = C(m,p,d,\rho,\chi)>0$ such that
         \begin{equation}\label{eq:control of the Lp norm}
            \| \rho \|_{L^p (\mathbb{R}^d)} \leq C<+\infty .
        \end{equation}
    \end{enumerate}
\end{prop}

\begin{proof}
 \textbf{Step 1:} Lower bound for the free energy. Let $1<m<2+\frac{2}{d}$. By applying Gagliardo--Nirenberg inequality to $\| \rho \|_{L^m(\mathbb{R}^d)}$ we find
    \begin{equation*}
        \| \rho \|_{L^m (\mathbb{R}^d)} \leq C \| \nabla \rho \|_{L^2 (\mathbb{R}^d)}^{\theta} \| \rho \|_{L^1(\mathbb{R}^d)}^{1- \theta},
    \end{equation*}
    where $\theta =\frac{2d}{d+2}\frac{m-1}{m}\in(0,1)$.
    By applying Young's inequality with $p= \frac{2}{m \theta} = \frac{1+\frac{2}{d}}{m-1}>1$ and $p'$ its conjugate, we have 
    \begin{equation*}
        \| \rho \|_{L^m (\mathbb{R}^d)}^m \leq   \frac{\varepsilon^p \| \nabla \rho\|_{L^2(\mathbb{R}^d)}^{2}}{p} + \frac{C^{m p'}}{\varepsilon^{p'}} \frac{\|\rho \|^{m (1- \theta) p'}_{L^1(\mathbb{R}^d)}}{p'} .
    \end{equation*}
    Therefore, taking any $0<\varepsilon <\left({p(m-1)}/2\right)^{1/p}$ we obtain the bound 
    \begin{align}
        \mf_m[\rho] & = \frac{1}{2}\| \nabla \rho\|_{L^2(\mathbb{R}^d)}^{2} -\frac{\chi}{m-1}\| \rho \|_{L^m (\mathbb{R}^d)}^m \nonumber\\
        & \geq \left(\frac{1}{2} -\frac{\varepsilon^p}{p(m-1)} \right)\| \nabla \rho\|_{L^2(\mathbb{R}^d)}^{2} - \frac{\chi C^{m p'}}{p'(m-1)\varepsilon^{p'}} \|\rho \|^{m (1- \theta) p'}_{L^1(\mathbb{R}^d)} \label{eq:Lm Young 2}\\
        & \geq - C \| \rho \|_{L^1(\mathbb{R}^d)}^{\alpha},\nonumber
    \end{align}
    where $C=C(m,d,\chi)$, and $\alpha = m(1- \theta) p' = 1+\frac{\frac{2}{d}(m-1)}{2+ \frac{2}{d}-m}$.

    Note that in case of linear diffusion $m=1$, i.e. $\mf_1[\rho]$ as functional, we can argue similarly by using that
    \[
      \mf_1[\rho]\ge\mf_2[\rho]\ge -C \| \rho \|_{L^1(\mathbb{R}^d)}^{2}. 
    \]
    Note that the first inequality holds because $\me_1[\rho]\le\me_2[\rho]$, since $x\log x\le x^2$, for $x>0$.

\noindent
 \textbf{Step 2:} $H^1$-bound. The result follows from \eqref{eq:Lm Young 2} by noting that $\mf_m[\rho]<+\infty$ and choosing again $0<\varepsilon <\left({p(m-1)}/2\right)^{1/p}$ .

\noindent
    \textbf{Step 3:} $L^p$-regularity. From the previous case, we have $\nabla\rho\in L^2(\Rd)$, and thus we can apply Gagliardo--Nirenberg inequality to obtain 
    \begin{equation*}
        \| \rho \|_{L^p (\mathbb{R}^d)} \leq C \| \nabla \rho \|_{L^2 (\mathbb{R}^d)}^{\theta} \| \rho \|_{L^1(\mathbb{R}^d)}^{1- \theta}\leq C<\infty,
    \end{equation*}
    with $\theta = \frac{2d}{d+2}\frac{p-1}{p}\in[0,1]$ for $p\in[1,2^{\ast}]$. Note that for $d\geq 3$ and $p = 2^{\ast}$ we have $\theta = 1$. \qedhere
\end{proof}

In the critical exponent case, $m_c = 2 + \frac{2}{d}$, deriving energy bounds and induced regularity as in~Proposition~\ref{prop:Some properties} reveals the critical mass 
\begin{equation}
       \chi_c := \frac{m_c-1}{2C_{GN}},\label{eq: critical mass}
\end{equation}
where $C_{GN}$ stands for the sharp constant from the Gagliardo--Nirenberg inequality, for $m = m_c$ given by
\begin{equation}\label{eq:crit eq 2}
        \|\rho\|_{L^{m_c}(\mathbb{R}^d)}^{m_c} \leq C_{GN} \| \nabla \rho \|_{L^2(\mathbb{R}^d)}^2 \| \rho \|_{L^1(\mathbb{R}^d)}^{2/d}
\end{equation}

\begin{rem}\label{rem: critical mass}(Critical mass and the parameter $\chi$).  The critical mass in~\eqref{eq: critical mass} is obtained for the sharp Gagliardo--Nirenberg constant, $C_{GN}$. This value is found in \cite{liu2017generalized}, extending to general dimension~\cite{Witel_Bern_Bert_EJAM04}. Note that we refer to $\chi_c$ as the critical mass since we assume that all densities are probability measures with unit mass. However, upon rescaling \eqref{eq:thin_film_intro} using the change of variables 
\[
t\mapsto t/\chi^{\frac{1}{m-2}} \quad\mbox{ and }\quad \rho\mapsto\rho \chi^{\frac{1}{m-2}},
\]
Eq.~\eqref{eq:thin_film_intro} becomes $\partial_t\rho=-\mathrm{div}(\rho\nabla(\Delta\rho))-\Delta\rho^m$. Therefore, one can distinguish between subcritical, critical, and supercritical regimes, in terms of the usual mass $\Vert \rho\Vert_{L^1(\Rd)}$.
\end{rem}
We show that for $\chi\leq \chi_c$ and $m=m_c$ the free energy is bounded from below, which covers~Theorem~\ref{thm:energy}, case~\ref{thm:critical_exp}.
\begin{prop}\label{prop:prop_energies_critical}
    Let $m= m_c$, $\chi\leq \chi_c$, and assume $\rho \in \mathcal{P}^a(\Rd)$, $\nabla \rho \in L^2 (\mathbb{R}^d)$. The free energy \eqref{eq:functional} satisfies the bound
        \begin{equation*}
            \mf_{m_c}[\rho] \geq \| \nabla \rho \|_{L^2(\Rd)}^2 \left( \frac{1}{2} - \frac{{\chi}C_{GN}}{m_c - 1} \| \rho \|_{L^1(\Rd)}^{\frac{2}{d}} \right) \geq 0.
        \end{equation*}
Moreover, if $\chi<\chi_c$ and $\mf_{m_c}[\rho] < +\infty$, then
\begin{equation*}
    \Vert\rho\Vert_{L^p(\Rd)},\Vert\nabla\rho\Vert_{L^2(\Rd)}< C,
\end{equation*} where $C = C(m,d,\chi)$ and for $p\in[1,2^{\ast}]$. Furthermore, for $\chi=\chi_c$, the optimisers of the Gagliardo--Nirenberg inequality \eqref{eq:crit eq 2} are the set of global minimisers of the free energy.
\end{prop}

\begin{proof}
    From the Gagliardo--Nirenberg inequality \eqref{eq:crit eq 2}, we can deduce that
    \begin{align*}
        \mf_{m_c}[\rho] & = \frac{1}{2} \| \nabla \rho \|_{L^2(\Rd)}^2 - \frac{\chi}{m_c - 1} \| \rho \|_{L^{m_c}(\Rd)}^{m_c} \\
        & \geq \| \nabla \rho \|_{L^2(\Rd)}^2 \left( \frac{1}{2} - \frac{\chi C_{GN}}{m_c - 1} \| \rho \|_{L^1(\Rd)}^{\frac{2}{d}} \right).
    \end{align*}
    In particular, since $\chi \leq \chi_c$ and $\| \rho \|_{L^1(\Rd)} = 1$, we obtain
    \begin{equation*}
        \mf_{m_c} [\rho] \geq 0.
    \end{equation*}
The last properties are simple consequences of the Gagliardo--Nirenberg inequality \eqref{eq:crit eq 2}, the definition of the free energy and $\chi<\chi_c$.
\end{proof}

Summarising the previous two propositions, using the Gagliardo--Nirenberg inequality we showed that the free energy $\mf_{m_c}[\rho]$ is uniformly bounded from below when the exponent $m$ is subcritical, $1 \leq m < m_c$, or when $m=m_c$ and we have subcritical or critical mass, $\chi \leq \chi_c$. Moreover, this induces further regularity in the subcritical exponent and critical exponent with subcritical mass cases. In section~\ref{sec:existence_one_species}, we use this information to prove existence of weak solutions to~\eqref{eq:thin_film_intro}.

In order to gain further intuitions on the remaining cases, $m=m_c$ with $\chi\ge\chi_c$ and $m>m_c$,  we study energy minimisers distinguishing between the two cases. 

\subsection{Critical exponent case}
First, we focus on the critical case given by $m = m_c = 2 + \frac{2}{d}$, and study properties of the free energy  \eqref{eq:functional}, following ideas from~\cite{blanchet2009critical}. This highlights an interesting connection with Patlak-Keller-Segel sytems \cite{Carrillo_Craig_Yao19}, and more broadly with aggregation-diffusion equations, as mentioned in the introduction.

A crucial observation concerns the homogeneity of the energy funcional $\mf_{m_c}$: mass-preservation dilation implies that, in this critical case, the aggregation and diffusion terms in the energy functional \eqref{eq:functional} have the same homogeneity.

\begin{lem}[Scaling properties of the free energy]\label{lem: scaling energy} 
Assume $\rho\in L^{m_c}(\Rd)$ such that $\nabla\rho\in L^2(\Rd)$. Let $\rho_\lambda(x) := \lambda^d\rho(\lambda x)$, for any $x\in\Rd$, then
\begin{equation*}
    \Vert\rho_\lambda\Vert_{L^{m_c}(\Rd)}^{m_c} = \lambda^{d+2} \Vert\rho\Vert_{L^{m_c}(\Rd)}^{m_c},\quad  \Vert\nabla\rho_\lambda\Vert_{L^{2}(\Rd)}^{2} = \lambda^{d+2} \Vert\nabla\rho\Vert_{L^{2}(\Rd)}^2,
\end{equation*}
for all $\lambda\in(0,+\infty)$. In particular,
    \begin{equation*}
        \mf_{m_c}[\rho_\lambda] = \lambda^{d+2}\mf_{m_c}[\rho].
    \end{equation*}
\end{lem}
\begin{proof}
    We have 
    \begin{align*}
        \mf_{m_c}[\rho_\lambda] &= \frac{\lambda^{2d}}{2}\int_\Rd|\nabla\rho(\lambda x)|^2\,\dx-\frac{\chi\lambda^{ dm_c}}{m_c-1}\int_\Rd \rho^{m_c}(\lambda x)\,\dx 
        \\
       & = \frac{\lambda^{d+2}}{2}\int_\Rd|\nabla\rho( x)|^2\,\dx-\frac{\chi \lambda^{d(m_c-1)}}{m_c-1}\int_\Rd \rho^{m_c}( x)\,\dx 
        \\ & = \lambda^{d+2}\mf_{m_c}[\rho]\,,
    \end{align*}
    since $d (m_c - 1) = d+2$.
\end{proof}

Next, we study the infimum of the free energy $\mf_{m_c}$. Let us define $\mu_\chi := \inf_{\rho\in\mathcal{Y}}\mf_{m_c}[\rho]$, where
\begin{equation*}
    \mathcal{Y} = \{\rho\in \mathcal{P}^a(\Rd)\cap L^{m_c}(\Rd):\nabla\rho\in L^2(\Rd)\}.
\end{equation*}
The next result completes~Theorem~\ref{thm:energy}, case~\ref{thm:critical_exp}.
\begin{prop}[Infimum of the free energy] \label{prop: infimum free energy}
We have
\begin{equation*}
    \mu_\chi = \begin{cases}
    0 & \mbox{if } \chi\in(0,\chi_c]\,,
    \\
    -\infty & \mbox{if } \chi>\chi_c\,.
    \end{cases}
\end{equation*}
    Moreover, for $\rho\in\mathcal{Y}$,
    \begin{equation}
    \label{eq: inequality infimum}
       \left(\chi_c-\chi\right)\Vert\nabla\rho\Vert_{L^2(\Rd)}\leq\frac{(m_c-1)\mf_{m_c}[\rho]}{C_{GN}}\leq\left(\chi_c+\chi\right)\Vert\nabla\rho\Vert_{L^2(\Rd)},
    \end{equation}
    where the critical mass $\chi_c$ is defined in \eqref{eq: critical mass} and $C_{GN}$ is the sharp constant in the Gagliardo--Nirenberg inequality \eqref{eq:crit eq 2}. In particular, the infimum is not achieved for $\chi<\chi_c$, and there exists a minimiser in $\mathcal{Y}$ for $\chi = \chi_ c$.
\end{prop}
\begin{proof}
    Let $\rho\in\mathcal{Y}$. By Gagliardo-Nirenberg inequality \eqref{eq:crit eq 2},
     \begin{align*}
        \mf_{m_c}[\rho] & \geq \| \nabla \rho \|_{L^2(\Rd)}^2 \left( \frac{1}{2} - \frac{\chi C_{GN}}{m_c - 1} \| \rho \|_{L^1(\Rd)}^{{2}/{d}} \right)
        \\
        &  = \| \nabla \rho \|^2_{L^2(\Rd)} \left(\chi_c-\chi\right)\frac{C_{GN}}{m_c-1},
    \end{align*}
    and also
    \begin{align*}
        \mf_{m_c}[\rho] & \leq \frac{1}{2}\Vert\nabla\rho\Vert^2_{L^2(\Rd)} + \frac{\chi}{m_c-1}\Vert\rho\Vert^{m_c}_{L^{m_c}(\Rd)}
        \\
       & \leq \| \nabla \rho \|^2_{L^2(\Rd)} \left( \chi_c+\chi\right)\frac{C_{GN}}{m_c-1},
    \end{align*}
    which gives \eqref{eq: inequality infimum}.

\emph{Case I:} $\chi\leq \chi_c$. We first show $\mu_\chi = 0$. From \eqref{eq: inequality infimum} we see that $\mu_\chi\geq 0$. Let $u_\varepsilon(x) = \varepsilon^d u(\varepsilon x)$, where $u\in\mathcal{Y}$.
Then, $u_\varepsilon\in\mathcal{Y}$ and by \Cref{lem: scaling energy}, we have $\Vert\nabla u_\varepsilon\Vert_{L^2(\Rd)} = O(\varepsilon^{d/2+1})$. Hence, by sending $\varepsilon\downarrow 0$ in \eqref{eq: inequality infimum} we obtain $\mu_\chi =0$.

Next note that if $\chi<\chi_c$ the inequality in \eqref{eq: inequality infimum} is strict and the infimum cannot be achieved. When the mass is critical, $\chi=\chi_c$, we exploit~\cite{liu2017generalized}, where equality in the Gagliardo--Nirenberg inequality is proven for a non-negative radial symmetric function that can be chosen in $\mathcal{Y}$. In particular, we have a minimiser for $\mf_{m_c}$. Moreover, all minimisers coincide with scalings of this fixed profile, that is, the set of global minimisers is given by the optimisers of  
the Gagliardo--Nirenberg inequality \eqref{eq:crit eq 2}.

\emph{Case II:} $\chi>\chi_c$. The arguments presented here are inspired by~\cite{weinstein1982nonlinear}. Fix $\delta \in\left(\chi_c/\chi,1\right) $. Due to the Gagliardo--Nirenberg inequality, there exists a nonzero function $\rho^*\in L^{m_c}(\Rd)$ with $\nabla\rho^*\in L^2(\Rd)$ such that
\begin{equation}
\label{eq: supercritical mass inequality}
     C_{GN}\delta\leq\frac{\Vert\rho^*\Vert^{m_c}_{L^{m_c}(\Rd)}}{\Vert\nabla\rho^*\Vert^2_{L^2(\Rd)}\Vert\rho^*\Vert^{2/d}_{L^1(\Rd)}}\leq C_{GN}\,.
\end{equation}
Now let $\lambda>0$, and consider the function $\rho_\lambda(x) = \lambda^d\rho^*\left(\lambda \Vert\rho^*\Vert_{L^1(\Rd)}^{1/d}x\right)$. It is easy to check $\rho_\lambda\in\mathcal{Y}$. From~\Cref{lem: scaling energy}, \eqref{eq: supercritical mass inequality}, and the definition of the critical mass \eqref{eq: critical mass}, we obtain 
\begin{align*}
    \mf_{m_c}[\rho_\lambda] & = \frac{\lambda^{d+2}}{\|\rho^*\|_{L^1(\Rd)}}\left[\frac{\Vert\nabla\rho^*\Vert_{L^2(\Rd)}^2\|\rho^*\|_{L^1(\Rd)}^{2/d}}{2}-\frac{\chi\Vert\rho^*\Vert_{L^{m_c}(\Rd)}^{m_c}}{m_c-1}\right]
    \\
    & = \frac{\lambda^{d+2}}{2}\|\rho^*\|_{L^1(\Rd)}^{2/d-1}\Vert\nabla\rho^*\Vert_{L^2(\Rd)}^2\left[1-\frac{2\chi}{m_c-1}\frac{\Vert\rho^*\Vert_{L^{m_c}(\Rd)}^{m_c}}{\Vert\rho^*\Vert_{L^{1}(\Rd)}^{2/d}\Vert\nabla\rho^*\Vert_{L^{2}(\Rd)}^{2}}\right]
    \\
    & \leq \frac{\lambda^{d+2}}{2}\|\rho^*\|_{L^1(\Rd)}^{2/d-1}\Vert\nabla\rho^*\Vert_{L^2(\Rd)}^2\left(1 - \frac{\chi}{\chi_c}\delta\right)\,.
\end{align*}
Owing to the choice of $\delta$ and taking the limit $\lambda\rightarrow+\infty$ we obtain $\mu_\chi = -\infty$.

\end{proof}

\subsubsection{Self-similarity}\label{sec:self_similarity}
In the critical case $m = m_c= 2 + \frac{2}{d}$ we may assume the self-similar ansatz
 \begin{equation}
     \label{eq:self_similar_ansatz}
     \rho(x,t) = t^{-a}u(xt^{-b})\,.
 \end{equation}
 Mass conservation gives the usual relation between the exponents $a = b d$. Moreover, assuming \eqref{eq:self_similar_ansatz} is a solution of \eqref{eq:thin_film_intro}, we obtain 
 \begin{equation*}
     a u +b\nabla u(z)\cdot z = \dive \left(u\nabla(\Delta u(z))\right) + \chi\Delta u(z)^{m_c}\,
 \end{equation*}
 with
 \begin{equation*}
     b =\frac{1}{d+4},\quad a= b d.
 \end{equation*}
In particular, we obtain the equation 
 \[
\dive \left(u\nabla(\Delta u(z))\right) + \chi\Delta u(z)^{m_c}-b \, \dive(zu)=0,
 \]
which is the equation for steady states of the corresponding evolution problem
\begin{equation}\label{eq:4ord-fokker-planck}
\partial_t u=- \dive \left(u\nabla(\Delta u(z))\right) -\chi \Delta u(z)^{m_c}+b \, \dive(zu).
\end{equation}
The above evolution PDE is (at least formally) a $2$-Wasserstein gradient flow of the energy
\[
\mathcal{L}[u]=\mathcal{F}_{m_c}[u]+\frac{b}{2}\int|z|^2u(z)\,\mathrm{d}z.
\]
For this energy, one can prove existence of minimisers using the direct method of Calculus of Variations. The main advantage with respect to the minimisation of $\mf_{m}$ is the presence of the additional term, fundamental for the compactness of the minimising sequence, as we shall see also in~Proposition~\ref{prop:Existence minimiser}. As the proof of the latter proposition applies to a wider range of exponents, including $m=m_c$, we postpone this proof to~Section~\ref{sec:existence_one_species}, below that of~Proposition~\ref{prop:Existence minimiser}. 

\begin{prop}[Existence of minimisers for $\mathcal{L}$]\label{prop:minimisers_fokker_planck}
Given $\chi<\chi_c$,  the functional $\mathcal{L}:\mP^a(\Rd)\to[-\infty,+\infty]$ admits minimisers in the set $\{u\in \mathcal{P}^a(\Rd): \nabla u\in L^2(\Rd), \, \mathrm{m}_2(u)<\infty\}$.
\end{prop}

A natural question to ask is whether one can characterise energy minimisers, in the spirit of \cite{blanchet2009critical,carrillo2015ground,Calvez_Carrillo_Hoffmann17b}, and check if these are steady states of Eq.~\eqref{eq:4ord-fokker-planck}. In turn, one would be able to characterise self-similar profiles for~\eqref{eq:thin_film_intro}. 

As mentioned in Remark~\ref{rem:addition_external_potential}, Eq.~\eqref{eq:4ord-fokker-planck} admits weak solutions arguing as in Section~\ref{sec:existence_one_species}. Studying the long-time behaviour of solutions to~\eqref{eq:4ord-fokker-planck} is also another interesting open problem we leave to future investigation, as well as a thorough study of energy minimisers for the subcritical case $m<m_c$.

\subsection{Supercritical exponent case}
We study the infimum of the free energy $\mf_m$ when $m > m_c$ , i.e. it is supercritical. As before, we define the set
\begin{equation*}
    \mathcal{Y} = \{\rho\in \mathcal{P}^a(\Rd)\cap L^{m}(\Rd):\nabla\rho\in L^2(\Rd)\},
\end{equation*}
and we prove that the free energy is not bounded from below. This is, indeed, Theorem~\ref{thm:energy}, case~\ref{thm:supercritical_exp}.
\begin{prop}[Infimum of the free energy]\label{prop: free energy supercritical}
   Assume $m>m_c$. Then 
    \begin{equation*}
        \inf_{\rho \in \mathcal{Y}} \mf_m [\rho] = - \infty.
    \end{equation*}
\end{prop}

\begin{proof}
    Given $\rho \in \mathcal{Y}$, we define $\rho_{\lambda}(x): = \lambda^d \rho (\lambda x )$, for any $x\in\Rd$. Note that $\rho_\lambda\in\mathcal{Y}$. Then, we have
    \begin{align*}
        \mf_m [\rho_{\lambda}] & = \frac{\lambda^{d+2}}{2} \| \nabla \rho \|_{L^2 (\Rd )}^2 - \frac{\chi\lambda^{d(m-1)}}{m-1} \| \rho \|_{L^m(\Rd)}^m
        \\
& = \lambda^{d+2} \left[ \mf_m [ \rho] - \frac{\chi \lambda^{d(m-m_c)} - 1 }{m-1} \| \rho \|_{L^m(\Rd)}^m \right].
    \end{align*}
Let us note that for any $\lambda$ big enough 
    \begin{equation*}
        \mf_m [ \rho] - \frac{ \chi\lambda^{d(m-m_c)} - 1 }{m-1} \| \rho \|_{L^m(\Rd)}^m < 0\,.
    \end{equation*}
    Therefore, by letting $\lambda \rightarrow +\infty$, $\mf_m [ \rho_{\lambda}] \rightarrow - \infty$, obtaining the desired result.    
\end{proof}

Finally, we briefly discuss on finite-time blow-up of classical solutions for the supercritical regimes. This shows that our main global in time existence results in Theorem \ref{thm:main_result_existence} for \eqref{eq:thin_film_intro} are sharp. Our arguments are based on the computation for the evolution of the second-order moment $m_2(\rho)$ as classically done in Keller-Segel models \cite{DP04,BDP06,blanchet2009critical,Calvez_Carrillo_Hoffmann17b,liu2017generalized}. We assume the solutions are classical solutions such that the following computations using integration by parts are allowed. More precisely, one can find that
\begin{align}
    \frac{\mathrm{d}}{\dt}\mathrm{m}_2(\rho) 
    & = 2\int_\Rd x \cdot\left(\rho \nabla(\Delta \rho))+\chi\nabla \rho^{m}\right)\,\dx \nonumber
    \\ 
    & = -2d\int_\Rd\rho\Delta\rho\,\dx - 2\int_\Rd(x\cdot\nabla\rho)\Delta\rho\,\dx - 2d\chi\int_\Rd\rho^m\,\dx \nonumber
    \\
    & = (d+2)\int_{\Rd}|\nabla\rho|^2\,\dx -2d\chi\int_\Rd\rho^m\,\dx \nonumber
    \\
    & = 2(d+2)\left[\mf_m[\rho] - \chi\left(\frac{1}{m_c-1}-\frac{1}{m-1}\right)\Vert\rho\Vert^m_{L^m(\Rd)}\right], \label{eq: evolution second moment}
\end{align}
where we used that
\begin{align*}
    \int_\Rd(x\cdot\nabla\rho)\Delta\rho\,\dx &= -\int_\Rd \nabla(x\cdot\nabla\rho)\cdot\nabla\rho\,\dx 
    \\ 
    &= -\int_\Rd|\nabla\rho|^2\,\dx - \int_\Rd x\cdot D^2\rho\nabla\rho\,\dx
    \\
    & = -\int_\Rd|\nabla\rho|^2\,\dx - \frac{1}{2}\int_\Rd x\cdot\nabla|\nabla\rho|^2\,\dx
    \\
    & = \left(\frac{d}{2}-1\right)\int_\Rd|\nabla\rho|^2\,\dx\,.
\end{align*}
We observe that this computation could be made rigorous for the solutions constructed in Theorem \ref{thm:main_result_existence} by using the flow interchange technique with a suitable auxiliary flow \cite{MMS09}, in the same spirit as in Proposition \ref{H2 bound flow interchange}. A short-time existence of solutions in the super critical exponent is expected as in \cite{CS18} but it is not a trivial question for the variational scheme below.

Note that for the critical case $m = m_c$, \eqref{eq: evolution second moment} reduces to
\begin{align*}
      \frac{\mathrm{d}}{\dt}\mathrm{m}_2(\rho) =  2(d+2)\mf_{m_c}[\rho]\,.
\end{align*}
In particular, by using Proposition \ref{prop: infimum free energy} we obtain that the second moment is non-decreasing in time for the subcritical and critical mass regimes, $\chi\le\chi_c$. In the supercritical mass regime, by using the above equation and that free energy $\mf_{m_c}[\rho]$ is unbounded from below, see Proposition \ref{prop: infimum free energy}, the authors of \cite{liu2017generalized} are able to show that any solution to \eqref{eq:thin_film_intro} with an initial datum $\rho_0$ satisfying $\mf[\rho_0]<0$, has a finite-time blow-up in the $L^{m_c}$-norm.

\begin{figure}[h!]
    \centering
\begin{subfigure}{\textwidth}
\centering\includegraphics[width=0.85\textwidth]{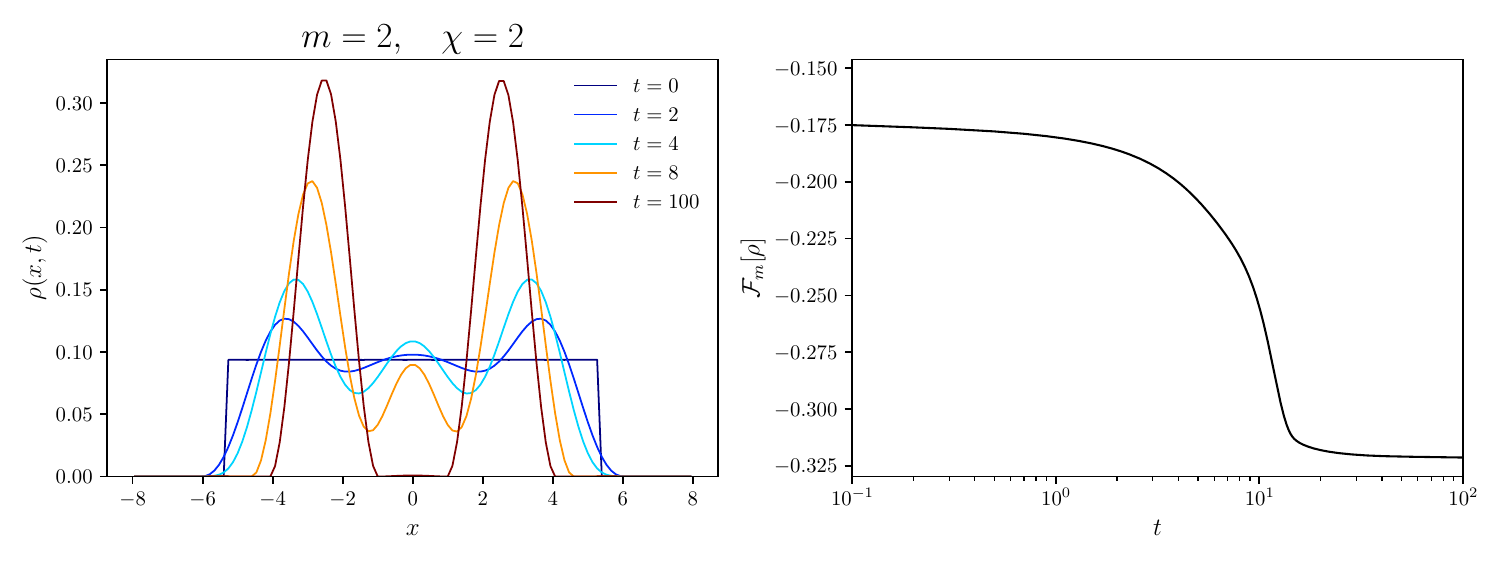}
    \caption{Subcritical exponent $m<m_c$.}
    \label{fig:first}
\end{subfigure}
\begin{subfigure}{\textwidth}
    \centering
    \includegraphics[width=0.85\textwidth]{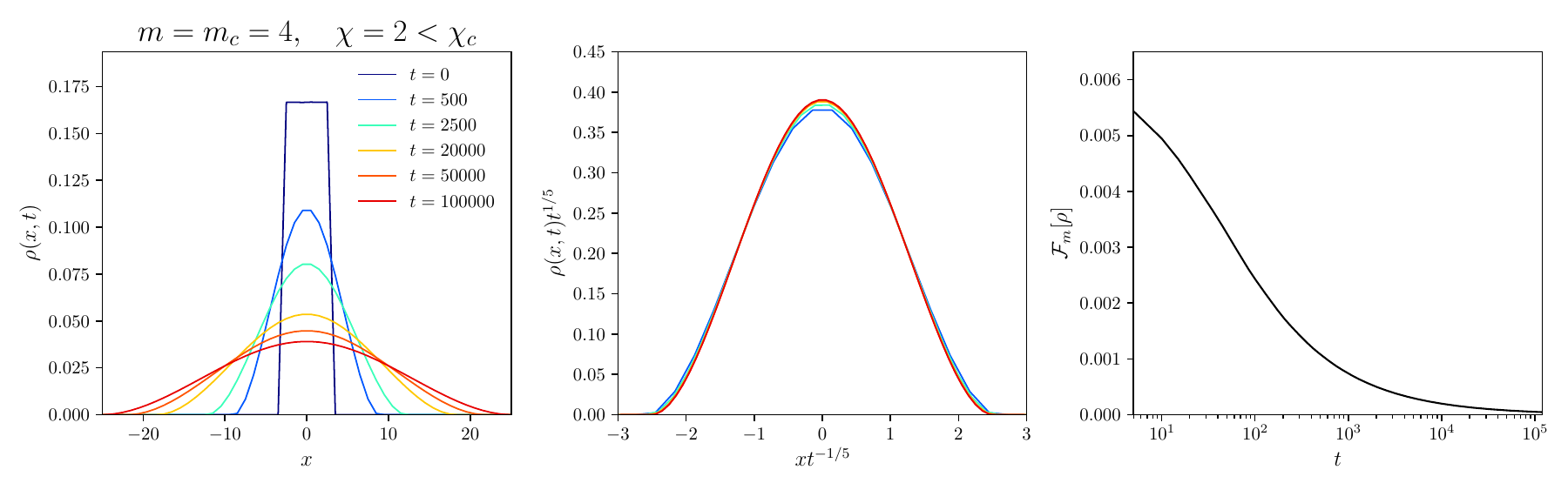}
    \caption{Critical exponent $m = m_c$, subcritical mass $\chi <\chi_c$.}
    \label{fig:second}
\end{subfigure}
\begin{subfigure}{\textwidth}
    \centering
\includegraphics[width=0.85\textwidth]{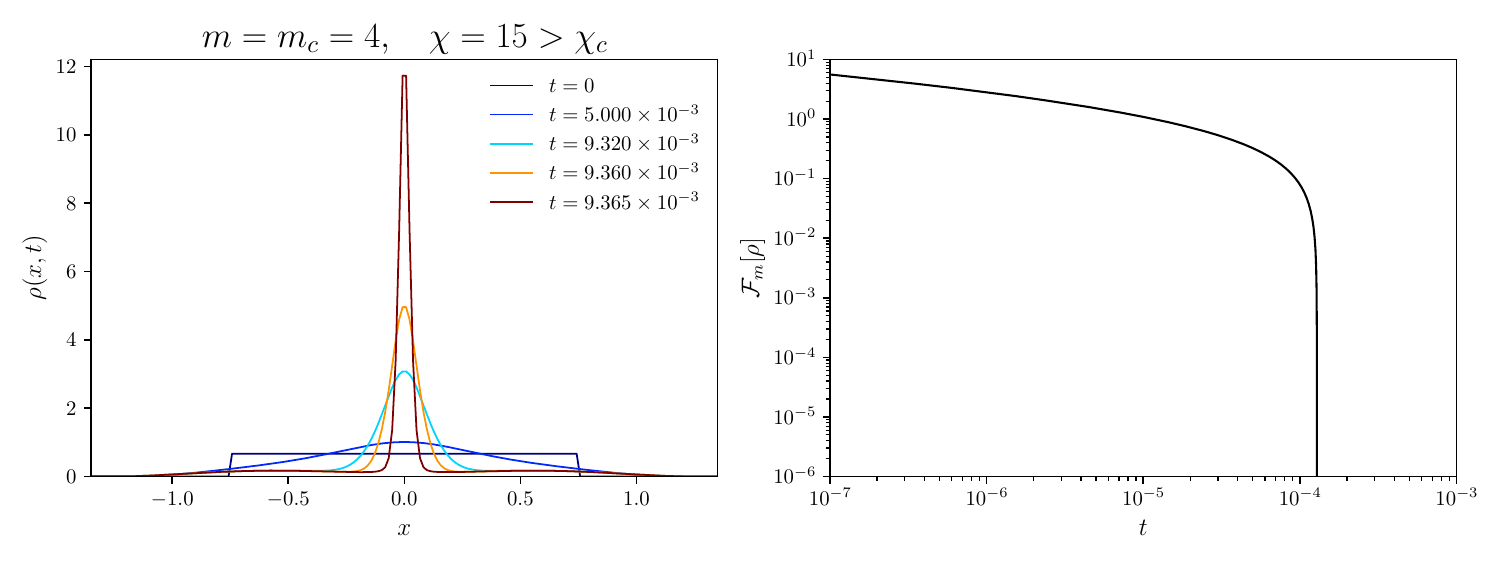}
    \caption{Critical exponent $m=m_c$, supercritical mass $\chi > \chi_c$.}
    \label{fig:third}
\end{subfigure}
\begin{subfigure}{\textwidth}
    \centering
\includegraphics[width=0.85\textwidth]{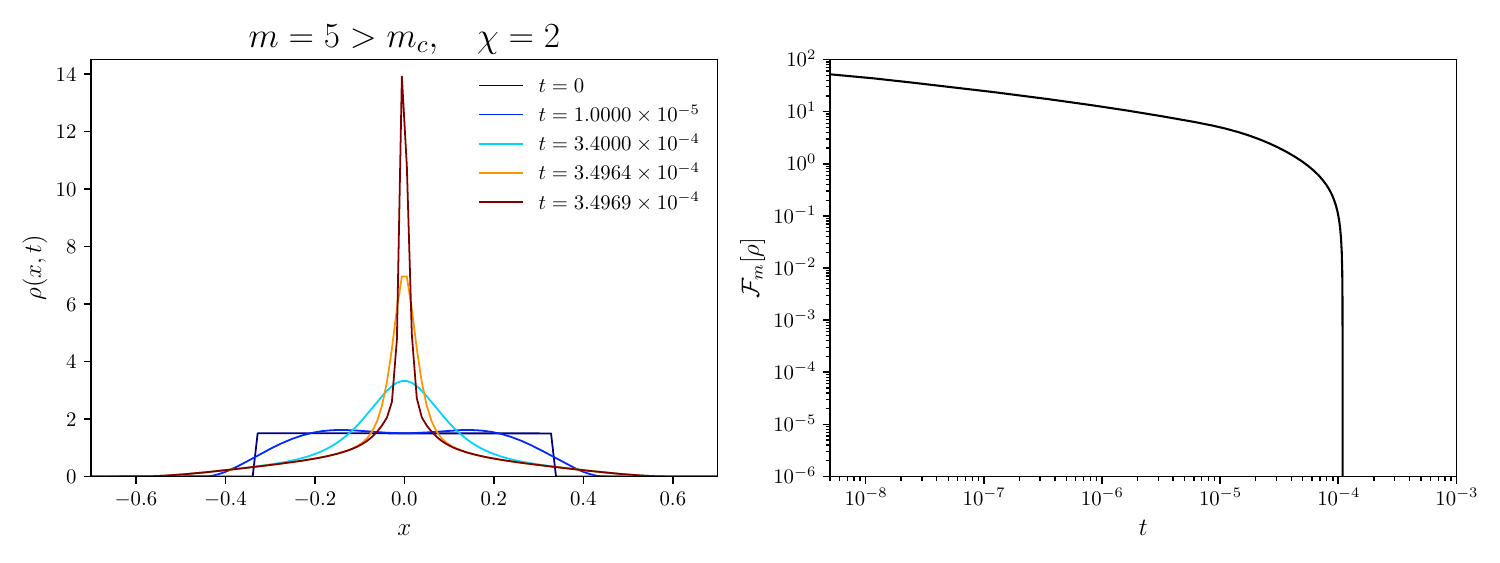}
    \caption{Supercritical exponent $m > m_c$.}
    \label{fig:fourth}
\end{subfigure}
\caption{Numerical solutions to \eqref{eq:thin_film_intro} in one spatial dimension for different values of $m$ and $\chi$, and decay of the free energy $\mf_m[\rho]$ as a function of time. }
\label{fig:figures}
\end{figure}

A similar argument also works in the supercritical exponent case. If $m>m_c$ then 
\begin{align*}
        \frac{\mathrm{d}}{\dt}\mathrm{m}_2(\rho) \leq  2(d+2)\mf_{m}[\rho]\leq 2(d+2)\mf_{m}[\rho_0]<0\,,
\end{align*}
for an initial datum with $\mf_m[\rho_0]<0$, which can be chosen by Proposition \ref{prop: free energy supercritical}. If the initial second moment is finite, then there exists some time $t^*>0$ such that $\mathrm{m}_2(\rho(t^*)) = 0$, implying that such solutions can only exist locally in time.

Our main results are also summarised in Figure~\ref{fig:figures}, where we plot numerical solutions to~\eqref{eq:thin_film_intro} in one spatial dimension and for different values of the exponent $m$ and the mass parameter $\chi$. These are based on the finite-volume scheme presented in \cite{bailo2021unconditional}. In particular, we observe that for subcritical exponent, solutions evolve towards a compactly supported steady state while the free energy stays bounded from below. For critical exponent with subcritical mass, we also notice that the free energy is bounded by zero from below, but in this case solutions tend to the self-similar profile mentioned in the previous sections. By plotting the solution in self-similar variables, this scaling is numerically verified. Finally, we observe finite-time blow-up for critical exponent with supercritical mass, and for supercritical exponent. In both cases, the free energy is unbounded from below.


\section{Existence of weak solutions via the JKO scheme}\label{sec:existence_one_species}

Once we understood the properties of the free energy \eqref{eq:functional} we study existence of weak solutions of~\eqref{eq:thin_film_intro}. The variational structure of Eq.~\eqref{eq:thin_film_intro} allows to construct a candidate approximate solutions by means of the so-called JKO scheme or minimising movement, cf.~\cite{JKO98,AGS}. For a fixed $\tau>0$, we define the following sequence recursively
\begin{equation}\label{eq:JKO}
\begin{split}
        \rho^0_{\tau} &:= \rho_0,\\
        \rho^{k+1}_{\tau}  &\in \argmin_{\rho \in \mathcal{P} (\mathbb{R}^d)} \left\lbrace \frac{\mathcal{W}_2^2(\rho, \rho^k_{\tau})}{2 \tau} + \mf_m[\rho] \right\rbrace,\, \mbox{given } \rhotk,\; k\ge0.
\end{split}
\end{equation}

First, we prove the above scheme is well-defined, which is not immediate due to the negative component in the energy functional, or destabilising term. Let us fix $\overline{\rho} \in \mathcal{P}_2^a(\mathbb{R}^d)$ and define the functional
\begin{align*}
	\begin{split}
		\JKOstep_m\colon& \mathcal{P} (\mathbb{R}^d)   \longrightarrow  \overline{\mathbb{R}} \\
		& \quad \rho \quad \mapsto \, \, \frac{\mathcal{W}_2^2(\rho, \overline{\rho})}{2 \tau} + \mf_m[\rho].
	\end{split}
\end{align*}
\begin{prop}\label{prop:Existence minimiser}
    Let $\overline{\rho} \in \mathcal{P}_2^a (\mathbb{R}^d)$ and $1\leq m<2+\frac{2}{d}$ or $m = 2+\frac{2}{d}$ with subcritical mass, i.e. $\chi<\chi_c$. The functional $\JKOstep_m$ admits minimisers in the set $\{\rho\in\mP^a(\Rd) : \nabla\rho\in L^2(\Rd)\}$. Moreover, $\rho \in L^p (\Rd )$ for $p\in[1,2^{\ast}]$.
\end{prop}

Existence of minimisers is based on the \textit{direct method of calculus of variations}, as we prove below.

\begin{rem}\label{rem:no weak l.s.c.} 
Note that $\mathcal{W}^2_2(\rho,\overline{\rho})$ and $\|\nabla\rho\|^2_{L^2(\mathbb{R}^d)}$ are lower semicontinuous with respect to weak convergence in $\mathcal{P}(\mathbb{R}^d)$ and $L^2(\mathbb{R}^d)$, respectively. However, the negative terms in the free energy,     
    \begin{equation*}
        -\frac{1}{m-1} \| \rho \|_{L^m(\mathbb{R}^d)}^m \qquad \text{and} \qquad - \int_{\Rd} \rho \log \rho \, dx,
    \end{equation*} 
    are both upper (and not lower) semicontinuous with respect to the weak convergence in $L^m(\mathbb{R}^d)$. In particular, our functional $\JKOstep_m$ cannot be weakly lower semicontinuous.
\end{rem}
\begin{proof}[Proof \Cref{prop:Existence minimiser}] We split the proof into three parts.

    \textbf{Step 1:} Boundedness from below and minimising sequence. Taking into account the definition of the free energy functional \eqref{eq:functional}, we look for minimisers $\rho\in\mathcal{P}^a(\mathbb{R}^d)$ such that $\nabla\rho\in L^2(\mathbb{R}^d)$, otherwise the functional is infinite. Due to \eqref{eq:Free energy bounded below} we have the bound from below
    \begin{equation}\label{eq:a.v. mf[rho] bounded}
        \mf_m[\rho] \geq C,
    \end{equation}
    which implies $\JKOstep_m [\rho] \ge C$.
Boundedness from below ensures we can consider a minimising sequence, $\rho_n$, for which we also know $\mathcal{A}_m\le C$. Since the functional  $\mathcal{F}_m$ is bounded from below, we obtain the bound for the second order moment
\begin{align}\label{eq:second moment bdd}
    \begin{split}
        \mathrm{m}_2 (\rho_n) & \leq 2 \mathcal{W}_2^2 (\rho_n , \overline{\rho} ) + 2 \mathrm{m}_2 (\overline{\rho}) = 4 \tau \JKOstep_m [\rho_n] - 4 \tau \mf_m[\rho_n] + 2 \mathrm{m}_2 (\overline{\rho}) \\
        & \leq C T \left(  1 + \mathrm{m}_2 (\overline{\rho}) \right), 
    \end{split}
    \end{align}
for a different constant $C$.
    
    \textbf{Step 2:} Lower semicontinuity and compactness. First we comment on the lower semicontinuity of $\JKOstep_m$ with respect to a suitable convergence, i.e.
    \begin{equation*}
        \liminf_{n \to\infty} \JKOstep_m [\rho_n] \geq \JKOstep_m [\rho].
    \end{equation*}
From \Cref{rem:no weak l.s.c.} we infer we cannot have lower semicontinuity with respect to weak convergence in all terms, but we have it with respect to the convergence
\begin{equation*}
    \nabla \rho_n \rightharpoonup \nabla \rho \quad  \text{in } L^2(\mathbb{R}^d),
\end{equation*}
\vspace{-.4cm}
    \begin{align*} 
        \begin{cases}\rho_n \rightarrow \rho \quad  \text{in } L^m (\mathbb{R}^d) & \text{if } 1 <m \le 2+ \frac{2}{d}, \\
        \rho_n \log \rho_n \rightarrow \rho \log \rho \quad  \text{in } L^1 (\mathbb{R}^d) &\text{if } m =1.
        \end{cases}
    \end{align*}
    Let us note that \eqref{eq:grad rho is L2}  combined with \eqref{eq:a.v. mf[rho] bounded} implies that $\rho_n$ and $\nabla \rho_n$ are uniformly bounded on $L^m (\mathbb{R}^d)$ and $L^2 (\mathbb{R}^d)$, respectively, as $\rho_n$ is a minimising sequence.
    
    \textbf{Step 2a:} Strong $L^m$ convergence of $\rho_n$. 
    If $1 < m < 2 + \frac{2}{d}$  or $m = 2+\frac{2}{d}$ with $\| \rho_0 \|_{L^1 (\mathbb{R}^d)} < M_c$, since 
    \begin{equation}\label{eq:Lm bdd}
        \| \rho_n \|_{L^m(\mathbb{R}^d)} \leq C,
    \end{equation}
    by Banach-Alaoglu Theorem, up to pass to a subsequence,
    \begin{equation}\label{eq:weak convergence Lm}
        \rho_n \rightharpoonup \rho \quad  \text{weakly in } L^m(\mathbb{R}^d).
    \end{equation}
    Taking into account \eqref{eq:a.v. mf[rho] bounded}, \eqref{eq:second moment bdd}, and \eqref{eq:Lm bdd}, 
    we can restrict to the set
    \begin{equation*}
        \mathcal{H}_m := \left\lbrace f \in \mathcal{P}^a(\mathbb{R}^d ) : \mathrm{m}_2(f), \, \| \nabla f \|_{L^2 (\mathbb{R}^d)}, \, |\mf_m[f]| \leq C \right\rbrace.
    \end{equation*}
    Furthermore, from \eqref{eq:control of the Lp norm}, if $f \in \mathcal{H}_m$, then
    \begin{equation}\label{eq:nabla f is unif bdd L2}
        \| f \|_{L^p (\mathbb{R}^d)} \leq C,
    \end{equation}
    for all $p\in[1,2^{\ast}]$.
    
    Next, we prove that $\mathcal{H}_m$ is relatively compact in $L^m(\mathbb{R}^d)$ by means of Kolmogorov--Riesz--Fréchet Theorem \cite[Corollary 4.27]{Bre11}. In particular, we first show the uniform continuity estimate: $\|f(\cdot+h)-f(\cdot)\|_{L^m(\mathbb{R}^d)}\rightarrow 0$ as $|h|\rightarrow 0^+$. We distinguish two cases: $m=2$ and $m \neq 2$. 

    \textit{Case I: $m=2$.} Let us take
    \begin{align*}
        \int_{\mathbb{R}^d} | f (x+h)-f(x)|^2 \dx & = \int_{\mathbb{R}^d} \left| \int_0^1 \frac{\mathrm{d} }{\mathrm{d}s} \left(f (x+ s h) \right) \, \mathrm{d}s    \right|^2 \, \dx \\
        & = |h|^2\int_{\mathbb{R}^d} \left| \int_0^1 \nabla f (x+ s h)  \, \mathrm{d}s    \right|^2 \, \dx \\
        & \leq |h|^2 \int_{\mathbb{R}^d}  \int_0^1 \left| \nabla f (x+ s h) \right|^2 \, \mathrm{d}s  \, \dx \rightarrow 0,
    \end{align*}
    since $\| \nabla f \|_{L^2 (\mathbb{R}^d)} \leq C$ for every $f \in \mathcal{H}_m$.

    \textit{Case II: $m \neq 2$.} We use $L^p$ interpolation and apply \textit{Case I} afterwards. If $1 < m< 2$,
    \begin{align*}
        \| f(\cdot+h)-f(\cdot) \|_{L^m (\mathbb{R}^d)} & \leq \| f(\cdot+h)-f(\cdot) \|_{L^1 (\mathbb{R}^d)}^{\frac{2-m}{m}} \|f(\cdot+h)-f(\cdot) \|_{L^2 (\mathbb{R}^d)}^{\frac{2m-2}{m}} \\
        & \leq \left(2 \| f \|_{L^1 (\Rd)} \right)^{\frac{2-m}{m}} \|f(\cdot+h)-f(\cdot) \|_{L^2 (\mathbb{R}^d)}^{\frac{2m-2}{m}} \rightarrow 0.
    \end{align*}
    If $2 < m < 2 + \frac{2}{d}$ or $m = 2+\frac{2}{d}$ with subcritical mass, by using a different $L^p$ interpolation we obtain
     \begin{align*}
        \| f(\cdot+h)-f(\cdot) \|_{L^m (\mathbb{R}^d)} & \leq \| f(\cdot+h)-f(\cdot) \|_{L^{2(m-1)} (\mathbb{R}^d)}^{\frac{m-1}{m}} \|f(\cdot+h)-f(\cdot) \|_{L^2 (\mathbb{R}^d)}^{\frac{1}{m}} \\
        & \leq \left(2 \| f \|_{L^{2(m-1)} (\mathbb{R}^d)} \right)^{\frac{m-1}{m}} \|f(\cdot+h)-f(\cdot) \|_{L^2 (\mathbb{R}^d)}^{\frac{1}{m}} \rightarrow 0,
    \end{align*}
    and the convergence follows from \eqref{eq:control of the Lp norm} given that $2(m-1)<2^{\ast}$.

    In order to prove uniform integrability at infinity we first use Holder's inequality to show that 
    \begin{equation*}
        \int_{\mathbb{R}^d \backslash B_R} f(x)^m\,\dx \leq \frac{1}{R^{2 \delta}} \left( \int_{\mathbb{R}^d 
        } |x|^2 f (x)\,\dx\right)^{\delta} \left( \int_{\mathbb{R}^d
        } f(x)^{\frac{m- \delta}{1 - \delta}}\,\dx \right)^{1-\delta}.
    \end{equation*}
    Now $\delta\in\left(0,1\right)$ can be chosen so that the exponent $p:=\frac{m-\delta}{1-\delta}$ satisfies $p\in[1,2^{\ast}]$. Hence, by  \eqref{eq:control of the Lp norm}, $\|f\|_{L^p(\mathbb{R}^d)}$ is uniformly bounded.
   In particular, by taking the $R\rightarrow+\infty$ limit, and using that $f\in \mathcal{H}_m$ has uniformly bounded second moments we obtain the uniform integrability at infinity.
    
    Then, $\mathcal{H}_m$ is relatively compact in $L^m (\mathbb{R}^d)$ and combining it with \eqref{eq:weak convergence Lm} we obtain,
\begin{equation}\label{eq:strong convergence Lm}
        \rho_n \rightarrow \rho \quad  \text{in } L^m (\mathbb{R}^d).
    \end{equation}
    
    If $m=1$, since $\mf_1[\rho] \geq \mf_2[\rho]$ we have that $\mathcal{H}_1 \subseteq \mathcal{H}_2$. From here, we recover \eqref{eq:nabla f is unif bdd L2}, and \eqref{eq:strong convergence Lm} for $m=2$. We show $\rho_n\log\rho_n\to\rho\log\rho$ in $L^1(\Rd)$ via an extended version of Lebesgue's Dominated Convergence Theorem \cite[Chapter~4, Theorem~17]{Roy88}. Note that strong convergence in $L^2(\Rd)$ implies that, up to a subsequence,
    \begin{equation*}
        \rho_n \log \rho_n \rightarrow \rho \log \rho \quad \text{a.e. in } x\in\Rd.
    \end{equation*}
Furthermore, it is easy to check the majorant $|\rho_n(x) \log \rho_n(x)| \leq \rho_n^2(x) + \rho_n^{\frac{1}{2}}(x)$, for any $x\in\Rd$.
    We claim that $\rho_n^2 + \rho_n^{\frac{1}{2}} \rightarrow \rho^2 + \rho^{\frac{1}{2}}$ strongly in $L^1 (\Rd)$. Since $\mathcal{H}_1 \subseteq \mathcal{H}_2$ it is enough to show $ \rho_n^{\frac{1}{2}} \rightarrow \rho^{\frac{1}{2}}$ strongly in $L^1 (\Rd)$.
    Applying Jensen's inequality for concave functions we have $\rho_n^{\frac{1}{2}}\in L^1(\Rd)$, while continuity of the square root function ensures
    \begin{equation*}
        \rho_n^{\frac{1}{2}}(x) \rightarrow \rho^{\frac{1}{2}}(x) \quad \text{a.e. in } x\in\Rd.
    \end{equation*}
    By applying Fatou's Lemma,
    \begin{equation}\label{eq:lim inf 1/2}
        \liminf_{n \rightarrow \infty} \int_{\Rd} \rho_n^{\frac{1}{2}} \, dx \geq \int_{\Rd} \rho^{\frac{1}{2}} \, \dx\,,
    \end{equation}
    and concavity implies,
    \begin{equation}\label{eq:lim sup 1/2}
        \limsup_{n \rightarrow \infty} \int_{\Rd} \rho_n^{\frac{1}{2}} \, \dx \leq \int_{\Rd} \rho^{\frac{1}{2}} \, \dx.
    \end{equation}
    Combining \eqref{eq:lim inf 1/2} and \eqref{eq:lim sup 1/2} we infer
    \begin{equation*}
        \lim_{n \rightarrow \infty} \int_{\Rd} \rho_n^{\frac{1}{2}} \, dx = \int_{\Rd} \rho^{\frac{1}{2}} \, \dx.
    \end{equation*}
Applying the extended Dominated Convergence Theorem we obtain
    \begin{equation*}
        \rho_n \log \rho_n \rightarrow \rho \log \rho \quad \text{in } L^1 (\Rd ).
    \end{equation*}

    \textbf{Step 2b:} Weak $L^2$ convergence of $\nabla \rho_n$.
    Given that $\nabla\rho$ is bounded in $L^2(\Rd)$, from Banach-Alaoglu Theorem we obtain that up to a subsequence,
    \begin{equation*}
        \nabla \rho_n \rightharpoonup \nabla \rho \quad  \text{weakly in } L^2(\mathbb{R}^d).
    \end{equation*}
    Note that the limit is $\nabla \rho$, which can be checked by testing $\nabla\rho$ against a smooth and compactly supported test function, and using the convergence $\rho_n\rightarrow \rho$ in $L^m(\mathbb{R}^d)$ that we proved in the previous step.
    
    \textbf{Step 3:} Existence of minimisers. Due to the Weierstrass criterion for the existence of minimisers, cf.~e.g.~\cite[Box 1.1]{San15}, $\JKOstep_m$ has at least one minimiser in $\mathcal{H}_m$.
\end{proof}

As mentioned in~Section~\ref{sec:self_similarity}, the proof of~Proposition~\ref{prop:minimisers_fokker_planck} can be obtained by adapting the previous one to the functional $\mathcal{L}:\mP^a(\Rd)\to\R$ given by
\[
\mathcal{L}[u]=\mathcal{F}_{m_c}[u]+\frac{b}{2}\int|z|^2u(z)\,\mathrm{d}z.
\]
\begin{proof}[Proof of Proposition~\ref{prop:minimisers_fokker_planck}]
    Boundedness from below follows from Gagliardo--Nirenberg inequality and non-negativity of the additional term in $\mathcal{L}[u]$, as noted in~\Cref{prop:prop_energies_critical}. For a minimising sequence $\{u_n\}_{n\in \mathbb{N}}$, since $\chi<\chi_c$ we derive the following bounds, again as a consequence of Gagliardo--Nirenberg, cf.~Proposition~\ref{prop:prop_energies_critical}:
\begin{align*}
    \|u_n\|_{L^p(\Rd)}\le C,\quad \|\nabla u_n\|_{L^2(\Rd)}\le C,\quad \mathrm{m}_2(u_n)\le C,
\end{align*}
for $p\in[1,2^*]$ and a constant $C=C(m_c,p,d,\chi)>0$. Kolmogorov--Riesz--Fréchet Theorem provides relatively compactness in $L^{m_c}(\Rd)$ for the set $\{f\in \mP^a(\Rd):  \mathrm{m}_2(f), \|\nabla f\|_{L^2(\Rd)}, |\mathcal{L}[\rho]|\le C \}$, arguing as in~\Cref{prop:Existence minimiser}. For the sake of completeness, we point out the additional term is lower semicontinuous with respect to the weak-$L^2$ convergence by applying a cut-off and monotone convergence Theorem --- choosing $p=2$ we infer weak-$L^2$ convergence of $u_n$ from the above uniform bounds. Proceeding as in~Proposition~\ref{prop:Existence minimiser}, and for $\chi<\chi_c$, we can show existence of minimisers in $\{u\in \mP^a(\Rd): \nabla u\in L^2(\Rd),  \mathrm{m}_2(u)<+\infty\}$. 
\end{proof}

Proposition \ref{prop:Existence minimiser} guarantees the sequence is well-defined, as we can solve the minimisation problem in \eqref{eq:JKO}. Next, we set up the approximating solution to \eqref{eq:thin_film_intro}. Let $T>0$, and consider $N:=\left[\frac{T}{\tau}\right]$. We define the curve $\rhot: [0, T ] \rightarrow \mathcal{P} (\mathbb{R}^d)$ as the piecewise constant interpolation
\begin{equation}\label{eq:piecewise interpolation}
    \rhot(t) := \rhot^k,\quad t\in \left((k-1) \tau , k \tau \right],
\end{equation}
where $\rhot^k$ is defined in $\eqref{eq:JKO}$. We can prove convergence of this piecewise interpolation to a continuous curve with respect to the $2$-Wasserstein distance.

\begin{lem}[Narrow convergence \&  discrete uniform estimates]\label{lem: interpolation narrow convergence}
  Let $\rho_0\in\mathcal{P}_2^a(\Rd)$ such that $\mf_m [\rho_0]<+\infty$ and $1\leq m<2+\frac{2}{d}$ or $m = 2+\frac{2}{d}$ with subcritical mass $\chi<\chi_c$. There exists an absolutely continuous curve $\tilde{\rho} : [0,T] \rightarrow \mathcal{P}_2 (\mathbb{R}^d)$
    such that, up to a subsequence, $\rhot(t)$ narrowly converges to $\Tilde{\rho}(t)$, uniformly in $t \in [0,T]$.

\noindent Moreover, we obtain the following discrete uniform bounds:
    \begin{align}
   \sup_k\Vert\nabla\rhot^k\Vert_{L^2(\Rd)} &\leq C_1\left(\mf_m[\rho_0] + \Vert\rho_0\Vert_{L^1(\Rd)}^\alpha\right)^{1/2}<+\infty;
   \label{eq: discrete H1}
\\
        \sup_k\Vert\rhot^k\Vert_{L^p(\Rd)} &\leq C_2 < +\infty;
        \label{eq: discrete Lp}
\\
        \mathrm{m}_2(\rhot^k)&\leq 2\mathrm{m}_2(\rho_0)+4T\left(\mf_m[\rho_0]+C\right),
         \label{eq:uniform second order moment bound}
    \end{align}
     for constants $C_1 = C_1(m,d,\chi)>0$ and $C_2 = C_2(m,p,d,\rho_0,\chi)>0$, and for $p\in[1,2^{\ast}]$.
\end{lem}

\begin{proof}
   By construction of the sequence we have
    \begin{equation}
      \mf_m[\rhot^{k}]  \leq\frac{\mathcal{W}_2^2 (\rhot^k, \rhot^{k-1})}{2 \tau} + \mf_m[\rhot^k] \leq \mf_m[\rhot^{k-1}].
        \label{eq:prop_limit_1}
    \end{equation}
    In particular, this gives
  \begin{equation*}
      \sup_{k} \mf_m[\rhot^k] \leq \mf_m[\rho_0]<+\infty,
    \end{equation*}
which together with \eqref{eq:grad rho is L2} and \eqref{eq:control of the Lp norm} implies that $\Vert \nabla\rhot^k\Vert_{L^2(\Rd)}$ and $\Vert \rhot^k\Vert_{L^p(\Rd)}$ are uniformly bounded in $k$ and $\tau$ for $p\in[1,2^{\ast}]$. Hence we obtain \eqref{eq: discrete H1} and \eqref{eq: discrete Lp}. 
    
     Next, by summing up over $k$ in \eqref{eq:prop_limit_1} and using that the free energy is bounded from below, \eqref{eq:Free energy bounded below}, we deduce
    \begin{equation}
        \sum_{k = i+1}^{j} \frac{\mathcal{W}_2^2 (\rhot^k, \rhot^{k-1})}{2 \tau} \leq \mf_m[\rhot^{i}] - \mf_m[\rhot^{j}]\leq\mathcal{F}_m[\rho_0] + C.
        \label{eq:prop_limit_2}
    \end{equation}
    Therefore the $2$-Wasserstein distance between $\rho_0$ and $\rhot(t)$ is uniformly bounded. Indeed, for $t\in((j-1)\tau,j\tau]$,
    \begin{align*}
        \mathcal{W}_2^2 (\rho_0 , \rhot(t) ) & \leq j\sum_{k = 1}^{j} \mathcal{W}_2^2 (\rhot^k,\rhot^{k-1}) \leq 2 j\tau  (\mf_m[\rho_0] + C) \le 2T (\mf_m[\rho_0]+C).
    \end{align*}
    Furthermore, we obtain second order moments are uniformly bounded on compact time intervals $[0,T]$ since
    \begin{equation*}
        \mathrm{m}_2\left(\rhot(t)\right) \leq 2\mathrm{m}_2(\rho_0) + 2\mw_2^2(\rho_0,\rhot(t))\leq 2\mathrm{m}_2(\rho_0) + 4T\left(\mf_m[\rho_0]+C\right).
    \end{equation*}

    Let us now prove equi-continuity. Consider $0 \leq s < t$ such that $s \in ((i-1) \tau , i \tau ]$ and $t \in ((j-1) \tau , j \tau ]$. Using Cauchy--Schwarz inequality and \eqref{eq:prop_limit_2} we have
    \begin{align}\label{eq:equicontinuity on d2}
    \begin{split}
        \mathcal{W}_2 (\rhot (s) , \rhot (t) ) & \!\leq\! \sum_{k=i+1}^j \mathcal{W}_2 (\rhot^{k} , \rhot^{k-1} ) \!\leq\! \left( \sum_{k=i+1}^j\mathcal{W}_2^2 (\rhot^{k} , \rhot^{k-1} ) \right)^{\frac{1}{2}} \!\!|j-i|^{\frac{1}{2}} \\
        & \!\leq\! \left( 2  (\mf_m[\rho_0] + C) \right)^{\frac{1}{2}} \left( \sqrt{|t-s|} + \sqrt{\tau} \right).
    \end{split}
    \end{align}
    Thus, $\rhot$ is $\frac{1}{2}$-H\"{o}lder equi-continuous up to a negligible error of order $\sqrt{\tau}$. Therefore, by a refined version of the Ascoli-Arzelà Theorem \cite[Proposition 3.3.1]{AGS}, we obtain that $\rhot$ admits a subsequence narrowly converging to a limit $\tilde{\rho}$ as $\tau \rightarrow 0^+$, uniformly on $[0,T]$. Moreover, using the uniform bound \eqref{eq:uniform second order moment bound} and that $|\cdot|^2$ is lower semicontinuous and bounded from below, we obtain that the limiting curve $\Tilde{\rho}$ has bounded second order moments,
    \begin{equation*}
      \mathrm{m}_2(\Tilde{\rho}(t))\leq \liminf_{\tau\downarrow 0} \mathrm{m}_2(\rhot(t)), \quad \forall t\in[0,T].\qedhere
    \end{equation*}
\end{proof}

The bounds \eqref{eq: discrete Lp} and \eqref{eq: discrete H1} imply weak convergence of the interpolation $\rhot$ to a probability density $\tilde{\rho}$ with regularity provided below.

\begin{prop}[Weak convergence]\label{prop: weak convergence interpolation} Let $\rho_0\in\mathcal{P}_2^a(\Rd)$ such that $\mf_m[\rho_0]<+\infty$ and $1\leq m<2+\frac{2}{d}$ or $m = 2+\frac{2}{d}$ with subcritical mass $\chi<\chi_c$. The piecewise interpolation $\rho_{\tau}$ in \eqref{eq:piecewise interpolation} is such that $\rho_\tau\in L^\infty([0,T];L^p(\Rd))\cap L^\infty([0,T];H^1(\Rd))$, for any $p\in[1,2^{\ast}]$. In particular, the limit $\tilde{\rho}$ belongs to $L^\infty([0,T];L^p(\Rd))\cap L^\infty([0,T];H^1(\Rd))$ and
\begin{equation*}
    \rhot\rightharpoonup \Tilde{\rho}\quad\mbox{in } L^2([0,T];H^1(\Rd)).
\end{equation*}
\end{prop}
\begin{proof}
    From \eqref{eq: discrete Lp} in \Cref{lem: interpolation narrow convergence} we have
    \begin{equation*}
        \Vert\rhot \Vert_{L^\infty([0,T];L^p(\Rd))} = \sup_{t\in(0,T)}\Vert\rhot(t)\Vert_{L^p(\Rd)} = \sup_k\Vert\rhot^k\Vert_{L^p(\Rd)}<+\infty,
    \end{equation*}
    for $p\in[1,2^{\ast}]$. 
    Analogously, from \eqref{eq: discrete H1} we obtain $\nabla\rhot\in L^\infty([0,T];L^2(\Rd))$. In particular, for any compact time interval $[0,T]$ with $T>0$, we have $\|\rhot\|_{L^2([0,T];H^1(\Rd))}\le C$ uniformly in $\tau$ and the weak convergence follows from Banach-Alaoglu Theorem. Regularity of the limit follows from standard arguments.
\end{proof}

The uniform-in-$\tau$ $L^\infty([0,T];H^1(\Rd))$ estimate allows us to obtain strong convergence of $\rhot$ via a refined version of the Aubin-Lions Lemma due to Rossi and Savar\'e --- cf.~\Cref{prop:RS03}.

\begin{prop}[Strong convergence of $\rhot$]\label{prop: strong convergence rho} Let $\rho_0\in\mathcal{P}_2^a(\Rd)$ such that $\mf_m[\rho_0]<+\infty$ and $1\leq m<2+\frac{2}{d}$ or $m = 2+\frac{2}{d}$ with subcritical mass $\chi<\chi_c$. The sequence $\rho_\tau$ converges, up to a subsequence, strongly to the curve $\Tilde{\rho}$ 
in $L^2([0,T]; L^2(\Rd))$
for every $T > 0$.
\end{prop}
\begin{proof}
We apply \Cref{prop:RS03} to a subset $U = \{\rho_\tau\}_{\tau\geq 0}$ for $X = L^2(\Rd)$ and $g:=\mw_2$, the 2-Wasserstein distance. Further, we consider the functional $\mathcal{I}:L^2(\Rd)\rightarrow [0,+\infty]$ defined by
\begin{equation*}
    \mathcal{I}[\rho]=\begin{cases}\Vert\rho\Vert_{H^1(\Rd)}^2+\int_\Rd|x|^2\rho(x)\,\dx&\rho\in\mathcal{P}_2(\mathbb{R}^d)\cap H^1(\Rd),\\
+\infty & \mbox{otherwise}.
\end{cases}
\end{equation*}
Note that $\mathcal{W}_2$ is a distance on the proper domain of $\mathcal{I}$. Indeed, if $\mathcal{I}[\rho]<\infty$ then $\rho\in\mathcal{P}_2(\mathbb{R}^d)$. Lower semicontinuity of $\mathcal{I}$ follows from standard arguments --- see for instance \cite{BE22}.

Next, let $B_c = \{\rho\in L^2(\Rd):\mathcal{I}[\rho]\leq c\}$ be a sublevel of $\mathcal{I}$. We notice that $B_c\subset \mathcal{P}_2(\Rd)\cap H^1(\Rd)$ and thus we can apply Kolmogorov-Riesz-Fréchet Theorem \cite[Corollary 4.27]{Bre11} as in the proof of \Cref{prop:Existence minimiser} to obtain that $B_c$ is relatively compact. Hence we have $\mathcal{I}$ is an admissible functional.

The tightness condition \eqref{tightness} follows from the uniform-in-$\tau$ second order moment and $L^\infty([0,T];H^1(\Rd))$ bounds for $\rhot$ given in \eqref{eq:uniform second order moment bound} and \Cref{prop: weak convergence interpolation}. The integral equi-continuity condition \eqref{integral equicontinuity} can be seen from the H\"{o}lder equi-continuity of $\rhot$, proved in \Cref{lem: interpolation narrow convergence}. More precisely, for $h>\tau$ we have
\begin{equation*}
\int_0^{T-h} \mathcal{W}_2\left(\rho_\tau(t+h),\rho_\tau(t)\right)\mathrm{d}t \leq\int_0^{T-h} C \left(\sqrt{h}+\sqrt{\tau}\right)\mathrm{d}t\leq 2CT\sqrt{h},
\end{equation*}
for a constant $C>0$ independent of $\tau$ and $h$. If instead, $h<\tau$, we can write
\begin{align*}
    \int_0^{T-h} \mathcal{W}_2\left(\rho_\tau(t+h),\rho_\tau(t)\right)\mathrm{d}t & \leq h\sum_{k=0}^{N-1}\mathcal{W}_2(\rho_\tau^{k+1},\rho_\tau^k)\\&\leq h \sqrt{N}\sum_{k=0}^{N-1}\mathcal{W}_2^2(\rho_\tau^{k+1},\rho_\tau^k) \leq C h\sqrt{T},
\end{align*}
where $C>0$ is the constant defined in \eqref{eq:prop_limit_2}. 

Hence we can apply \Cref{prop:RS03} to obtain that there exists a subsequence, that we label by $\tau\downarrow 0$, such that $\rhot$ converges in measure to $\Tilde{\rho}$, as in \eqref{eq:convergence_measure_rs}, where $X:=L^2(\Rd)$. 
Let us denote by $A_\delta(\tau):=\{t\in(0,T):\;\Vert \rhot(t) - \Tilde{\rho}(t)\Vert_X\geq\delta\}$, which vanishes as $\tau\to0$. Owing to~\eqref{eq: discrete Lp}~and Proposition~\ref{prop: weak convergence interpolation} we can prove (see e.g.~\cite[Proposition 4.3]{JAC_ESP_WU2023nonlocal})
\[
\limsup_{\tau\to0}\Vert\rhot-\Tilde{\rho}\Vert^2_{L^2 ([0,T];L^2(\Rd))}\le \delta T^{1/2},
\]
hence strong convergence in $L^2 ([0,T];L^2(\Rd))$ since $\delta$ is arbitrarily small.
\end{proof}

\subsection{Flow interchange}
The strong convergence of the sequence $\rhot$ obtained in~\Cref{prop: strong convergence rho} is not enough to pass to the limit in the Euler-Lagrange equation associated to \eqref{eq:JKO} and arrive to a weak formulation of our equation. We use the heat equation as auxiliary flow to obtain uniform bounds on the Hessian of the sequence $\{\rhot\}_\tau$, cf.~\cref{sec:preliminaries}. More precisely, we exploit that the heat equation is a $2$-Wasserstein gradient flow of the entropy functional $\me[\rho]=\int\rho\log\rho\, \dx$.

In the following, for $\mu \in \mathcal{P}_2 (\Rd)$ such that $\me[\mu]<\infty$, we denote by $S_{\me}^t \mu$ the solution at time $t$ of the heat equation for an initial value $\mu$ at $t=0$. Furthermore, we also define the dissipation of $\mf_m$ along $S_{\me}$ by
\begin{equation*}
    D_{\me} \mf_m [\rho] := \limsup_{s \downarrow 0} \left\lbrace \frac{\mf_m[\rho] - \mf_m[S_{\me}^s \rho]}{s} \right\rbrace.
\end{equation*}
\begin{rem} \label{rem: HE}
Given some initial datum $\mu_0\in \mathcal{P}(\Rd)$ the solution of the heat equation, $S_\me^t\mu_0$, can be written as the convolution of the heat kernel $G_t$ with the initial condition, i.e.
\begin{equation*}
    S_\me^t\mu_0 = G_t * \mu_0 = {\left(4\pi t\right)^{-d/2}}\int_\Rd e^{-{|x-y|^2}/{4t}}\mathrm{d}\mu_0(y).
\end{equation*}
As a consequence, $S_\me^t\mu_0\in C^\infty\left((0,+\infty)\times\Rd\right)$. Moreover, for solutions of the heat equation we can integrate by parts to obtain the well-known equality
\begin{equation}
\label{HE: hessian laplacian}
    \int_\Rd|\Delta S_\me^t\mu_0|^2 \,\dx = \int_\Rd| D^2S_\me^t\mu_0|^2 \,\dx\,.
\end{equation}
\end{rem}

We are now ready to prove an $H^2$ bound for $\rho^{\tau}$.

\begin{lem}[$H^2$ uniform bound]\label{H2 bound flow interchange}
     Let $\rho_0$ such that $\mf_m[\rho_0]<+\infty$, and $1\leq m<2+\frac{2}{d}$ or $m = 2+\frac{2}{d}$ with subcritical mass $\chi<\chi_c$. The piecewise interpolation $\rho_{\tau}$ constructed in \eqref{eq:piecewise interpolation} is such that $\rho_\tau\in L^2([0,T];H^2(\Rd))$. In particular, we obtain the uniform-in-$\tau$ bound
    \begin{equation*}
        \| \Delta \rho_{\tau} \|_{L^2 ([0,T];L^2(\Rd))}^2 \leq d \| D^2\rhot\|_{L^2([0,T];L^2(\Rd))}^2\leq C,
    \end{equation*}
    where $C= C(m,d, \rho_0, T)>0$.
\end{lem}

\begin{proof} 
    For all $s > 0$, we consider $S_{\me}^s \rho_{\tau}^{k+1}$. Then, by the definition of the scheme \eqref{eq:JKO} and of $\rho_{\tau}^{k+1}$, we have the inequality
    \begin{equation*}
        \frac{1}{2 \tau} \mathcal{W}_2^2 (\rho_{\tau}^k, \rho_{\tau}^{k+1} ) + \mf_m[\rho_{\tau}^{k+1}] \leq \frac{1}{2 \tau} \mathcal{W}_2^2 (\rho_{\tau}^k , S_{\me}^s \rho_{\tau}^{k+1} ) + \mf_m [S_{\me}^s \rho_{\tau}^{k+1}],
    \end{equation*}
    from which we obtain,
    \begin{equation*}
        \tau \frac{\mf_m[\rho_{\tau}^{k+1}] - \mf_m[S_{\me}^s \rho_{\tau}^{k+1}]}{s} \leq \frac{1}{2} \frac{\mathcal{W}_2^2 (\rho_{\tau}^k, S_{\me}^s \rho_{\tau}^{k+1}) - \mathcal{W}_2^2 (\rho_{\tau}^k, \rho_{\tau}^{k+1})}{s}.
    \end{equation*}
    By taking the $\limsup$ as $s \downarrow 0$ we obtain
    \begin{equation}\label{eq:Apply EVI}
        \tau D_{\me} \mf_m[\rho_{\tau}^{k+1}] \leq \frac{1}{2}\left.\frac{\mathrm{d}^+}{\mathrm{d}t}\right|_{t=0}  \mathcal{W}_2^2 (\rho_{\tau}^k , S_{\me}^t \rho_{\tau}^{k+1} )  \leq \me[\rho_{\tau}^k] - \me[\rho_{\tau}^{k+1}],
    \end{equation}
    where in the last inequality we use the \eqref{eq:EVI}, as $S_{\me}$ is a $0$-flow, cf.~Definition \ref{def:lambda_flow}. Note that
    \begin{align}\label{eq:compute DE}
        \begin{split}
            D_{\me} \mf_m [\rho_{\tau}^{k+1}] & = \limsup_{s \downarrow 0} \left\lbrace \frac{\mf_m[\rho_{\tau}^{k+1}] - \mf_m[S_{\me}^s \rho_{\tau}^{k+1}]}{s} \right\rbrace \\
            & = \limsup_{s \downarrow 0} \int_0^1 \left( - \left.\frac{\mathrm{d
            }}{\mathrm{d}z}\right|_{z=st} \mf_m [S_{\me}^z \rho_{\tau}^{k+1}] \right) \, \mathrm{d}t.
        \end{split}
    \end{align}
    From this point of the proof, we distinguish between two cases. 

    \textit{Case I: $1 < m < 2 + \frac{2}{d}$ or $m = 2+\frac{2}{d}$ with subcritical mass $\chi<\chi_c$.}
    Let us compute the time derivative:  \begin{align}\label{eq:Heat flow interchange}
        \begin{split}
            \ddt \mf_m[S_{\me}^t \rho_{\tau}^{k+1}] = & - \int_{\Rd} | \Delta S_{\me}^t \rho_{\tau}^{k+1} |^2 \dx \\
            &  - \frac{\chi m}{m-1} \int_{\Rd} (S_{\me}^t \rho_{\tau}^{k+1})^{m-1} \Delta S_{\me}^t \rho_{\tau}^{k+1} \dx.
       \end{split}
    \end{align}
    Therefore, combining \eqref{eq:Apply EVI}, \eqref{eq:compute DE} and \eqref{eq:Heat flow interchange} we obtain
    \begin{align*}
        & \tau \limsup_{s \downarrow 0} \int_0^1  \left( \int_{\Rd} | \Delta S_{\me}^{st} \rho_{\tau}^{k+1} |^2 \, \dx + \frac{\chi m}{m-1} \int_{\Rd} (S_{\me}^{st} \rho_{\tau}^{k+1})^{m-1} \Delta S_{\me}^{st} \rho_{\tau}^{k+1} \, \dx \right) \, \mathrm{d}t \\
        & \qquad \leq  \me[\rho_{\tau}^k] - \me[\rho_{\tau}^{k+1}].
    \end{align*}
  By applying Young's inequality we have 
    \begin{align*}
        & \int_{\Rd} | \Delta S_{\me}^{st} \rho_{\tau}^{k+1} |^2 \, \dx + \frac{\chi m}{m-1} \int_{\Rd} (S_{\me}^{st} \rho_{\tau}^{k+1})^{m-1} \Delta S_{\me}^{st} \rho_{\tau}^{k+1} \, \dx \\ &\qquad \geq
     \int_{\Rd} | \Delta S_{\me}^{st} \rho_{\tau}^{k+1} |^2 \, \dx - \frac{\chi m}{m-1} \int_{\Rd} |(S_{\me}^{st} \rho_{\tau}^{k+1})^{m-1}| | \Delta S_{\me}^{st} \rho_{\tau}^{k+1} | \, \dx\\&\qquad \geq
        \frac{1}{2} \int_{\Rd} | \Delta S_{\me}^{st} \rho_{\tau}^{k+1} |^2 \, \dx - \frac{\chi^2 m^2}{2(m-1)^2} \int_{\Rd} |S_{\me}^{st} \rho_{\tau}^{k+1}|^{2(m-1)} \, \dx  \,,
    \end{align*}
    which gives
    \begin{align*}
        & \frac{\tau}{2} \liminf_{s \downarrow 0} \int_0^1  \int_{\Rd} | \Delta S_{\me}^{st} \rho_{\tau}^{k+1} |^2 \, \dx \, \mathrm{d}t \\
        & \qquad \leq  \me[\rho_{\tau}^k] - \me[\rho_{\tau}^{k+1}] + \tau \frac{\chi^2 m^2}{2(m-1)^2} \limsup_{s \downarrow 0} \int_0^1 \int_{\Rd} |S_{\me}^{st} \rho_{\tau}^{k+1}|^{2(m-1)} \, \dx  \, \mathrm{d}t.
    \end{align*}

    In order to take the $s\downarrow 0$ limit in the above expression, first we note that, in view of \Cref{rem: HE}, we can write $\Vert \Delta S_\me^{st}\rhot^{k+1}\Vert_{L^2(\Rd)} = \Vert  D^2 S_\me^{st}\rhot^{k+1}\Vert_{L^2(\Rd)}$. Since the auxiliary flow is the heat equation with initial datum $\rhotkk\in H^1(\Rd)$, we have $S_\me^t \rhotkk\to\rhotkk$ in $L^2(\Rd)$ as well as $\nabla S_\me^t \rhotkk\to\nabla\rhotkk$ in $L^2(\Rd)$ as $t\downarrow 0$ 
   --- by noticing that $\nabla S_\me^t\rhotkk$ is a solution to the heat equation with initial datum $\nabla\rhotkk\in L^2(\Rd)$.
   By the weak lower-semicontinuity of the $H^1$ seminorm we have 
   \begin{equation}\label{eq:Hessian and Laplacian}
      \liminf_{s \downarrow 0} \int_0^1  \int_{\Rd} |  D^2 S_{\me}^{st} \rho_{\tau}^{k+1} |^2 \, \dx \, \mathrm{d}t \geq \int_{\Rd} |  D^2 \rho_{\tau}^{k+1} |^2 \, \dx\,.
   \end{equation}
   
Next, we focus on the term involving $\Vert S_\me^{st}\rhot^{k+1}\Vert_{L^{2(m-1)}(\Rd)}$ and distinguish between two cases, depending on the value of $m$. We apply Young's convolution inequality to $S_\me^{st}\rhot^{k+1} = G_{st}*\rhot^{k+1}$, as noticed in \Cref{rem: HE}.

If $\frac{3}{2}\leq m<2 + \frac{2}{d}$ or $m = 2+\frac{2}{d}$ with subcritical mass, then $1\leq 2(m-1)<2^{\ast}$ and, by \eqref{eq:control of the Lp norm}, $\rho_\tau^{k+1}\in L^{2(m-1)}(\Rd)$. Furthermore, we have
\begin{equation*}
    \Vert S_\me^{st}\rhot^{k+1}\Vert_{L^{2(m-1)}(\Rd)} \leq \Vert G_{st}\Vert_{L^1(\Rd)}\Vert \rhot^{k+1}\Vert_{L^{2(m-1)}(\Rd)}=\Vert \rhot^{k+1}\Vert_{L^{2(m-1)}(\Rd)},
\end{equation*}
In particular, we obtain
\begin{equation*}
       \limsup_{s \downarrow 0} \int_0^1 \int_{\Rd} |S_{\me}^{st} \rho_{\tau}^{k+1}|^{2(m-1)} \, \dx  \, \mathrm{d}t \leq  \int_{\Rd} | \rho_{\tau}^{k+1}|^{2(m-1)} \, \dx\,.
   \end{equation*}

     If $1<m<\frac{3}{2}$, we use that the function $|\cdot|^{2(m-1)}$ is concave and apply Jensen's inequality to find
   \begin{align*}
\int_{\Rd} |S_{\me}^{st} \rho_{\tau}^{k+1}|^{2(m-1)} \, \dx &\leq    \left| \int_{\Rd} S_{\me}^{st} \rho_{\tau}^{k+1} \, \dx \right|^{2(m-1)}  =\Vert G_{st}*\rhot^{k+1}\Vert_{L^1(\Rd)}^{2(m-1)}\\ &\leq \Vert G_{st}\Vert_{L^1(\Rd)}^{2(m-1)}\Vert\rhot^{k+1}\Vert_{L^1(\Rd)}^{2(m-1)} = \Vert\rhot^{k+1}\Vert_{L^1(\Rd)}^{2(m-1)},
   \end{align*}
    whence
   \begin{equation*}
        \limsup_{s \downarrow 0} \int_0^1 \int_{\Rd} |S_{\me}^{st} \rho_{\tau}^{k+1}|^{2(m-1)} \, \dx  \, \mathrm{d}t \leq \left| \int_{\Rd}  \rho_{\tau}^{k+1}\, \dx\,\right|^{2(m-1)}.
   \end{equation*}   
As a consequence,
    \begin{equation*}
        \frac{\tau}{2} \int_{\Rd} |  D^2 \rho_{\tau}^{k+1} |^2 \, \dx \leq \me[\rho_{\tau}^k] - \me[\rho_{\tau}^{k+1}] + \tau \frac{\chi^2 m^2}{2(m-1)^2}   \Vert\rhot^{k+1}\Vert_{L^q(\Rd)}^{2(m-1)},
    \end{equation*}
    with $q = 2(m-1)$ for $\frac{3}{2}\leq m<2+\frac{2}{d}$ or $m = 2+\frac{2}{d}$ with subcritical mass, and $q = 1$ for $1<m<\frac{3}{2}$.
    By summing up over $k$ from $0$ to $N-1$, considering that $x \log x \leq x^2$ and \Cref{rem:lower bound rho log rho}, we recover, further using Jensen's inequality for concave functions for $q=1$,
    \begin{align*}
        \frac{1}{2} \|  D^2 \rho_{\tau} \|^2_{L^2([0,T];L^2(\Rd))} & \leq \me[\rho_0] - \me[\rho_{\tau}^N] + \frac{\chi^2 m^2}{2(m-1)^2} \sum_{k=0}^{N-1} \tau\Vert\rhot^{k+1}\Vert_{L^q(\Rd)}^{2(m-1)}\\
        & \leq \| \rho_0 \|_{L^2 (\Rd)}^2 + C(1 + \mathrm{m}_2(\rho_{\tau}^N))\\
        &\quad+ \frac{\chi^2 m^2}{2(m-1)^2 T^{2(m-1)-1}} \| \rho_{\tau} \|_{L^{q} ([0,T] ;L^q(\Rd))}^{2(m-1)},
    \end{align*}
    which is uniformly bounded, due to \Cref{lem: interpolation narrow convergence}.
    In particular, we also obtain 
    \begin{equation*}
        \Vert\Delta\rhot\Vert_{L^2([0,T];L^2(\Rd))}\leq \sqrt{d} \Vert D^2\rhot\Vert_{L^2([0,T];L^2(\Rd))}\leq C(m,d, \rho_0, \chi,T).
\end{equation*}

    \textit{Case II: $m=1$.}
    Let us compute the time derivative, 
    \begin{align}\label{eq:m=1 Heta flow interchange}
        \begin{split}
            \ddt \mf_1[S_{\me}^t \rho_{\tau}^{k+1}] &=  - \int_{\Rd} | \Delta S_{\me}^t \rho_{\tau}^{k+1} |^2 \, \dx \\
            &\quad  - \chi \int_{\Rd}  \Delta S_{\me}^t \rho_{\tau}^{k+1} (1 + \log S_{\me}^t \rho_{\tau}^{k+1} ) \, \dx\\
            &=- \int_{\Rd} | \Delta S_{\me}^t \rho_{\tau}^{k+1} |^2 \, \dx \\
            &\quad+\chi\int_{\Rd}  \nabla S_{\me}^{t} \rho_{\tau}^{k+1}\cdot \nabla \log S_{\me}^{t} \rho_{\tau}^{k+1} \, \dx\,.
        \end{split}
    \end{align}
    By combining \eqref{eq:Apply EVI}, \eqref{eq:compute DE} and \eqref{eq:m=1 Heta flow interchange}, we obtain
    \begin{align*}
        & \tau \limsup_{s \downarrow 0} \int_0^1 \left( \int_{\Rd} | \Delta S_{\me}^{st} \rho_{\tau}^{k+1} |^2 \, \dx - \chi\int_{\Rd}  \nabla S_{\me}^{st} \rho_{\tau}^{k+1}\cdot \nabla \log S_{\me}^{st} \rho_{\tau}^{k+1} \, \dx \right) \, \mathrm{d}t \\
        & \qquad \leq  \me[\rho_{\tau}^k] - \me[\rho_{\tau}^{k+1}]\,.
    \end{align*}
   Similarly to the previous case, we obtain
    \begin{align*}
        & \tau \liminf_{s \downarrow 0} \int_0^1 \int_{\Rd} | \Delta S_{\me}^{st} \rho_{\tau}^{k+1} |^2 \, \dx\, \mathrm{d}t \\
        & \qquad \leq  \me[\rho_{\tau}^k] - \me[\rho_{\tau}^{k+1}] + \chi\tau\limsup_{s \downarrow 0}\int_0^1\int_{\Rd}  \nabla S_{\me}^{st} \rho_{\tau}^{k+1}\cdot \nabla \log S_{\me}^{st} \rho_{\tau}^{k+1} \, \dx \, \mathrm{d}t \\
        & \qquad = \me[\rho_{\tau}^k] - \me[\rho_{\tau}^{k+1}] + \chi\tau\limsup_{s \downarrow 0} \left( \me[\rho_{\tau}^{k+1}] - \me[S_{\me}^{s} \rho_{\tau}^{k+1}] \right),
        \end{align*}
where we recognised the third term as the Fisher information functional for solutions of the heat equation. Next, using well-known properties of the heat equation and the estimates in Lemma \ref{lem: interpolation narrow convergence} we have
\begin{equation*}
   \limsup_{s \downarrow 0} \left( \me[\rho_{\tau}^{k+1}] - \me[S_{\me}^{s} \rho_{\tau}^{k+1}] \right)\leq C,
    \end{equation*}
    for a constant $C$ independent of $k$.
    By summing up over $k$ from $0$ to $N-1$, and using \eqref{HE: hessian laplacian} and \eqref{eq:Hessian and Laplacian} again we obtain
    \begin{align*}
        \|  D^2 \rho_{\tau} \|_{L^2([0,T]; L^2(\Rd))}^2  \leq \me[\rho_0] - \me[\rho_{\tau}^N] + \tau N C,
    \end{align*}
    and in particular, $\Delta \rho_{\tau}$ is uniformly bounded in $L^2([0,T]; L^2(\Rd))$.
\end{proof}

\begin{prop}[Strong convergence of $\nabla\rhot$]\label{prop: interpolation strong convergence gradient} Let $\rho_0$ such that $\mf_m[\rho_0]<+\infty$, and $1 \leq m<2+\frac{2}{d}$ or $m = 2+\frac{2}{d}$ with subcritical mass $\chi<\chi_c$. Up to a subsequence, the sequence $\rho_\tau:[0,T]\rightarrow\mathcal{P}_2(\Rd)$ converges strongly to the curve $\Tilde{\rho}$ in $L^2([0,T]; H^1(\Rd))$. 
\end{prop}
\begin{proof}
First note that due to~\Cref{H2 bound flow interchange}, $ D^2\rhot\rightharpoonup D^2\Tilde{\rho}$ in $L^2([0,T];L^2(\Rd))$. The limit can be uniquely identified by integrating against a smooth and compactly supported test function and using the convergence $\nabla\rhot\rightharpoonup\nabla\Tilde{\rho}$ in $L^2([0,T];L^2(\Rd))$, cf.~\Cref{prop: weak convergence interpolation}.
Next, we claim strong convergence of $\rhot$ in $L^2([0,T];H^1(\Rd))$ follows from the strong convergence in $L^2([0,T];L^2(\Rd))$, cf.~\Cref{prop: strong convergence rho}, and the fact that $\|\rhot\|_{L^2([0,T];H^2(\Rd))}$ is uniformly bounded in $\tau$, as given in~\Cref{H2 bound flow interchange}. More precisely, using Gagliardo--Nirenberg (for the gradient) and Cauchy--Schwarz inequalities, we obtain
\begin{align*}
\int_0^T&\Vert\nabla\rhot(t)-\nabla\Tilde{\rho}(t)\Vert^2_{L^2(\Rd)}\,\mathrm{d}t \\
    &\leq C\int_0^T\Vert D^2\rhot(t)- D^2\Tilde{\rho}(t)\Vert_{L^2(\Rd)}\Vert\rhot(t)-\Tilde{\rho}(t)\Vert_{L^2(\Rd)}\,\mathrm{d}t\\&\leq C\Vert D^2\rhot- D^2\Tilde{\rho}\Vert_{L^2([0,T];L^2(\Rd))}\Vert\rhot-\Tilde{\rho}\Vert_{L^2([0,T];L^2(\Rd))}.
\end{align*}
The result is obtained by using that the norms $\Vert D^2\rhot\Vert_{L^2([0,T];L^2(\Rd))}$ and $\Vert D^2\tilde\rho\Vert_{L^2([0,T];L^2(\Rd))}$ are uniformly bounded in $\tau$ --- \Cref{H2 bound flow interchange}.
\end{proof}

The strong convergence of $\nabla\rhot$ allows us to improve the result of $\rhot$ given by Proposition~\ref{prop: strong convergence rho} via interpolation inequalities. In particular, we obtain the integrability exponent needed to pass to the limit $\tau\rightarrow 0$ in the weak formulation.

\begin{cor}[Higher integrability]\label{cor:higher integrability}
Assume $1<m < 2+\frac{2}{d}$ or $m = 2+\frac{2}{d}$ with subcritical mass $\chi<\chi_c$. Then, the sequence $\rho_\tau:[0,T]\rightarrow\mathcal{P}_2(\Rd)$ converges strongly, up to subsequence, to the curve $\Tilde{\rho}$ 
in $L^m([0,T]; L^m(\Rd))$
for every $T > 0$.
\end{cor}
\begin{proof} The proof is based on that of \Cref{prop:Some properties}. For $1<m<2+\frac{2}{d}$, by applying Gagliardo--Nirenberg and H\"{o}lder inequalities we obtain
\begin{align}
    \int_0^T&\Vert\rhot(t)-\Tilde{\rho}(t)\Vert_{L^m(\Rd)}^m\,\mathrm{d}t \nonumber\\
    &\leq C \int_0^T\Vert\nabla\rhot(t)-\nabla\Tilde{\rho}(t)\Vert_{L^2(\Rd)}^{m\theta}\Vert\rhot(t)-\Tilde{\rho}(t)\Vert_{L^1(\Rd)}^{m(1-\theta)}\,\mathrm{d}t\label{eq: GN higher integrability}\\&\leq C\Vert\nabla\rhot-\nabla\Tilde{\rho}\Vert^{m\theta}_{L^2([0,T];L^2(\Rd))}\left(\int_0^T\Vert\rhot(t)-\Tilde{\rho}(t)\Vert^\alpha_{L^1(\Rd)}\,\mathrm{d}t\right)^{\frac{m(1-\theta)}{\alpha}}\nonumber,
\end{align}
where $\theta =\frac{2d}{d+2}\frac{m-1}{m}\in(0,1)$ and $\alpha = 1+\frac{\frac{2}{d}(m-1)}{2+ \frac{2}{d}-m}$. The result follows from the strong convergence of $\nabla\rhot$ and by noting that the second term is uniformly bounded in $\tau$ due to the narrow convergence of $\rhot$ given in \Cref{lem: interpolation narrow convergence}, being $\rho_t$ and $\rho$ probability densities.

In the critical case $m = m_c = 2 + \frac{2}{d}$, \eqref{eq: GN higher integrability} gives 
\begin{align*}
    \int_0^T&\Vert\rhot(t)-\Tilde{\rho}(t)\Vert_{L^m(\Rd)}^m\,\mathrm{d}t \\
    &\leq C \int_0^T\Vert\nabla\rhot(t)-\nabla\Tilde{\rho}(t)\Vert_{L^2(\Rd)}^{2}\Vert\rhot(t)-\Tilde{\rho}(t)\Vert_{L^1(\Rd)}^{\tfrac{2}{d}}\,\mathrm{d}t\,,
    \\
    &\leq C\Vert\nabla\rhot-\nabla\Tilde{\rho}\Vert^{2}_{L^2([0,T];L^2(\Rd))}\Vert\rhot(t)-\Tilde{\rho}(t)\Vert^{\tfrac{2}{d}}_{L^\infty([0,T];L^1(\Rd))},
\end{align*}
where the second term is uniformly bounded in $\tau$ by \Cref{lem: interpolation narrow convergence} and \Cref{prop: weak convergence interpolation}. Again, the result follows from the strong convergence of $\nabla\rhot$.
\end{proof}

\subsection{Consistency of the scheme}

The results from the previous subsection ensure we can prove that $\Tilde{\rho}$ is a weak solution of~\eqref{eq:thin_film_intro} in the sense of \Cref{def:weak solution}. This subsection completes the proof of Theorem \ref{thm:main_result_existence}.

\begin{proof}[Proof of Theorem~\ref{thm:main_result_existence}] We prove the theorem by showing that the sequence $\rho_\tau:[0,T]\rightarrow\mathcal{P}_2(\Rd)$ converges, up to a subsequence, to a weak solution $\Tilde{\rho}$ of \eqref{eq:thin_film_intro}. Let us focus on two consecutive steps in the JKO scheme, $\rhot^{k}$ and $\rhot^{k+1}$, and consider the perturbation $\rho^\varepsilon = P^\varepsilon_\#\rhot^{k+1}$ given by $P^\varepsilon = \text{id} + \varepsilon\zeta$, where $\zeta$ is a vector field $\zeta\in C_c^\infty(\Rd;\Rd)$ and $\varepsilon\geq 0$. From the definition of the scheme we have
\begin{equation}
\label{eq: consistency 1}
    \frac{1}{2\tau}\left(\frac{\mw_2^2(\rhot^{k},\rho^\varepsilon)-\mw_2^2(\rhot^{k},\rhot^{k+1})}{\varepsilon}\right) + \frac{\mf_m[\rho^\varepsilon] - \mf_m[\rhot^{k+1}]}{\varepsilon}\geq 0. 
\end{equation}
As we want to let $\varepsilon\to0$ and recover the Euler--Lagrange equation of the minimisation problem~\eqref{eq:JKO}, we examine each term in \eqref{eq: consistency 1}.

\textbf{Step 1:} Wasserstein distance terms. We consider, in view of Brenier's Theorem, the optimal map $\mathcal{T}$ between $\rhot^{k}$ and $\rhot^{k+1}$ (see, e.g.,~\cite{Vil03, Vil09, San15}), so that
\begin{equation*}
    \mw_2^2(\rhot^k,\rhot^{k+1}) = \int_\Rd |x-\mathcal{T}(x)|^2\rhot^k(x)\,\dx\,.
\end{equation*}
Moreover, from the definition of the Wasserstein distance, we also have
\begin{align*}
     \mw_2^2(\rhot^k,\rho^\varepsilon) &\leq \int_\Rd |x-P^\varepsilon(\mathcal{T}(x))|^2\rhot^k(x)\,\dx \\
     & =\int_\Rd |x-\mathcal{T}(x) - \varepsilon\zeta(\mathcal{T}(x))|^2\rhot^k(x)\,\dx \\
     & =  \mw_2^2(\rhot^k,\rhot^{k+1}) - 2\varepsilon\int_\Rd (x-\mathcal{T}(x))\cdot\zeta(\mathcal{T}(x))\rhot^k(x)\,\dx + O(\varepsilon^2)\,.
\end{align*}
Consequently, 
\begin{equation}\label{eq: consistency wasserstein}
\begin{split}
 \frac{\mw_2^2(\rhot^{k},\rho^\varepsilon)-\mw_2^2(\rhot^{k},\rhot^{k+1})}{2\tau\varepsilon}
 \leq-\frac{1}{\tau}\int_\Rd (x-\mathcal{T}(x))\cdot\zeta(\mathcal{T}(x))\rhot^k(x)\,\dx + O(\varepsilon)\,.
\end{split}
\end{equation}

\textbf{Step 2:} Aggregation terms. We use the area formula \cite[Section 5.5]{AGS} and that $\det\nabla P^\varepsilon(x) = 1 + \varepsilon\dive\,\zeta(x) + O(\varepsilon^2)$. For the case $1<m < 2+\frac{2}{d}$ or $m = 2+\frac{2}{d}$ with subcritical mass, we obtain
\begin{align*}
    \int_\Rd (\rho^\varepsilon)^m\,\dx&=\int_\Rd \left(\frac{\rhot^{k+1}}{\det\nabla P^\varepsilon}\right)^m\det\nabla P^\varepsilon\,\dx\\&=\int_\Rd(\rhot^{k+1})^m\left(1 - \varepsilon(m-1)(\dive\,\zeta)+O(\varepsilon^2)\right)\,\dx\,.
\end{align*}
Thus, we find
\begin{align*}
    -\frac{1}{m-1}\int_\Rd \frac{(\rho^\varepsilon)^m-(\rhot^{k+1})^m}{\varepsilon}\,\dx =\int_\Rd (\rhot^{k+1})^m(\dive\,\zeta)\,\dx + O(\varepsilon)\,. 
\end{align*}
For the case $m=1$ we have that
\begin{align*}
    \int_\Rd \rho^\varepsilon \log (\rho^{\varepsilon} ) \, \dx & = \int_{\Rd} \frac{\rhot^{k+1}}{\det\nabla P^\varepsilon} \log \left(\frac{\rhot^{k+1}}{\det\nabla P^\varepsilon} \right) \det\nabla P^\varepsilon \, \dx \\
    & = \int_{\Rd} \rhot^{k+1} \log \rhot^{k+1} - \rhot^{k+1} \log \left( 1 + \varepsilon\dive\,\zeta + O(\varepsilon^2) \right) \, \dx.
\end{align*}
Therefore,
\begin{equation*}
    - \int_{\Rd} \frac{\rho^{\varepsilon} \log \rho^{\varepsilon} - \rhot^{k+1} \log \rhot^{k+1}}{\varepsilon} \, \dx = \int_{\Rd} \frac{\rhot^{k+1} \log \left( 1 + \varepsilon\dive\,\zeta + O(\varepsilon^2) \right)}{\varepsilon} \, \dx ,
\end{equation*}
and taking the limit in $\varepsilon$ we obtain,
\begin{equation*}
   - \lim_{\varepsilon \rightarrow 0}\int_{\Rd} \frac{\rho^{\varepsilon} \log \rho^{\varepsilon} - \rhot^{k+1} \log \rhot^{k+1}}{\varepsilon} \, \dx = \int_{\Rd} \rhot^{k+1} (\dive\,\zeta)  \, \dx .
\end{equation*}
In particular,
\begin{equation}\label{eq: consistency aggregation}
  -  \lim_{\varepsilon \rightarrow 0}  \frac{\me_m [\rho^{\varepsilon}]-\me_m [ \rhot^{k+1}]}{\varepsilon} = \int_\Rd (\rhot^{k+1})^m(\dive\,\zeta)\,\dx,
\end{equation}
holds for every $1 \leq m < 2+\frac{2}{d}$ or $m = 2+\frac{2}{d}$ with subcritical mass.

\textbf{Step 3:} Diffusion terms. We use the definition of push-forward and the area formula to obtain
\begin{align*}
    & \int_\Rd |\nabla P_\#\rhot^{k+1}(x)|^2\,\dx  = \int_\Rd \left|\nabla \left(\frac{\rhot^{k+1}}{\det \nabla P^\varepsilon}\circ (P^{\varepsilon})^{-1}\right)(x)\right|^2\,\dx 
    \\
    & \qquad =
    \int_\Rd \left|\nabla(P^\varepsilon)^{-1}(x)\nabla \left(\frac{\rhot^{k+1}}{\det \nabla P^\varepsilon}\right)\left((P^\varepsilon)^{-1}(x)\right)\right|^2\,\dx
    \\
    & \qquad =
    \int_\Rd \left|\nabla(P^\varepsilon)^{-1}(P^\varepsilon(x))\nabla \left(\frac{\rhot^{k+1}}{\det \nabla P^\varepsilon}\right)(x)\right|^2\left|\det\nabla P^\varepsilon(x)\right|\,\dx
    \\
    & \qquad =
    \int_\Rd \left|(\nabla P^\varepsilon(x))^{-1}\nabla \left(\frac{\rhot^{k+1}}{\det \nabla P^\varepsilon}\right)(x)\right|^2\left|\det\nabla P^\varepsilon(x)\right|\,\dx\,.
\end{align*}
Next, we observe that $(\nabla P^\varepsilon)^{-1} = \mathrm{I}_d - \varepsilon\nabla\zeta + O(\varepsilon^2)$, with $\mathrm{I}_d$ the identity matrix. Hence, we have
\begin{align*}
    & \int_\Rd |\nabla P_\#\rhot^{k+1}|^2\,\dx \\
    & = \int_\Rd \left|\nabla\rhot^{k+1} \!-\! \varepsilon(\rhot^{k+1}\nabla(\dive\,\zeta)+\nabla\zeta\nabla\rhot^{k+1}+\frac{1}{2}(\dive\,\zeta)\nabla\rhot^{k+1}) \right|^2\dx + O(\varepsilon^2),
\end{align*}
and, in particular,
\begin{align}\label{eq: consistency diffusion}
\begin{split}
   & \frac{1}{2}\int_\Rd\frac{|\nabla P_\#\rhot^{k+1}|^2-|\nabla\rhot^{k+1}|^2}{\varepsilon}\,\dx \qquad \\
   &= -\int_\Rd\!\!\left(\rhotkk\nabla(\dive\,\zeta)\!\cdot\!\nabla\rhot^{k+1}\! +\! (\nabla\zeta\nabla\rhot^{k+1})\!\cdot\!\nabla\rhot^{k+1}\!+\!\frac{1}{2}\dive\,\zeta|\nabla\rhot^{k+1}|^2\right)\!\dx\\ 
   &\quad+ O(\varepsilon).
\end{split}
\end{align}
\textbf{Step 4:} Letting $\varepsilon \rightarrow 0$. Let us perform again the same computation for $\varepsilon \leq 0$. Then, we consider $\zeta = \nabla \varphi$ and compute the limit $\varepsilon \rightarrow 0$. By taking into account~\eqref{eq: consistency wasserstein},~\eqref{eq: consistency aggregation}, and~\eqref{eq: consistency diffusion}, we have that,
\begin{align}\label{eq:weak formulation version 1}
\begin{split}
    & \frac{1}{\tau} \int_\Rd (x-\mathcal{T}(x))\cdot\nabla \varphi (\mathcal{T}(x))\rhot^k(x)\,\dx \\
    & =-\int_\Rd\!\!\left(\rhotkk\nabla(\Delta \varphi )\!\cdot\!\nabla \rhot^{k+1}\! +\! ( D^2\varphi \nabla\rhot^{k+1})\!\cdot\!\nabla\rhot^{k+1}\!+\! \frac{1}{2}\Delta \varphi |\nabla\rhot^{k+1}|^2\right)\!\dx \\
    & \quad + \chi\int_\Rd (\rhot^{k+1})^m\Delta \varphi \,\dx\,.
\end{split}
\end{align}
Next, we rewrite the left-hand side of \eqref{eq:weak formulation version 1} by considering a Taylor expansion of $\varphi$ on $\mathcal{T} (x)$. Since $\rhot$ is Holder continuous, \eqref{eq:equicontinuity on d2}, we have
\begin{equation*}
    \int_\Rd (x-\mathcal{T}(x))\cdot\nabla \varphi (\mathcal{T}(x))\rhot^k(x)\,\dx = \int_{\Rd} \varphi (x) \left[ \rhot^k(x) - \rhot^{k+1}(x) \right] \, \dx + O(\tau )\,.
\end{equation*}
Let $0 \leq s_1 < s_2 \leq T$ be fixed with,
\begin{equation*}
    h_1 = \left[ \frac{s_1}{\tau} \right] +1 \quad \text{and} \quad h_2 = \left[ \frac{s_2}{\tau} \right].
\end{equation*}
By summing with respect to $k$ in \eqref{eq:weak formulation version 1}, we obtain,
\begin{align*}
    &\int_{\Rd} \varphi (x) \rhot^{h_2+1}(x) \, \dx - \int_{\Rd} \varphi (x) \rhot^{h_1}(x) \, \dx +O(\tau) \\
    & = \sum_{j=h_1}^{h_2} \tau \int_\Rd\!\! \left(\rho_\tau^{j+1}\nabla(\Delta \varphi )\!\cdot\!\nabla\rhot^{j+1} \!+\! ( D^2\varphi \nabla\rhot^{j+1})\!\cdot\!\nabla\rhot^{j+1}\!+\! \frac{1}{2}\Delta \varphi |\nabla\rhot^{j+1}|^2\right)\!\dx \\
    & \quad - \chi\sum_{j=h_1}^{h_2} \tau \int_\Rd (\rhot^{j+1})^m\Delta \varphi \,\dx\,.
\end{align*}
Using the definition of the piecewise constant interpolation $\rho_{\tau}$ and integration by parts, cf.~Remark~\ref{rem:integration_part_weak_form}, this is equivalent to
\begin{align}\label{eq:Consistency tau limit}
    \begin{split}
        &\int_{\Rd} \varphi (x) \rhot (s_2, x) \, \dx - \int_{\Rd} \varphi (x) \rhot (s_1, x) \, \dx + O(\tau)
        \\
        & = \int_{s_1}^{s_2} \!\!\int_\Rd\!\left(\rho_\tau\nabla(\Delta \varphi)\!\cdot\!\nabla \rhot \!+\! ( D^2\varphi \nabla\rhot)\!\cdot\!\nabla\rhot\!+\!\frac{1}{2}\Delta \varphi |\nabla\rhot|^2\right)\dx\,\dt 
        \\
        & \quad - \chi\int_{s_1}^{s_2} \int_\Rd \rhot^m\Delta \varphi \,\dx \, \dt\\
        &=\!-\!\int_{s_1}^{s_2}\!\!\!\int_\Rd\left(\rho_\tau\Delta\rho_\tau\Delta\varphi+\Delta\rho_\tau\nabla\rho_\tau\cdot\nabla\varphi\right)\dx  \,\dt \!-\chi\! \int_{s_1}^{s_2}\!\! \int_\Rd \rhot^m\Delta \varphi \,\dx \, \dt.
    \end{split}
\end{align}
By combining~\Cref{H2 bound flow interchange},~\Cref{prop: strong convergence rho},~\Cref{prop: interpolation strong convergence gradient}, and \Cref{cor:higher integrability} we can pass to the limit in \eqref{eq:Consistency tau limit} as $\tau \rightarrow 0^+$, and recover a weak solution. 
\end{proof}

\begin{rem}\label{rem:integration_part_weak_form}
Assume $\rho\in H^2(\Rd)$ and $\varphi\in C_0^3(\Rd)$ --- this is indeed not a restriction as $\zeta\in C^\infty_c(\Rd;\Rd)$. Using integration by parts several times, we have
    \begin{align*}
       &\int_\Rd\left(\rho\nabla\rho\cdot\nabla(\Delta\varphi)+ \nabla\rho\cdot( D^2\varphi\nabla\rho)+\frac{1}{2}\Delta\varphi|\nabla\rho|^2\right)\,\dx
       \\
       & = - \int_\Rd\rho\Delta\rho\Delta\varphi\,\dx + \frac{1}{2}\int_\Rd\left( 2\nabla\rho\cdot( D^2\varphi\nabla\rho)-\Delta\varphi|\nabla\rho|^2\right)\,\dx
       \\
       & =- \int_\Rd\rho\Delta\rho\Delta\varphi\,\dx + \int_\Rd \left(\nabla\rho\cdot( D^2\varphi\nabla\rho)+\nabla\varphi\cdot( D^2\rho\nabla\rho)\right)\,\dx
       \\
       & =- \int_\Rd\left(\rho\Delta\rho\Delta\varphi+\Delta\rho\nabla\rho\cdot\nabla\varphi\right)\,\dx\,.
    \end{align*}
\end{rem}

\begin{rem}\label{rem:addition_external_potential}
	We observe that the addition of an external potential to the energy $\mf_m$, thus to~\eqref{eq:thin_film_intro}, even nonlocal, does not bring further difficulties to our strategy under minimal regularity assumptions. Indeed, the above proof can be integrated with previous results, e.g.~\cite{JKO98,MMS09}.
\end{rem}

\section{Extension to systems of two interacting species}\label{sec:existence_two_species}

In this section, we extend the one-species theory to study system~\eqref{eq:two species} and prove existence of weak solutions. First, we obtain some basic properties of the free energy functional, defined in~\eqref{eq:functiona_two_species}, we recall here for the reader's convenience:
\begin{equation*}
    \bm{\mf}[\rho,\eta] = \begin{cases}
      \Tilde{\bm{\mf}}[\rho,\eta] & \mbox{if } (\rho,\eta)\in\mathcal{P}^a(\mathbb{R}^d)^2,\, (\nabla\rho,\nabla\eta) \in L^2(\mathbb{R}^d)^2,
        \\
        +\infty & \mbox{otherwise,} 
    \end{cases}
\end{equation*}
being
\begin{equation*}
    \Tilde{\bm{\mf}}[\rho,\eta] = \int_{\Rd}\left(\frac{\kappa}{2}|\nabla\rho|^2+\frac{1}{2}|\nabla\eta|^2 + \alpha\nabla\rho\cdot\nabla\eta-\frac{\beta}{2}\rho^2-\frac{1}{2}\eta^2-\omega\rho\eta\right)\,\dx.
\end{equation*}
We remind the reader the parameters in the model are such that $\beta,\omega\in\mathbb{R}$ and the matrix 
\begin{equation*}
   A = \begin{pmatrix}
    \kappa & \alpha
    \\
    \alpha & 1
    \end{pmatrix}
\end{equation*}
is assumed to be positive definite.
\begin{prop}[Lower bound for the free energy and induced regularity]\label{prop: two species bound from below}
     Assume $(\rho,\eta) \in \mathcal{P}^a(\Rd)^2$. The following properties hold.
    \begin{enumerate}[label=$(\arabic*)$]
        \item \underline{Lower bound for the free energy}: let $\nabla \rho,\nabla\eta \in L^2 (\mathbb{R}^d)$, then $\bm{\mf}[\rho,\eta]$ is bounded from below as
        \begin{equation}\label{eq: energy bounded below two species}
         \bm{\mf}[\rho,\eta] \geq - C \left(\| \rho \|_{L^1(\mathbb{R}^d)}^{2}+\| \eta \|_{L^1(\mathbb{R}^d)}^{2}\right),
        \end{equation}
        where $C = C(\kappa,\alpha,\beta,\omega,d)> 0$.\\
        \item \underline{$H^1$-bound}: assume $\bm{\mf}[\rho,\eta] < +\infty$, then the following bound holds
        \begin{equation}\label{eq:grad is L2 two species}
            \| \nabla f \|^2_{L^2 (\Rd)} \leq C \left( \mf_2[f] + \| f \|_{L^1(\mathbb{R}^d)}^{2} \right),\quad\text{for } f\in\{\rho,\eta\}
        \end{equation}
        where $C=C(d)> 0$.  
        \item \underline{$L^p$-regularity}: assume $\bm{\mf}[\rho,\eta] < +\infty$, then $\rho,\eta \in L^{p} (\Rd)$ for $p\in[1,2^*]$. In particular,  there exists a constant $C = C(p,d,f)>0$ such that
         \begin{equation}\label{eq:control of the Lp norm two species}
            \| f \|_{L^p (\mathbb{R}^d)} \leq C<+\infty ,\quad\text{for } f\in\{\rho,\eta\}.
        \end{equation}
    \end{enumerate}
\end{prop}
\begin{proof}

\textbf{Step 1:} Lower bound for the free energy. By using Cauchy--Schwarz and Young inequalities we obtain
\begin{align*}
     {\bm{\mf}}[\rho,\eta] &=\int_{\Rd}\left(\frac{\kappa}{2}|\nabla\rho|^2+\frac{1}{2}|\nabla\eta|^2 + \alpha\nabla\rho\cdot\nabla\eta-\frac{\beta}{2}\rho^2-\frac{1}{2}\eta^2-\omega\rho\eta\right)\,\dx
     \\
     & \geq \!\!\int_{\Rd}\left(\frac{\kappa}{2}|\nabla\rho|^2+\frac{1}{2}|\nabla\eta|^2 - |\alpha||\nabla\rho||\nabla\eta|-\frac{\beta+|\omega|}{2}\rho^2-\frac{1 + |\omega|}{2}\eta^2\right)\dx
     \\
     & \geq\!\! \int_{\Rd}\left(\frac{\kappa - |\alpha|\varepsilon}{2}|\nabla\rho|^2\!+\!\frac{1 - |\alpha|\varepsilon^{-1}}{2}|\nabla\eta|^2 \!-\!\frac{\beta+|\omega|}{2}\rho^2\!-\!\frac{1 + |\omega|}{2}\eta^2\right)\dx.
\end{align*}
Since the matrix $A$ is positive definite, $\kappa - \alpha^2>0$ we can choose $\varepsilon\in (|\alpha| , \frac{\kappa}{|\alpha|})$ so that $1-|\alpha|\varepsilon^{-1}>0$ and $\kappa - |\alpha|\varepsilon>0$. Hence, we obtain
\begin{align}\label{eq: bound two one species}
     {\bm{\mf}}[\rho,\eta] & \geq C\left(\mf_2[\rho] + \mf_2[\eta]\right),
\end{align}
for $C = C(\kappa,\alpha,\beta,\omega)>0$ and the result follows from the one-species case \eqref{eq:Free energy bounded below}.

\textbf{Step 2:} $H^1$-bound and $L^p$-regularity. Given $\bm{\mf}[\rho,\eta] < +\infty$, then  \eqref{eq: bound two one species} implies $\mf_2[\rho],\mf_2[\eta]<+\infty$. The results follow from the one-species case \eqref{eq:grad rho is L2}, \eqref{eq:control of the Lp norm}.
\end{proof}

\subsection{The JKO scheme} 

Analogously to the problem for the one-species case, we can use the JKO scheme to construct an approximation to a candidate of a solution.

\begin{rem}
For the sake of completeness we specify the notation for the $2$-Wasserstein distance in the product space. Let $\sigma_1 = (\rho_1,\eta_1)\in\mathcal{P}_2(\Rd)^2$ and $\sigma_2 = (\rho_2,\eta_2)\in\mathcal{P}_2(\Rd)^2$. The $2$-Wasserstein distance between $\sigma_1$ and $\sigma_2$ is denoted as
    \begin{equation}
        \dW^2(\sigma_1,\sigma_2) = \mw_2^2(\rho_1,\rho_2) + \mw_2^2(\eta_1,\eta_2)\,.
        \label{eq: distance two species}
    \end{equation}
Furthermore, note that for $\sigma = (\rho,\eta)\in\mathcal{P}_2(\Rd)^2$, $\mathrm{m}_2(\sigma)=\mathrm{m}_2(\rho)+\mathrm{m}_2(\eta)$.
\end{rem}

As in the one-species case, we consider the following recursive scheme, for $\sigma_0\in\mP_2(\Rd)^2$.  
    \begin{itemize}
        \item Let $\tau>0$ and set $\sigma_\tau^0 := \sigma_0 = (\rho_0,\eta_0)$.
        \item Given $\sigma_\tau^k = (\rho_\tau^k,\eta_\tau^k)\in\mP(\Rd)^2$ for $k\geq0$, choose
        \begin{equation}\label{eq:JKO two species}
            \sigma_\tau^{k+1} =(\rho_\tau^{k+1},\eta_\tau^{k+1})\in\argmin_{\sigma\in \mathcal{P}(\Rd)^2} \left\lbrace\frac{\dW^2(\sigma,\sigma_\tau^k)}{2\tau} + \bm{\mf}[\sigma] \right\rbrace\,.
        \end{equation}
    \end{itemize}

We start checking that the scheme~\eqref{eq:JKO two species} is well-defined. Let us fix $\bar{\sigma} = (\bar{\rho}, \bar{\eta})\in\mP_2^a(\Rd)^2$ and define the functional 
    \begin{align*}
	\begin{split}
		\bm{\JKOstep}\colon& \mathcal{P} (\mathbb{R}^d)^2 \longrightarrow  \overline{\mathbb{R}} \\
		& \quad \, \,\sigma \quad  \longmapsto \, \, \frac{\dW^2(\sigma, \overline{\sigma})}{2 \tau} + \bm{\mf}[\sigma].
	\end{split}
    \end{align*}

    \begin{prop}\label{prop: minimisers two species}
        Let $\overline{\sigma} \in \mathcal{P}_2^a(\Rd)^2$. The functional $\bm{\JKOstep}$ admits a minimiser in the set $\left\lbrace \sigma = (\rho, \eta) \in \mathcal{P}^a(\Rd)^2  \, : \, \nabla \rho, \nabla \eta \in L^2(\Rd) \right\rbrace$.
    \end{prop}

    Again, we employ the direct method of calculus of variations and the results from the one-species case, cf. \Cref{prop:Existence minimiser}. 

    \begin{proof}
        \textbf{Step 1:} Boundedness from below. Analogously to \Cref{prop:Existence minimiser} we note that 
        \begin{equation*}
            \bm{\JKOstep} [\sigma] \geq C.
        \end{equation*}
        This ensures that we can consider a minimising sequence $\{\sigma_n\}_n$, where $\sigma_n = (\rho_n,\eta_n)$,   satisfying:
        \begin{equation*}
            \mathrm{m}_2 (\rho_n) + \mathrm{m}_2 (\eta_n) \leq CT(1 + \mathrm{m}_2 (\overline{\rho}) + \mathrm{m}_2 (\overline{\eta})).
        \end{equation*}

        \textbf{Step 2:} $\bm{\JKOstep}$ is lower semicontinuous. Repeating the argument in ~\Cref{prop:Existence minimiser} we know that, up to a subsequence,
        \begin{subequations}\label{eq:lsc convergence quadratic terms}
        \begin{align}
                \nabla \rho_n \rightharpoonup \nabla \rho & \quad \text{and} \quad \nabla \eta_n \rightharpoonup \nabla \eta  &&  \text{in } L^2(\Rd), &&&\\
                \rho_n \rightarrow \rho & \quad \text{and} \quad \eta_n \rightarrow \eta &&  \text{in } L^2(\Rd).&&&
        \end{align}
        \end{subequations}
        Next, we write
        \begin{equation*}
            \nabla\rho_n\cdot\nabla\eta_n = \frac{\alpha}{2}\left|\nabla(\rho_n + \alpha^{-1}\eta_n)\right|^2 - \frac{\alpha}{2}|\nabla\rho_n|^2- \frac{1}{2\alpha}|\nabla\eta_n|^2.
        \end{equation*}
        Note that $\rho_n+\alpha^{-1}\eta_n\rightarrow\rho + \alpha^{-1}\eta$ and also $\nabla(\rho_n+\alpha^{-1}\eta_n)\rightharpoonup \nabla\rho + \alpha^{-1}\nabla\eta$ in $L^2(\Rd)$. By using the lower semicontinuity of the $H^1$ seminorm and that $\kappa-\alpha^2>0$, we obtain
        \begin{align*}
            &\liminf_{n\rightarrow+\infty}\int_\Rd\left(\frac{\kappa}{2}|\nabla\rho_n|^2+\frac{1}{2}|\nabla\eta_n|^2 + \alpha\nabla\rho_n\cdot\nabla\eta_n\right)\,\dx 
            \\ 
            & = \liminf_{n\rightarrow+\infty}\int_\Rd\left(\frac{\kappa-\alpha^2}{2}|\nabla\rho_n|^2
            +\frac{\alpha^2}{2}\left|\nabla(\rho_n + \alpha^{-1}\eta_n)\right|^2\right)\,\dx 
            \\
            & \geq 
            \int_\Rd\left(\frac{\kappa-\alpha^2}{2}|\nabla\rho|^2
            +\frac{\alpha^2}{2}\left|\nabla(\rho + \alpha^{-1}\eta)\right|^2\right)\,\dx 
            \\
            & =
            \int_\Rd\left(\frac{\kappa}{2}|\nabla\rho|^2+\frac{1}{2}|\nabla\eta|^2 + \alpha\nabla\rho\cdot\nabla\eta\right)\,\dx\,.
        \end{align*}
        In order to deal with the other terms involved in the free energy, the quadratic terms follow from the convergence \eqref{eq:lsc convergence quadratic terms}. In order to deal with the last term, we now claim that
        \begin{equation*}
            \rho_n \eta_n \rightarrow \rho \eta \quad \text{in } L^1(\Rd).
        \end{equation*}
        This follows from
        \begin{align*}
            \| \rho_n \eta_n - \rho \eta \|_{L^1(\Rd)} & \leq \| \eta_n (\rho - \rho_n) \|_{L^1(\Rd)} + \| \rho (\eta - \eta_n) \|_{L^1(\Rd)} \\
            & \leq \| \eta_n \|_{L^2(\Rd)} \| \rho - \rho_n \|_{L^2(\Rd)} + \|\rho \|_{L^2(\Rd)} \|\eta- \eta_n \|_{L^2(\Rd)} \\
            & \rightarrow 0\,.
        \end{align*}

          \textbf{Step 3:} Existence of minimisers follows then from the Weierstrass criterion, cf.~e.g.~\cite[Box 1.1]{San15}.
       
    \end{proof}

   Let $T>0$, and consider $N:=\left[\frac{T}{\tau}\right]$. We define the curve $\sigma_\tau: [0, T ] \rightarrow \mathcal{P} (\mathbb{R}^d)^2$ as the piecewise constant interpolation
\begin{equation}\label{eq:piecewise interpolation two species}
    \sigma_\tau(t) = \sigma_\tau^k,\quad t\in \left((k-1) \tau , k \tau \right],
\end{equation}
where $\sigma_\tau^k = (\rho_\tau^k,\eta_\tau^k)$ is defined in $\eqref{eq:JKO two species}$. In the following, we prove the two-species analogous of \Cref{lem: interpolation narrow convergence}, \Cref{prop: weak convergence interpolation}, and \Cref{prop: strong convergence rho}.

\begin{lem}[Narrow convergence \&  discrete uniform estimates]\label{lem: interpolation narrow convergence two species}
  Let $\sigma_0\in\mathcal{P}_2^a(\Rd)^2$ such that $\bm{\mf}[\sigma_0]<+\infty$. There exists an absolutely continuous curve $\tilde{\sigma} : [0,T] \rightarrow \mathcal{P}_2 (\mathbb{R}^d)^2$
    such that, up to a subsequence, $\sigma_\tau(t)$ narrowly converges to $\Tilde{\sigma}(t)$, uniformly in $t \in [0,T]$.

\noindent Moreover, we obtain the following discrete uniform bounds:
    \begin{align}
& \sup_k\Vert\nabla\rho_\tau^k\Vert_{L^2(\Rd)}+\sup_k\Vert\nabla\eta_\tau^k\Vert_{L^2(\Rd)} \leq C_1 < +\infty\,,
        \label{eq: discrete H1 two species}
\\
& \sup_k\Vert\rho_\tau^k\Vert_{L^p(\Rd)} + \sup_k\Vert\eta_\tau^k\Vert_{L^p(\Rd)} \leq C_2 < +\infty\,,
 \label{eq: discrete Lp two species}
\\
 &       \mathrm{m_2}(\sigmat (t))\leq 2\mathrm{m_2}(\sigma_0) + 4T\left(\bm\mf[\sigma_0]+C\right),
         \label{eq:uniform second order moment bound two species}
    \end{align}
     for $p\in[1,2^*]$ and constants $C_1>0$ and $C_2>0$, independent of $k$ and $\tau$.
\end{lem}
\begin{proof}
    The proof works analogously to the one from \Cref{lem: interpolation narrow convergence}. By construction of the sequence we obtain that,
    \begin{equation}\label{eq:bound free energy JKO two species}
        \bm\mf [\sigmat^k] \leq \frac{d_W^2(\sigmatk, \sigmat^{k-1})}{2 \tau} + \bm\mf[\sigmatk] \leq \bm\mf[\sigmat^{k-1}],
    \end{equation}
    and, in particular,
    \begin{equation*}
        \sup_k \bm\mf [\sigmatk] \leq \bm\mf [\sigma_0] < + \infty.
    \end{equation*}
    This combined with \eqref{eq:grad is L2 two species} and \eqref{eq:control of the Lp norm two species} implies that $\| \nabla \rhotk \|_{L^2(\Rd)}$, $\| \nabla \etatk \|_{L^2(\Rd)}$, and $\| \rhotk \|_{L^p(\Rd)}$, $\|\etatk \|_{L^p (\Rd)}$ are uniformly bounded in $k$ and $\tau$ for $p \in [1, 2^{\ast}]$. From here we recover \eqref{eq: discrete H1 two species} and \eqref{eq: discrete Lp two species}.

    Summing up over $k$ in \eqref{eq:bound free energy JKO two species}, we obtain that
    \begin{equation}\label{eq:equicontinuity 1st part two species}
        \sum_{k=i+1}^j \frac{d_W^2(\sigmatk,\sigmat^{k-1})}{2 \tau} \leq \bm\mf[\sigmat^i] - \bm\mf [\sigmat^j] \leq \bm\mf [\sigma_0] + C,
    \end{equation}
    where the last inequality holds because the free energy is bounded from below from \eqref{eq: energy bounded below two species}. 
    Therefore, the distance $d_W$ between $\sigma_0$ and $\sigmat (t)$ is uniformly bounded, as for $t \in ((j-1) \tau, j \tau]$,
    \begin{equation*}
        d_W^2 (\sigma_0, \sigmat (t)) \leq j \sum_{k=1}^j d_W^2 (\sigmat^k, \sigmat^{k-1}) \leq 2 j \tau (\bm\mf[\sigma_0] + C) \leq 2T(\bm\mf [\sigma_0] + C).
    \end{equation*}
    Furthermore, this last inequality combined with the triangular inequality for the $2$-Wasserstein distance gives us that second order moments are uniformly bounded on compact time intervals $[0,T]$:
    \begin{align*}
        \mathrm{m_2}(\sigmat(t)) & \leq 2\mathrm{m_2}(\sigma_0) + 2 d_W^2(\sigma_0 , \sigmat (t))  \leq 2\mathrm{m_2}(\sigma_0) +  4T\left(\bm\mf[\sigma_0]+C\right).
    \end{align*}
    We can now prove equicontinuity. Consider $0 \leq s < t \leq T$ such that $s \in ((i-1)\tau, i \tau ]$ and $t \in ((j-1) \tau, j \tau ]$. Then, combining Cauchy--Schwarz inequality with \eqref{eq:equicontinuity 1st part two species} we have,
\begin{align}
        d_W (\sigmat (s), \sigmat (t) ) & \leq \sum_{k=i+1}^j d_W(\sigmatk,\sigmat^{k-1})
       \nonumber \\  
        &\leq \left( \sum_{k=i+1}^j d_W^2(\sigmatk,\sigmat^{k-1}) \right)^{\frac{1}{2}} |j-i |^{\frac{1}{2}} 
        \nonumber\\
        & \leq (2 (\bm\mf[\sigma_0] + C))^{\frac{1}{2}} \left( \sqrt{|t-s|} + \sqrt{\tau} \right).\label{eq:equicontinuity on d2 two species}
    \end{align}
    From here we obtain that $\sigma_{\tau}$ is $\frac{1}{2}$-Holder equi-continuous up to a negligible error of order $\sqrt{\tau}$. Thus, using the refined version of the Ascoli-Arzelà Theorem \cite[Proposition 3.3.1]{AGS}, it follows that $\sigmat$ admits a subsequence narrowly converging to a limit $\Tilde{\sigma} = (\Tilde{\rho}, \Tilde{\eta}) \in \mathcal{P}(\Rd)^2$ as $\tau \rightarrow 0^+$, uniformly on $[0,T]$. Furthermore, using that $| \cdot |^2$ is lower semicontinuous and the uniform bound \eqref{eq:uniform second order moment bound two species}, we obtain that the limiting curve $\Tilde{\sigma}$ is such that,
    \begin{equation*}
        \mathrm{m_2} (\Tilde{\sigma} (t) ) \leq \liminf_{\tau \downarrow 0} \mathrm{m_2} (\sigmat (t) )\leq C\,.\qedhere
    \end{equation*}
\end{proof}

\begin{prop}[Weak convergence]\label{prop: weak convergence interpolation two species} Let $\sigma_0\in\mathcal{P}_2^a(\Rd)^2$ such that $\bm{\mf}[\sigma_0]<+\infty$. The piecewise interpolation $\sigma_{\tau}$ constructed in \eqref{eq:piecewise interpolation two species} is such that $\sigma_\tau\in L^\infty([0,T];H^1(\Rd))^2$. In particular, the limit $\tilde{\sigma}$ belongs to $L^\infty([0,T];H^1(\Rd))^2$ and
\begin{equation*}
    \sigma_\tau\rightharpoonup \Tilde{\sigma}\quad\mbox{in } L^2([0,T];H^1(\Rd))^2.
\end{equation*}
\end{prop}
\begin{proof}
    From \eqref{eq: discrete H1 two species} in \Cref{lem: interpolation narrow convergence two species} we have
    \begin{equation*}
        \Vert\rho_\tau \Vert_{L^\infty([0,T];H^1(\Rd))} = \sup_{t\in(0,T)}\Vert\rho_\tau(t)\Vert_{L^2(\Rd)} = \sup_k\Vert\rho_\tau^k\Vert_{H^1(\Rd)}<+\infty\,,
    \end{equation*}
    and analogously for $\eta_\tau$.
In particular, for any compact time interval $[0,T]$ with $T>0$, we have $\|\rho_\tau\|_{L^2([0,T];H^1(\Rd))}+\|\eta_\tau\|_{L^2([0,T];H^1(\Rd))}\le C$ uniformly in $\tau$ and the weak convergence follows from Banach-Alaoglu Theorem. Regularity of the limit follows from standard arguments.
\end{proof}
\begin{prop}[Strong convergence of $\sigma_\tau$]\label{prop: strong convergence rho two species} Let $\sigma_0\in\mathcal{P}_2^a(\Rd)$ such that $\bm\mf[\sigma_0]<+\infty$. The sequence $\sigma_\tau:[0,T]\rightarrow\mathcal{P}_2(\Rd)^2$ converges, up to a subsequence, strongly to the curve $\Tilde{\sigma}$ 
in $L^2([0,T]; L^2(\Rd))^2$
for every $T > 0$.
\end{prop}
\begin{proof}
We apply \Cref{prop:RS03} to a subset $U = \{\sigma_\tau\}_{\tau\geq 0}$ for $X = L^2(\Rd)^2$ ang $g:=d_W$ defined in \eqref{eq: distance two species}. Similarly to the one-species case, we consider the functional $ \bm{\mathcal{I}}:L^2(\Rd)^2\rightarrow [0,+\infty]$ defined by
\begin{equation*}
    \bm{\mathcal{I}}[\rho,\eta]=\begin{cases}\Vert\rho\Vert_{H^1(\Rd)}^2\!+\!\Vert\eta\Vert_{H^1(\Rd)}^2\!+\!\mathrm{m}_2(\rho)\!+\!\mathrm{m}_2(\eta)&\rho,\eta\in\mathcal{P}_2(\mathbb{R}^d)\cap H^1(\Rd),\\
+\infty & \mbox{otherwise}.
\end{cases}
\end{equation*}
Note that $d_W$ is a distance on the proper domain of $\bm{\mathcal{I}}$. Indeed, given $\sigma = (\rho,\eta)$, if $\bm{\mathcal{I}}[\sigma]<+\infty$ then $\sigma\in\mathcal{P}_2(\mathbb{R}^d)^2$. As in \Cref{prop: strong convergence rho}, the functional $\bm{\mathcal{I}}$ is lower semicontinuous from standard arguments \cite{BE22} and has relatively compact subsets from Kolmogorov-Riesz-Fr\'echet Theorem \cite[Corollary 4.27]{Bre11}.

Proving that $\bm{\mathcal{I}}$ and $d_W$ satisfy the tightness and integral equicontinuity conditions in \Cref{prop:RS03} can be done as in the one-species case by using arguments analogous to those in \Cref{prop: strong convergence rho}. Tightness follows from the uniform-in-$\tau$ second order moment and $L^{\infty}([0,T]; H^1(\Rd ))$ bounds for $\sigmatk$ given in \Cref{lem: interpolation narrow convergence two species}. Equi-continuity is a consequence from the H\"older equi-continuity of $\sigmat$ proved in \Cref{lem: interpolation narrow convergence two species}.
\end{proof}

\subsection{Flow interchange}

As in the one-species case we can obtain $H^2$ bounds for $\rho$ and $\eta$ using the flow interchange technique. In order to do so, we consider the decoupled system of heat equations as an auxiliary flow,
\begin{align}\label{eq:Decouple heat equation}
    \begin{split}
        \begin{dcases}
            \partial_t\mu_1= \Delta \mu_1\,, \\
            \partial_t\mu_2 = \Delta \mu_2\,,
        \end{dcases}
    \end{split}
\end{align}
and the auxiliary functional,
\begin{align*}
    \bm\me [\mu_1, \mu_2] = \begin{dcases}
        \int_{\Rd} [\mu_1  \log \mu_1  + \mu_2  \log \mu_2 ] \,\dx , \quad & \mu_1 \log \mu_1, \, \mu_2 \log \mu_2 \in L^1(\Rd); \\
        + \infty & \text{otherwise}.
    \end{dcases}
\end{align*}
For any $\mu = (\mu_1, \mu_2) \in \mathcal{P}_2 (\Rd)^2$ such that $\bm\me[\mu] < \infty$, we denote by $S_{\bm\me}^t \mu := (S_{\bm\me}^t \mu_1 , S_{\bm\me}^t \mu_2)$ the solution at time $t>0$ to system \eqref{eq:Decouple heat equation} for an initial value $\mu$. Furthermore, we define the dissipation of $\bm\mf$ along the flow $S_{\bm\me}$ as
\begin{equation*}
    D_{\bm\me} \bm\mf [\sigma] := \limsup_{s \downarrow 0} \left\lbrace \frac{\bm\mf[\sigma] - \bm\mf[S_{\bm\me}^s\sigma]}{s} \right\rbrace,
\end{equation*}
where $\sigma$ denotes $\sigma := (\rho ,\eta) \in \mathcal{P}(\Rd)^2$. 

\begin{lem}[$H^2$ uniform bound]\label{H2 bound flow interchange two species}
     Let $\sigma_0$ such that $\bm{\mf}[\sigma_0]<+\infty$. The piecewise interpolation $\sigma_{\tau}$ in~\eqref{eq:piecewise interpolation two species} is such that $\sigma_\tau  \in L^2([0,T];H^2(\Rd))^2$. In particular, we obtain the uniform  bound
    \begin{equation*}
        \| D^2 \rho_{\tau} \|_{L^2 ([0,T];L^2(\Rd))}^2 + \| D^2 \eta_{\tau} \|_{L^2 ([0,T];L^2(\Rd))}^2  \leq C,
    \end{equation*}
    where $C>0$ is independent of $\tau$.
\end{lem}

\begin{proof} 
   We proceed analogously to the one-species case. Note that $\sigma_\tau\in L^2([0,T];H^1(\Rd))^2$ by \Cref{prop: weak convergence interpolation two species}. For all $s > 0$, we consider $S_{\bm\me}^s \sigma_{\tau}^{k+1} = (S_{\bm\me}^s \rho_{\tau}^{k+1},S_{\bm\me}^s \eta_{\tau}^{k+1})$.
   Then, by the definition of the scheme \eqref{eq:JKO} and of $\sigma_{\tau}^{k+1}$, we have the inequality
    \begin{equation*}
        \frac{1}{2 \tau} d_W^2 (\sigma_{\tau}^k, \sigma_{\tau}^{k+1} ) + \bm{\mf}[\sigma_{\tau}^{k+1}] \leq \frac{1}{2 \tau} d_W^2 (\sigma_{\tau}^k , S_{\bm\me}^s \sigma_{\tau}^{k+1} ) + \bm{\mf} [S_{\bm\me}^s \sigma_{\tau}^{k+1}],
    \end{equation*}
    from which we obtain,
    \begin{equation*}
        \tau \frac{\bm{\mf} [\sigma_{\tau}^{k+1}] - \bm{\mf} [S_{\bm\me}^s \sigma_{\tau}^{k+1}]}{s} \leq \frac{1}{2} \frac{d_W^2 (\sigma_{\tau}^k, S_{\bm\me}^s \sigma_{\tau}^{k+1}) - d_W^2 (\sigma_{\tau}^k, \sigma_{\tau}^{k+1})}{s}.
    \end{equation*}
    By taking the $\limsup$ as $s \downarrow 0$ and considering the definition of the distance $d_W$, we obtain
    \begin{equation}\label{eq:Apply EVI two species}
        \tau D_{\bm\me} \bm\mf[\sigma_{\tau}^{k+1}] \leq \frac{1}{2}\left.\frac{\mathrm{d}^+}{\mathrm{d}t}\right|_{t=0}  d_W^2 (\sigma_{\tau}^k , S_{\bm\me}^t \sigma_{\tau}^{k+1} )  \leq \bm\me[\sigma_{\tau}^k] - \bm\me[\sigma_{\tau}^{k+1}],
    \end{equation}
    where in the last inequality we use the \eqref{eq:EVI}, as $S_{\bm\me}$ is a $0$-flow, cf.~Definition \ref{def:lambda_flow}. 
    
    The dissipation of $\bm\mf$ along the flow $S_{\bm\me}$ can be written as
    \begin{align}\label{eq:compute DE two species}
        \begin{split}
            D_{\bm\me} \bm\mf [\sigma_{\tau}^{k+1}] & = \limsup_{s \downarrow 0} \left\lbrace \frac{\bm\mf [\sigma_{\tau}^{k+1}] - \bm\mf [S_{\bm\me}^s \sigma_{\tau}^{k+1}]}{s} \right\rbrace \\
            & = \limsup_{s \downarrow 0} \int_0^1 \left( - \left.\frac{\mathrm{d
            }}{\mathrm{d}z}\right|_{z=st} \bm\mf[S_{\bm\me}^z \sigma_{\tau}^{k+1}] \right) \, \mathrm{d}t.
        \end{split}
    \end{align}
    Let us calculate the time derivative:
     \begin{align*}
            \ddt \bm\mf[S_{\bm\me}^t \sigma_{\tau}^{k+1}] \!=\! & -\! \int\! \left(\kappa| \Delta S_{\bm\me}^t \rho_{\tau}^{k+1}|^2\! +\! | \Delta S_{\bm\me}^t \eta_{\tau}^{k+1} |^2\!+\!2\alpha \Delta S_{\bm\me}^t \rho_{\tau}^{k+1}\Delta S_{\bm\me}^t \eta_{\tau}^{k+1} \right)\dx 
            \\
            &+\!\!\int\!\! \!\left(\beta| \nabla S_{\bm\me}^t \rho_{\tau}^{k+1}|^2 \! + \! | \nabla S_{\bm\me}^t \eta_{\tau}^{k+1} |^2\!+\!2\omega \nabla S_{\bm\me}^t \rho_{\tau}^{k+1}\cdot\nabla S_{\bm\me}^t \eta_{\tau}^{k+1} \right)\!\dx.
    \end{align*}
    By applying Young's inequality, we obtain
    \begin{align}
\label{eq:Heat flow interchange two species}
        \begin{split}
           - \ddt \bm\mf[S_{\bm\me}^t \sigma_{\tau}^{k+1}] & \geq  \int_{\Rd} \left((\kappa\!-\!|\alpha|\varepsilon)| \Delta S_{\bm\me}^t \rho_{\tau}^{k+1}|^2 \! + \! (1\!-\!|\alpha|\varepsilon^{-1})| \Delta S_{\bm\me}^t \eta_{\tau}^{k+1} |^2 \right)\dx 
            \\
            &-\int_{\Rd}\! \left((\beta\!+\!|\omega|)| \nabla S_{\bm\me}^t \rho_{\tau}^{k+1}|^2 \! +\! (1\!+\!|\omega|)| \nabla S_{\bm\me}^t \eta_{\tau}^{k+1} |^2\right)\dx,
       \end{split}
    \end{align}
    where $\varepsilon$ can be chosen such that $\kappa - |\alpha|\varepsilon>0$ and $1-|\alpha|\varepsilon^{-1}>0$.
    Therefore, combining \eqref{eq:Apply EVI two species}, \eqref{eq:compute DE two species} and \eqref{eq:Heat flow interchange two species} we obtain
    \begin{align*}
        & \tau \liminf_{s \downarrow 0} \int_0^1 \int_{\Rd} \left((\kappa-|\alpha|\varepsilon)| \Delta S_{\bm\me}^{st} \rho_{\tau}^{k+1}|^2 + (1-|\alpha|\varepsilon^{-1})| \Delta S_{\bm\me}^{st}  \eta_{\tau}^{k+1} |^2 \right)\,\dx   \, \mathrm{d}t 
        \\ &\qquad \leq \tau \limsup_{s \downarrow 0} \int_0^1\int_{\Rd} \left((\beta+|\omega|)| \nabla S_{\bm\me}^{st}  \rho_{\tau}^{k+1}|^2 + (1+|\omega|)| \nabla S_{\bm\me}^{st} \eta_{\tau}^{k+1} |^2\right)\dx  \mathrm{d}t
        \\ 
        & \qquad +  \bm\me[\sigma_{\tau}^k] - \bm\me[\sigma_{\tau}^{k+1}].
    \end{align*}
    Next, we recognize $\nabla S_{\bm\me}^{st}\sigma_\tau^{k+1}$ as the solution of the system of heat equations with initial data $\nabla\sigma_\tau^{k+1}\in L^2(\Rd)^2$. Hence, $\nabla S_{\bm\me}^{st}\sigma_\tau^{k+1}\rightarrow\nabla\sigma_\tau^{k+1}$ in $L^2(\Rd)^2$ as $s\downarrow 0$. In particular,
    \begin{align*}
        \limsup_{s \downarrow 0} \int_0^1\int_{\Rd} | \nabla S_{\bm\me}^{st}  \sigma_{\tau}^{k+1}|^2\,\dx\,\dt = \int_{\Rd} | \nabla \sigma_{\tau}^{k+1}|^2\,\dx
    \end{align*}
    Moreover, by well-known properties of the heat equation and the weak lower semicontinuity of the $H^1$ seminorm we have
    \begin{align*}
        &\liminf_{s \downarrow 0} \int_0^1\int_\Rd |\Delta S_{\bm\me}^{st} \rhot^{k+1}|^2 + |\Delta S_{\bm\me}^{st}\eta_\tau^{k+1}|^2\,\dx\,\dt \\
        & \qquad = \liminf_{s \downarrow 0} \int_0^1\int_\Rd | D^2 S_{\bm\me}^{st}\rho_\tau^{k+1}|^2+| D^2 S_{\bm\me}^{st}\eta_\tau^{k+1}|^2  \,\dx\,\dt
        \\
        & \qquad \geq \int_\Rd | D^2\rho_\tau^{k+1}|^2 + | D^2\eta_\tau^{k+1}|^2 \,\dx\,.
    \end{align*}
    Thus we have found
    \begin{align*}
        &\tau\Vert D^2\rho_\tau^{k+1}\Vert^2_{L^2(\Rd)}+\tau\Vert D^2\eta_\tau^{k+1}\Vert^2_{L^2(\Rd)} \\
        &\quad\leq C\left(\bm\me[\sigma_{\tau}^k] - \bm\me[\sigma_{\tau}^{k+1}]\right)
         + C\left(\tau\Vert\nabla\rho_\tau^{k+1}\Vert^2_{L^2(\Rd)}+\tau\Vert\nabla\eta_\tau^{k+1}\Vert^2_{L^2(\Rd)}\right),
    \end{align*}
    for a constant $C = C(\kappa,\alpha,\beta,\omega)$ independent of $\tau$.  By summing up over $k$ from $0$ to $N-1$ we obtain the desired $H^2$ bound by using~\Cref{lem: interpolation narrow convergence two species} since we have:
    \begin{align*}
        &\Vert D^2\rho_\tau\Vert^2_{L^2 ([0,T];L^2(\Rd))}+\Vert D^2\eta_\tau\Vert^2_{L^2 ([0,T];L^2(\Rd))}\\
        &\quad\leq C\left(\bm\me[\sigma_{0}] - \bm\me[\sigma_{\tau}^{N}]\right)+ C\left(\Vert\nabla\rho_\tau\Vert^2_{L^2 ([0,T];L^2(\Rd))}+\Vert\nabla\eta_\tau\Vert^2_{L^2 ([0,T];L^2(\Rd))}\right). 
    \end{align*}
    \end{proof}

The obtained $H^2$ bound allows us to obtain a two-species analogous of \Cref{prop: interpolation strong convergence gradient}. 
\begin{prop}[Strong convergence of $\nabla\sigma_\tau$]\label{prop: interpolation strong convergence gradient two species} Let $\sigma_0$ such that $\bm\mf[\sigma_0]<+\infty$. Up to a subsequence, the sequence $\sigma_\tau:[0,T]\rightarrow\mathcal{P}_2(\Rd)^2$ converges strongly to the curve $\Tilde{\sigma}$ in $L^2([0,T]; H^1(\Rd))^2$. 
\end{prop}
\begin{proof}
    The result follows by applying \Cref{prop: interpolation strong convergence gradient} to $\nabla\rho_\tau$ and $\nabla\eta_\tau$ together with the uniform $H^2$ bound derived in \Cref{H2 bound flow interchange two species}.
\end{proof}

\subsection{Consistency of the scheme}
Now we are ready to prove that $\Tilde{\sigma} = (\Tilde{\rho} , \Tilde{\eta} )$ is a weak solution of the problem \eqref{eq:two species} in the sense of \Cref{def:Weak solution two species}. This subsection completes the proof of Theorem \ref{thm:main_result_existence two species}.

\begin{proof}[Proof of Theorem~\ref{thm:main_result_existence two species}]
We prove the theorem by showing that the sequence $\sigma_\tau:[0,T]\rightarrow\mathcal{P}_2(\Rd)^2$ converges, up to a subsequence, to a weak solution $\Tilde{\sigma}$ of \eqref{eq:two species}.
    We will only prove the consistency for the first equation \eqref{eq: two species rho}. The case \eqref{eq: two species eta} will work analogously. Let us fix two consecutive steps in the JKO scheme $\sigmatk = (\rhotk, \etatk)$, $\sigmatkk = (\rhotkk, \etatkk )$, and consider the perturbation $\sigma^{\varepsilon} = ( \rho^{\varepsilon}, \etatkk)$ where $\rho^{\varepsilon} = P^\varepsilon_\#\rhot^{k+1}$ given by $P^\varepsilon = \text{id} + \varepsilon \zeta$, where $\zeta$ is a vector field $\zeta\in C_c^\infty(\Rd;\Rd)$, and $\varepsilon\geq 0$. By applying the definition of the scheme we obtain,
    \begin{equation}\label{eq:consistency two species 1}
        \frac{1}{2\tau} \left( \frac{d_W^2 (\sigmatk, \sigma^{\varepsilon})- d_W^2 (\sigmatk, \sigmatkk)}{\varepsilon} \right) + \frac{\bm\mf[\sigma^{\varepsilon}] - \bm\mf[\sigmatkk]}{\varepsilon} \geq 0.
    \end{equation}
    We proceed now to analyse each one of the terms in \eqref{eq:consistency two species 1}.

    \textbf{Step 1:} Wasserstein distance terms. We first realise that 
    \begin{equation}\label{eq:Consistency wasserstein two species}
        \frac{d_W^2 (\sigmatk, \sigma^{\varepsilon})- d_W^2 (\sigmatk, \sigmatkk)}{2 \tau \varepsilon} = \frac{\mw_2^2 (\rhotk , \rho^{\varepsilon})- \mw_2^2 (\rhotk, \rhotkk )}{2 \tau \varepsilon}.
    \end{equation}
    Therefore, Step 1 of the proof of \Cref{thm:main_result_existence} applies to this case. Let $\mathcal{T}$ be the optimal map between $\rhotk$ and $\rhotkk$, then
    \begin{equation*}
        \frac{d_W^2 (\sigmatk, \sigma^{\varepsilon})- d_W^2 (\sigmatk, \sigmatkk)}{2 \tau \varepsilon} \leq - \frac{1}{\tau} \int_{\Rd} (x - \mathcal{T} (x)) \cdot \zeta (\mathcal{T} (x)) \rhotk (x) \dx + O(\varepsilon ).
    \end{equation*}

    \textbf{Step 2:} Self-aggregation and self-diffusion terms. As in the one-species case; cf.~\Cref{thm:main_result_existence}, we have
    \begin{equation}\label{eq:Consistency diffusion two species}
     - \int_{\Rd} \frac{(\rho^{\varepsilon})^2 - (\rhotkk)^2}{\varepsilon} = \int_{\Rd} (\rhotkk)^2 (\dive \,\zeta ) \,\dx + O(\varepsilon)
    \end{equation}    
    and
    \begin{align}\label{eq:Consistency aggregation two species}
        \begin{split}
            & \frac{1}{2}\int_\Rd\frac{|\nabla \rho^{\varepsilon}|^2-|\nabla\rhot^{k+1}|^2}{\varepsilon}\,\dx \qquad\\
            & \quad = - \int_\Rd \Big(\rhotkk \nabla(\dive\, \zeta) \cdot \nabla \rhot^{k+1} + (\nabla \zeta \nabla\rhot^{k+1}) \cdot \nabla\rhot^{k+1} \\
            & \quad \qquad  \qquad+ \frac{1}{2}\dive\,\zeta |\nabla\rhot^{k+1}|^2 \Big)\,\dx + O(\varepsilon).
        \end{split}
    \end{align}

    \textbf{Step 3:} Cross-interaction terms.  For the second-order term we use the area formula to obtain,
      \begin{align}\label{eq:Consistency cross-interaction 1 two species}
      \begin{split}
        \int_\Rd\frac{\rho^{\varepsilon}(x) - \rhot^{k+1}(x)}{\varepsilon}\eta_\tau^{k+1}(x)\,\dx & = \int_\Rd\rhot^{k+1}(x)\frac{\eta_\tau^{k+1}(P^\varepsilon(x))-\eta_\tau^{k+1}(x)}{\varepsilon}\,\dx
        \\
        & = \int_\Rd\rhot^{k+1}(x)\nabla\eta_\tau^{k+1}(x)\cdot\zeta(x)\,\dx + O(\varepsilon).
    \end{split}
    \end{align}
    Similarly, for the fourth-order term, we use the fact that $\nabla\etatkk(P^\varepsilon(x)) = \nabla\etatkk(x) + \varepsilon D^2\etatkk(x)\zeta(x) + O(\varepsilon^2)$, and argue as in the one-species case to obtain
    \begin{align}\label{eq:Consistency cross-interaction 2 two species}
      \begin{split}
       & \int_\Rd \frac{\nabla\rho^\varepsilon -\nabla\rhot^{k+1}}{\varepsilon}\cdot\nabla\eta_\tau^{k+1}\,\dx 
        \\ 
         & \quad = -\int_\Rd \Big(\rhot^{k+1}\nabla(\dive\,\zeta)\cdot\nabla\eta_\tau^{k+1} + \nabla\rhotkk\cdot(\nabla\zeta\nabla\etatkk) \\
         & \quad \qquad \qquad - \nabla\rho_\tau^{k+1}\cdot( D^2\eta_\tau^{k+1}\zeta)\Big)\,\dx+ O(\varepsilon).
    \end{split}
    \end{align}
  \textbf{Step 4:} Taking the limit $\varepsilon \rightarrow 0$. Analogously to the one species case we perform the same computation for $\varepsilon \leq 0$ and we take again $\zeta = \nabla \varphi$. If we consider $\varepsilon \rightarrow 0$, and thanks to \eqref{eq:Consistency wasserstein two species}, \eqref{eq:Consistency diffusion two species}, \eqref{eq:Consistency aggregation two species}, \eqref{eq:Consistency cross-interaction 1 two species}, and \eqref{eq:Consistency cross-interaction 2 two species}, we have,
  \begin{align}
      & \frac{1}{\tau} \int_{\Rd} (x - \mathcal{T} (x)) \cdot \nabla \varphi (\mathcal{T} (x)) \rhot^k (x) \,\dx 
     \nonumber \\
      & = \!-\kappa\int\left(\rhotkk\nabla(\Delta \varphi)\cdot\nabla\rhot^{k+1} + ( D^2\varphi \nabla\rhot^{k+1}) \cdot \nabla\rhot^{k+1} + \frac{1}{2}\Delta \varphi |\nabla\rhot^{k+1}|^2\right)\,\dx 
    \nonumber  \\
      & \quad \!-\!\alpha\!\!\int\!\!\! \left(\rhot^{k+1}\nabla(\Delta\varphi)\!\cdot\!\nabla\eta_\tau^{k+1} \!+\! \nabla\rhot^{k+1}\!\cdot\!(D^2\varphi\nabla\eta_\tau^{k+1}) \!-\! \nabla\rho_\tau^{k+1}\!\cdot\!( D^2\eta_\tau^{k+1}\nabla\varphi)\right)\!\dx
     \nonumber \\
      & \quad + \frac{\beta}{2} \int (\rhotkk)^2 \Delta \varphi \,\dx  - \omega\int\rhot^{k+1}\nabla\eta_\tau^{k+1}\cdot\nabla\varphi\,\dx . \label{eq:weak formulation two species}
  \end{align}
As in the one-species case, and using the Holder continuity of $\rhot$, \eqref{eq:equicontinuity on d2 two species}, we have
\begin{equation*}
    \int_\Rd (x-\mathcal{T}(x))\cdot\nabla \varphi (\mathcal{T}(x))\rhot^k(x)\,\dx = \int_{\Rd} \varphi (x) \left[ \rhot^k(x) - \rhot^{k+1}(x) \right] \, \dx + O(\tau )\,.
\end{equation*}
Let $0 \leq s_1 < s_2 \leq T$ be fixed with,
\begin{equation*}
    h_1 = \left[ \frac{s_1}{\tau} \right] +1 \quad \text{and} \quad h_2 = \left[ \frac{s_2}{\tau} \right].
\end{equation*}
By summing on \eqref{eq:weak formulation two species} and using the definition of piecewise interpolation, we obtain,
 \begin{align}\label{eq:Consistency tau limit two species}
        &\int_{\Rd} \varphi (x) \rhot (s_2, x) \, \dx - \int_{\Rd} \varphi (x) \rhot (s_1, x) \, \dx + O(\tau)
     \nonumber \\
      & = \kappa\int_{s_1}^{s_2} \int_\Rd\left(\rhot\nabla(\Delta \varphi)\cdot\nabla\rhot + ( D^2\varphi \nabla\rhot) \cdot \nabla\rhot+\frac{1}{2} \Delta \varphi |\nabla\rhot|^2\right)\,\dx \,\dt
    \nonumber  \\
      & \quad +\alpha\int_{s_1}^{s_2} \int_\Rd \left(\rhot\nabla(\Delta\varphi)\cdot\nabla\eta_\tau + \nabla\rhot\cdot(D^2\varphi\nabla\eta_\tau) - \nabla\rho_\tau\cdot( D^2\eta_\tau\nabla\varphi)\right)\!\dx\dt
     \nonumber \\
      & \quad - \frac{\beta}{2} \int_{s_1}^{s_2} \int_{\Rd} \rhot^2 \Delta \varphi \,\dx \,\dt + \omega\int_{s_1}^{s_2} \int_\Rd\rhot\nabla\eta_\tau\cdot\nabla\varphi\,\dx \,\dt .
  \end{align}
Integrating by parts in the first two terms after the equality, as in~Remarks~\ref{rem:integration_part_weak_form}~and~\ref{rem:integ_parts_two_species}, we obtain
\begin{align*}
    \int_{\Rd}  \varphi (x) \rhot (s_2, x) \, \dx & = \int_{\Rd} \varphi (x) \rhot (s_1, x) \, \dx +O(\tau)\\
    & \quad -\kappa\int_{s_1}^{s_2}  \int_\Rd\left(\rhot \Delta \rhot \Delta \varphi + \Delta \rhot \nabla \rhot \cdot \nabla \varphi \right)\,\dx \, \dt 
    \\
     & \quad -\alpha\int_{s_1}^{s_2}  \int_\Rd\left(\rhot\Delta\etat\Delta\varphi+\Delta\etat\nabla\rhot\cdot\nabla\varphi \right)\,\dx \, \dt
     \\
    & \quad - \frac{\beta}{2} \int_{s_1}^{s_2} \!\int_{\Rd} \rhot^2 \Delta \varphi \dx \dt + \omega\int_{s_1}^{s_2}\! \int_\Rd\rhot\nabla\etat\cdot\nabla\varphi\dx \dt.
\end{align*}
By combining \Cref{prop: strong convergence rho two species}, \Cref{H2 bound flow interchange two species}, and \Cref{prop: interpolation strong convergence gradient two species} we can pass to the limit as $\tau \rightarrow 0^+$, and, in this way, recover a weak solution. As aforementioned, an analogous argument for the species $\eta$ can be repeated to obtain~\eqref{eq: two species eta}.
\end{proof}

\begin{rem}\label{rem:integ_parts_two_species} Assume $\rho,\eta\in H^2(\Rd)$ and $\varphi\in C_0^3(\Rd)$. Using integration by parts, we have
    \begin{align*}
       & \int_\Rd \left(\rho\nabla\Delta\varphi\cdot\nabla\eta + \nabla\rho\cdot( D^2\varphi\nabla\eta) - \nabla\rho\cdot( D^2\eta\nabla\varphi)\right)\,\dx
       \\
       & = - \int_\Rd\rho\Delta\eta\Delta\varphi\,\dx +\int_\Rd\left( \nabla\rho\cdot( D^2\varphi\nabla\eta)-\nabla\rho\cdot( D^2\eta\nabla\varphi)-\Delta\varphi\nabla\rho\cdot\nabla\eta\right)\,\dx
       \\
           & =- \int_\Rd\rho\Delta\eta\Delta\varphi\,\dx + \int_\Rd \left(\nabla\rho\cdot(D^2\varphi\nabla\eta)+\nabla\varphi\cdot( D^2\rho\nabla\eta)\right)\,\dx
       \\
       & =- \int_\Rd\left(\rho\Delta\eta\Delta\varphi+\Delta\eta\nabla\rho\cdot\nabla\varphi\right)\,\dx\,.
    \end{align*}
\end{rem}

\subsection{Extension to generalised self-diffusion systems}
In this subsection we remark that, taking advantage of the one and two species cases, we can generalise the existence theory to the following system with nonlinear self-diffusion terms
\begin{subequations}\label{eq:two species general}    
\begin{align}
        \partial_t\rho &= -\dive\left(\rho\nabla\left(\kappa\Delta\rho + \alpha\Delta\eta + \frac{\beta}{m_1-1}\rho^{m_1 - 1} +\omega\eta\right)\right), \label{eq: two species general rho}
        \\
        \partial_t\eta & =-\dive\left(\eta\nabla\left(\alpha\Delta\rho + \Delta\eta + \omega\rho+ \frac{1}{m_2 - 1} \eta^{m_2 -1} \right)\right), \label{eq: two species general eta}
\end{align}
\end{subequations}
where $1\leq m_1, m_2 < 2 + \frac{2}{d}$. As before, the parameters in the model are such that $\beta,\omega\in\mathbb{R}$ and the matrix 
\begin{equation*}
   A = \begin{pmatrix}
    \kappa & \alpha
    \\
    \alpha & 1
    \end{pmatrix},
\end{equation*}
is positive definite.
In this case, we consider
\begin{align*}
    \Tilde{\bm{\mf}}_{m_1, m_2}[\rho,\eta] & = \int_{\Rd}\left(\frac{\kappa}{2}|\nabla\rho|^2+\frac{1}{2}|\nabla\eta|^2 + \alpha \nabla\rho \cdot \nabla\eta  - \omega\rho\eta\right)\dx \\
    & \quad - \frac{\beta}{m_1} \me_{m_1}[\rho] - \frac{1}{m_2} \me_{m_2} [\eta],
\end{align*}
where $\me_m$ is the entropy defined in \eqref{eq:entropy}.
The system of equations above can be written as a 2-Wasserstein gradient flow with respect to the (extended) free energy functional
\begin{equation*}
    \bm{\mf}_{m_1, m_2}[\rho,\eta] = \begin{cases}
      \Tilde{\bm{\mf}}_{m_1, m_2}[\rho,\eta] & \mbox{if } (\rho,\eta)\in\mathcal{P}^a(\mathbb{R}^d)^2,\, (\nabla\rho,\nabla\eta) \in L^2(\mathbb{R}^d)^2,
        \\
        +\infty & \mbox{otherwise.} 
    \end{cases}
\end{equation*}
We can obtain the following lower bound for the free energy:
\begin{align*}
    {\bm{\mf}}_{m_1, m_2}[\rho,\eta] &=\int_{\Rd}\left(\frac{\kappa}{2}|\nabla\rho|^2+\frac{1}{2}|\nabla\eta|^2 + \alpha\nabla\rho\cdot\nabla\eta -\omega \rho\eta \right)\,\dx
    \\
    & \quad - \frac{\beta}{m_1} \me_{m_1}[\rho] - \frac{1}{m_2} \me_{m_2}[\eta] 
     \\
     &  \geq \int_{\Rd}\left(\frac{\kappa}{2}|\nabla\rho|^2+\frac{1}{2}|\nabla\eta|^2 - |\alpha||\nabla\rho||\nabla\eta| - \frac{|\omega|}{2}\rho^2 - \frac{|\omega|}{2} \eta^2 \right)\,\dx
     \\
     &\quad  - \frac{\beta}{m_1}\me_{m_1}[\rho] - \frac{1}{m_2} \me_{m_2}[\eta] 
     \\
     &  \geq \int_{\Rd} \left( \frac{\kappa - |\alpha|\varepsilon}{4}|\nabla\rho|^2 -\frac{|\omega|}{2}\rho^2 \right) \, \dx \\
     & \quad + \int_{\Rd} \left( \frac{1 - |\alpha|\varepsilon^{-1}}{4}|\nabla\eta|^2 -\frac{|\omega|}{2}\eta^2 \right) \, \dx 
     \\
     & \quad + \int_{\Rd} \frac{\kappa - |\alpha|\varepsilon}{4}|\nabla\rho|^2 \, \dx - \frac{\beta}{m_1}\me_{m_1}[\rho] 
     \\
     &\quad + \int_{\Rd}  \frac{1 - |\alpha|\varepsilon^{-1}}{4}|\nabla\eta|^2 \, \dx - \frac{1}{m_2} \me_{m_2}[\eta] .
\end{align*}
Therefore, since we can take $\varepsilon$ such that $\kappa - |\alpha| \varepsilon, \, 1 - |\alpha| \varepsilon^{-1} > 0$, it follows that
\begin{equation}\label{eq: energy inequality general system}
    {\bm{\mf}}_{m_1, m_2}[\rho,\eta] \geq C \left( \mf_2 [\rho] + \mf_2[\eta] + \mf_{m_1} [\rho] + \mf_{m_2} [\eta] \right) .
\end{equation}
In particular, for $1\leq m_1, m_2 < 2 + \frac{2}{d}$, the free energy is bounded from below. Furthermore, \eqref{eq: energy inequality general system} gives the basic estimates that we used for the existence of the one and two species cases. Since the cross-interacting terms are kept as in \eqref{eq:two species} and the new terms with exponents $m_1$ and $m_2$ have already been treated on the one species case, our previous results can be easily generalised to obtain existence for the problem \eqref{eq:two species general}. 

In addition to that, using a scaling argument, we can show that the free energy is unbounded from below if $m_1 > 2 + \frac{2}{d}$, or equally $m_2>2 + \frac{2}{d}$. Without loss of generality we state the result for $m_1$. A thorough analysis of more general systems, as well as the other cases for the exponents, will be object of further investigation, as it is beyond the purpose of the current manuscript.

\begin{prop}
    Assume  $m_1 > m_c$ and denote
    $$    \mathcal{Y} := \left\lbrace (\rho, \eta) \in \mathcal{P}^a(\Rd)^2\cap L^{m_1}(\Rd) \times L^{m_2}(\Rd) : \nabla \rho, \nabla \eta \in L^2(\Rd)\right\rbrace.$$ Then
    \begin{equation*}
        \inf_{(\rho, \eta) \in \mathcal{Y}} \bm{\mf}_{m_1, m_2} [\rho , \eta] = - \infty,
    \end{equation*}
\end{prop}

\begin{proof}
    Given $(\rho, \eta) \in \mathcal{Y}$ we define $\rho_{\lambda} (x) := \lambda^d \rho (\lambda x)$, for any $x \in \Rd$ and any $\lambda \in (0, + \infty)$. Note that $(\rho_{\lambda} , \eta ) \in \mathcal{Y}$. Then, we have, 
    \begin{align*}
        \bm{\mf}_{m_1, m_2} [\rho_{\lambda} , \eta] & = \frac{\kappa}{2} \lambda^{d+2} \| \nabla \rho \|_{L^2(\Rd)}^2 - \lambda^{d (m_1-1)} \frac{\beta}{m_1(m_1 -1)} \| \rho \|_{L^{m_1}(\Rd)}^{m_1} \\
        & \quad + \frac{1}{2} \| \nabla \eta \|_{L^2(\Rd)}^2 - \frac{1}{m_2} \me_{m_2}[\eta ] \\
        & \quad + \int_{\Rd} \alpha \lambda^d \nabla \rho (\lambda x) \cdot \nabla \eta (x) \, \dx - \int_{\Rd} \omega \lambda^d \rho (\lambda x) \eta (x) \, \dx \\
        & = \lambda^{d+2} \left( \frac{\kappa}{2}  \| \nabla \rho \|_{L^2(\Rd)}^2 -  \lambda^{d (m_1-m_c)} \frac{\beta}{m_1(m_1 -1)} \| \rho \|_{L^{m_1}(\Rd)}^{m_1} \right) \\
        & \quad + \frac{1}{2} \| \nabla \eta \|_{L^2(\Rd)}^2 - \frac{1}{m_2 } \me_{m_2}[\eta]  \\
        & \quad + \int_{\Rd} \alpha \lambda^d \nabla \rho (\lambda x) \cdot \nabla \eta (x) \, \dx - \int_{\Rd} \omega \lambda^d \rho (\lambda x) \eta (x) \, \dx .
    \end{align*}
    Therefore, if we take $\rho$ and $\eta$ such that $\lambda \times \supp (\rho ) \cap \supp (\eta) = \emptyset$ for big enough $\lambda$ it follows that $\bm{\mf}_{m_1, m_2} [\rho_{\lambda} , \eta] \rightarrow - \infty$ when $\lambda \rightarrow + \infty$; for instance we could consider the support of $\rho$ to be an annulus and that of $\eta$ to be a ball.
\end{proof}

\subsection*{Acknowledgements}
The authors were supported by the Advanced Grant Nonlocal-CPD (Nonlocal PDEs for Complex Particle Dynamics: Phase Transitions, Patterns and Synchronization) of the European Research Council Executive Agency (ERC) under the European Union’s Horizon 2020 research and innovation programme (grant agreement No. 883363). CF acknowledges support of a fellowship from "la Caixa" Foundation (ID 100010434) with code LCF/BQ/EU21/11890128.

\bibliography{references}

\begin{thebibliography}{10}

\bibitem{AGS}
L.~Ambrosio, N.~Gigli, and G.~Savar\'{e}.
\newblock {\em Gradient flows in metric spaces and in the space of probability
  measures}.
\newblock Lectures in Mathematics ETH Z\"{u}rich. Birkh\"{a}user Verlag, Basel,
  second edition, 2008.

\bibitem{bailo2021unconditional}
R.~Bailo, J.~A. Carrillo, S.~Kalliadasis, and S.~P. Perez.
\newblock {Unconditional bound-preserving and energy-dissipating finite-volume
  schemes for the Cahn-Hilliard equation}.
\newblock {\em arXiv preprint arXiv:2105.05351}, 2021.

\bibitem{Bedrossian11}
J.~Bedrossian.
\newblock Intermediate asymptotics for critical and supercritical aggregation
  equations and {P}atlak-{K}eller-{S}egel models.
\newblock {\em Commun. Math. Sci.}, 9(4):1143--1161, 2011.

\bibitem{Bedrossian_Rodriguez_Bertozzi11}
J.~Bedrossian, N.~Rodr\'{\i}guez, and A.~L. Bertozzi.
\newblock Local and global well-posedness for aggregation equations and
  {P}atlak-{K}eller-{S}egel models with degenerate diffusion.
\newblock {\em Nonlinearity}, 24(6):1683--1714, 2011.

\bibitem{BernoffTopazCH}
A.~Bernoff and C.~Topaz.
\newblock {Biological aggregation driven by social and environmental factors: a
  nonlocal model and its degenerate Cahn--Hilliard approximation}.
\newblock {\em SIAM Journal on Applied Dynamical Systems}, 15(3):1528--1562, 07
  2015.

\bibitem{Ber98}
A.~L. Bertozzi.
\newblock The mathematics of moving contact lines in thin liquid films.
\newblock {\em Notices Amer. Math. Soc.}, 45(6):689--697, 1998.

\bibitem{Bertozzi_Pugh_Nonlinearity_94}
A.~L. Bertozzi and M.~Pugh.
\newblock The lubrication approximation for thin viscous films: the moving
  contact line with a ``porous media'' cut-off of van der {W}aals interactions.
\newblock {\em Nonlinearity}, 7(6):1535--1564, 1994.

\bibitem{Bertozzi_Pugh_CPAM98}
A.~L. Bertozzi and M.~C. Pugh.
\newblock Long-wave instabilities and saturation in thin film equations.
\newblock {\em Comm. Pure Appl. Math.}, 51(6):625--661, 1998.

\bibitem{BCC12}
A.~Blanchet, E.~A. Carlen, and J.~A. Carrillo.
\newblock Functional inequalities, thick tails and asymptotics for the critical
  mass {P}atlak-{K}eller-{S}egel model.
\newblock {\em J. Funct. Anal.}, 262(5):2142--2230, 2012.

\bibitem{blanchet2009critical}
A.~Blanchet, J.~A. Carrillo, and P.~Lauren{\c{c}}ot.
\newblock {Critical mass for a Patlak--Keller--Segel model with degenerate
  diffusion in higher dimensions}.
\newblock {\em Calculus of Variations and Partial Differential Equations},
  35(2):133--168, 2009.

\bibitem{BDP06}
A.~Blanchet, J.~Dolbeault, and B.~Perthame.
\newblock Two-dimensional {K}eller-{S}egel model: optimal critical mass and
  qualitative properties of the solutions.
\newblock {\em Electron. J. Differential Equations}, pages No. 44, 32, 2006.

\bibitem{Bre11}
H.~Br{\'e}zis.
\newblock {\em Functional analysis, Sobolev spaces and partial differential
  equations}, volume~2.
\newblock Springer, 2011.

\bibitem{Brezis_Mironescu_GN_2018}
H.~Brezis and P.~Mironescu.
\newblock Gagliardo-{N}irenberg inequalities and non-inequalities: the full
  story.
\newblock {\em Ann. Inst. H. Poincar\'{e} C Anal. Non Lin\'{e}aire},
  35(5):1355--1376, 2018.

\bibitem{BE22}
M.~Burger and A.~Esposito.
\newblock Porous medium equation and cross-diffusion systems as limit of
  nonlocal interaction.
\newblock {\em Nonlinear Analysis}, 235:113347, 2023.

\bibitem{Calvez_Carrillo06}
V.~Calvez and J.~A. Carrillo.
\newblock Volume effects in the {K}eller-{S}egel model: energy estimates
  preventing blow-up.
\newblock {\em J. Math. Pures Appl. (9)}, 86(2):155--175, 2006.

\bibitem{Calvez_Carrillo_Hoffmann17a}
V.~Calvez, J.~A. Carrillo, and F.~Hoffmann.
\newblock Equilibria of homogeneous functionals in the fair-competition regime.
\newblock {\em Nonlinear Anal.}, 159:85--128, 2017.

\bibitem{Calvez_Carrillo_Hoffmann17b}
V.~Calvez, J.~A. Carrillo, and F.~Hoffmann.
\newblock The geometry of diffusing and self-attracting particles in a
  one-dimensional fair-competition regime.
\newblock In {\em Nonlocal and nonlinear diffusions and interactions: new
  methods and directions}, volume 2186 of {\em Lecture Notes in Math.}, pages
  1--71. Springer, Cham, 2017.

\bibitem{carrillo2015ground}
J.~A. Carrillo, D.~Castorina, and B.~Volzone.
\newblock Ground states for diffusion dominated free energies with logarithmic
  interaction.
\newblock {\em SIAM Journal on Mathematical Analysis}, 47(1):1--25, 2015.

\bibitem{Carrillo_Craig_Yao19}
J.~A. Carrillo, K.~Craig, and Y.~Yao.
\newblock Aggregation-diffusion equations: dynamics, asymptotics, and singular
  limits.
\newblock In {\em Active particles. {V}ol. 2. {A}dvances in theory, models, and
  applications}, Model. Simul. Sci. Eng. Technol., pages 65--108.
  Birkh\"{a}user/Springer, Cham, 2019.

\bibitem{CdFEFS20}
J.~A. Carrillo, M.~Di~Francesco, A.~Esposito, S.~Fagioli, and M.~Schmidtchen.
\newblock {Measure solutions to a system of continuity equations driven by
  Newtonian nonlocal interactions}.
\newblock {\em Discrete Continuous Dynamical Systems}, 40(2):1191--1231, 2020.

\bibitem{carrillo2023degenerate}
J.~A. Carrillo, C.~Elbar, and J.~Skrzeczkowski.
\newblock Degenerate cahn-hilliard systems: From nonlocal to local.
\newblock {\em arXiv preprint arXiv:2303.11929}, 2023.

\bibitem{JAC_ESP_WU2023nonlocal}
J.~A. Carrillo, A.~Esposito, and J.~S.-H. Wu.
\newblock Nonlocal approximation of nonlinear diffusion equations.
\newblock {\em arXiv preprint arXiv:2302.08248}, 2023.

\bibitem{CHMV18}
J.~A. Carrillo, F.~Hoffmann, E.~Mainini, and B.~Volzone.
\newblock Ground states in the diffusion-dominated regime.
\newblock {\em Calc. Var. Partial Differential Equations}, 57(5):Paper No. 127,
  28, 2018.

\bibitem{CarrilloKe21}
J.~A. Carrillo and K.~Lin.
\newblock Sharp conditions on global existence and blow-up in a degenerate
  two-species and cross-attraction system.
\newblock {\em Adv. Nonlinear Anal.}, 11(1):1--39, 2022.

\bibitem{carrillo2019population}
J.~A. Carrillo, H.~Murakawa, M.~Sato, H.~Togashi, and O.~Trush.
\newblock A population dynamics model of cell-cell adhesion incorporating
  population pressure and density saturation.
\newblock {\em J Theor Biol}, 474:14--24, 2019.

\bibitem{CS18}
J.-A. Carrillo and F.~Santambrogio.
\newblock {$L^\infty$} estimates for the {JKO} scheme in parabolic-elliptic
  {K}eller-{S}egel systems.
\newblock {\em Quart. Appl. Math.}, 76(3):515--530, 2018.

\bibitem{CW14}
L.~Chen and J.~Wang.
\newblock Exact criterion for global existence and blow up to a degenerate
  {K}eller-{S}egel system.
\newblock {\em Doc. Math.}, 19:103--120, 2014.

\bibitem{DalPasso_Giacomelli_Shishkov_CPDE01}
R.~Dal~Passo, L.~Giacomelli, and A.~Shishkov.
\newblock The thin film equation with nonlinear diffusion.
\newblock {\em Comm. Partial Differential Equations}, 26(9-10):1509--1557,
  2001.

\bibitem{DS08}
S.~Daneri and G.~Savar{\'e}.
\newblock {Eulerian calculus for the displacement convexity in the Wasserstein
  distance}.
\newblock {\em SIAM Journal on Mathematical Analysis}, 40(3):1104--1122, 2008.

\bibitem{dFEF18}
M.~Di~Francesco, A.~Esposito, and S.~Fagioli.
\newblock Nonlinear degenerate cross-diffusion systems with nonlocal
  interaction.
\newblock {\em Nonlinear Analysis}, 169:94--117, 2018.

\bibitem{dFM14}
M.~Di~Francesco and D.~Matthes.
\newblock Curves of steepest descent are entropy solutions for a class of
  degenerate convection--diffusion equations.
\newblock {\em Calculus of Variations and Partial Differential Equations},
  50(1):199--230, 2014.

\bibitem{DP04}
J.~Dolbeault and B.~Perthame.
\newblock Optimal critical mass in the two-dimensional {K}eller-{S}egel model
  in {$\Bbb R^2$}.
\newblock {\em C. R. Math. Acad. Sci. Paris}, 339(9):611--616, 2004.

\bibitem{Ehrla_DiMa_Pietschamm21}
V.~Ehrlacher, G.~Marino, and J.-F. Pietschmann.
\newblock {Existence of weak solutions to a cross-diffusion Cahn-Hilliard type
  system}.
\newblock {\em Journal of Differential Equations}, 286:578--623, 2021.

\bibitem{elbar2022degenerate}
C.~Elbar and J.~Skrzeczkowski.
\newblock {Degenerate Cahn-Hilliard equation: From nonlocal to local}.
\newblock {\em Journal of Differential Equations}, 364:576--611, 2023.

\bibitem{elliott1997diffusional}
C.~M. Elliott and H.~Garcke.
\newblock Diffusional phase transitions in multicomponent systems with a
  concentration dependent mobility matrix.
\newblock {\em Physica D: Nonlinear Phenomena}, 109(3-4):242--256, 1997.

\bibitem{elliott1991generalised}
C.~M. Elliott and S.~Luckhaus.
\newblock A generalised diffusion equation for phase separation of a
  multi-component mixture with interfacial free energy.
\newblock 1991.

\bibitem{EGK07b}
J.~Evans, V.~Galaktionov, and J.~King.
\newblock Blow-up similarity solutions of the fourth-order unstable thin film
  equation.
\newblock {\em European Journal of Applied Mathematics}, 18(2):195--231, 2007.

\bibitem{EGK07a}
J.~Evans, V.~Galaktionov, and J.~King.
\newblock Source-type solutions of the fourth-order unstable thin film
  equation.
\newblock {\em European Journal of Applied Mathematics}, 18(3):273--321, 2007.

\bibitem{falco2022local}
C.~Falc{\'o}, R.~E. Baker, and J.~A. Carrillo.
\newblock A local continuum model of cell-cell adhesion.
\newblock {\em arXiv e-prints}, page arXiv:2206.14461, 2022.

\bibitem{FG21}
A.~Figalli and F.~Glaudo.
\newblock {\em An Invitation to Optimal Transport, Wasserstein Distances, and
  Gradient Flows}.
\newblock 2021.

\bibitem{Grun_CPDE04}
G.~Gr\"{u}n.
\newblock Droplet spreading under weak slippage---existence for the {C}auchy
  problem.
\newblock {\em Comm. Partial Differential Equations}, 29(11-12):1697--1744,
  2004.

\bibitem{HOCHERMAN_ROSENAU_93}
T.~Hocherman and P.~Rosenau.
\newblock On ks-type equations describing the evolution and rupture of a liquid
  interface.
\newblock {\em Physica D: Nonlinear Phenomena}, 67(1):113--125, 1993.

\bibitem{JKO98}
R.~Jordan, D.~Kinderlehrer, and F.~Otto.
\newblock The variational formulation of the {F}okker-{P}lanck equation.
\newblock {\em SIAM J. Math. Anal.}, 29(1):1--17, 1998.

\bibitem{Laugesen_Pugh_Arma00}
R.~S. Laugesen and M.~C. Pugh.
\newblock Linear stability of steady states for thin film and {C}ahn-{H}illiard
  type equations.
\newblock {\em Arch. Ration. Mech. Anal.}, 154(1):3--51, 2000.

\bibitem{Laugesen_Pugh_EJAM2000}
R.~S. Laugesen and M.~C. Pugh.
\newblock Properties of steady states for thin film equations.
\newblock {\em European J. Appl. Math.}, 11(3):293--351, 2000.

\bibitem{Laugesen_Pugh_JDE22}
R.~S. Laugesen and M.~C. Pugh.
\newblock Energy levels of steady states for thin-film-type equations.
\newblock {\em J. Differential Equations}, 182(2):377--415, 2002.

\bibitem{Lisini_Matthes_Savare12}
S.~Lisini, D.~Matthes, and G.~Savar\'{e}.
\newblock Cahn-{H}illiard and thin film equations with nonlinear mobility as
  gradient flows in weighted-{W}asserstein metrics.
\newblock {\em J. Differential Equations}, 253(2):814--850, 2012.

\bibitem{liu2017generalized}
J.-G. Liu and J.~Wang.
\newblock {A generalized Sz. Nagy inequality in higher dimensions and the
  critical thin film equation}.
\newblock {\em Nonlinearity}, 30(1):35, 2017.

\bibitem{LW17}
J.-G. Liu and J.~Wang.
\newblock Global existence for a thin film equation with subcritical mass.
\newblock {\em Discrete Contin. Dyn. Syst. Ser. B}, 22(4):1461--1492, 2017.

\bibitem{LS07}
S.~Luckhaus and Y.~Sugiyama.
\newblock Asymptotic profile with the optimal convergence rate for a parabolic
  equation of chemotaxis in super-critical cases.
\newblock {\em Indiana Univ. Math. J.}, 56(3):1279--1297, 2007.

\bibitem{MMS09}
D.~Matthes, R.~J. McCann, and G.~Savar{\'e}.
\newblock A family of nonlinear fourth order equations of gradient flow type.
\newblock {\em Communications in Partial Differential Equations},
  34(11):1352--1397, 2009.

\bibitem{Inequalities_book}
D.~S. Mitrinovi\'{c}, J.~E. Pe\v{c}ari\'{c}, and A.~M. Fink.
\newblock {\em Inequalities involving functions and their integrals and
  derivatives}, volume~53 of {\em Mathematics and its Applications (East
  European Series)}.
\newblock Kluwer Academic Publishers Group, Dordrecht, 1991.

\bibitem{Nirenberg_1959}
L.~Nirenberg.
\newblock On elliptic partial differential equations.
\newblock {\em Ann. Scuola Norm. Sup. Pisa Cl. Sci. (3)}, 13:115--162, 1959.

\bibitem{Otto_CPDE98}
F.~Otto.
\newblock Lubrication approximation with prescribed nonzero contact angle.
\newblock {\em Comm. Partial Differential Equations}, 23(11-12):2077--2164,
  1998.

\bibitem{RS03}
R.~Rossi and G.~Savar{\'e}.
\newblock Tightness, integral equicontinuity and compactness for evolution
  problems in banach spaces.
\newblock {\em Annali della Scuola Normale Superiore di Pisa-Classe di
  Scienze}, 2(2):395--431, 2003.

\bibitem{Roy88}
H.~L. Royden.
\newblock {\em Real analysis}.
\newblock Macmillan Publishing Company, New York, third edition, 1988.

\bibitem{San15}
F.~Santambrogio.
\newblock Optimal transport for applied mathematicians.
\newblock {\em Birk{\"a}user, NY}, 55(58-63):94, 2015.

\bibitem{Slepcev_IntFree_09}
D.~Slep\v{c}ev.
\newblock Linear stability of selfsimilar solutions of unstable thin-film
  equations.
\newblock {\em Interfaces Free Bound.}, 11(3):375--398, 2009.

\bibitem{Slepcev_Pugh_05}
D.~Slep\v{c}ev and M.~C. Pugh.
\newblock Selfsimilar blowup of unstable thin-film equations.
\newblock {\em Indiana Univ. Math. J.}, 54(6):1697--1738, 2005.

\bibitem{Sugiyama07}
Y.~Sugiyama.
\newblock Time global existence and asymptotic behavior of solutions to
  degenerate quasi-linear parabolic systems of chemotaxis.
\newblock {\em Differential Integral Equations}, 20(2):133--180, 2007.

\bibitem{Nagy_1941}
B.~v.~Sz.~Nagy.
\newblock \"{U}ber {I}ntegralungleichungen zwischen einer {F}unktion und ihrer
  {A}bleitung.
\newblock {\em Acta Univ. Szeged. Sect. Sci. Math.}, 10:64--74, 1941.

\bibitem{Vil03}
C.~Villani.
\newblock {\em Topics in optimal transportation}, volume~58 of {\em Graduate
  Studies in Mathematics}.
\newblock American Mathematical Society, Providence, RI, 2003.

\bibitem{Vil09}
C.~Villani.
\newblock {\em Optimal transport: old and new}, volume 338.
\newblock Springer, 2009.

\bibitem{weinstein1982nonlinear}
M.~I. Weinstein.
\newblock {Nonlinear Schr{\"o}dinger equations and sharp interpolation
  estimates}.
\newblock {\em Comm. Math. Phys.}, 87(4):567--576, 1983.

\bibitem{Witel_Bern_Bert_EJAM04}
T.~P. Witelski, A.~J. Bernoff, and A.~L. Bertozzi.
\newblock Blowup and dissipation in a critical-case unstable thin film
  equation.
\newblock {\em European J. Appl. Math.}, 15(2):223--256, 2004.

\bibitem{Yao14}
Y.~Yao.
\newblock Asymptotic behavior for critical {P}atlak-{K}eller-{S}egel model and
  a repulsive-attractive aggregation equation.
\newblock {\em Ann. Inst. H. Poincar\'{e} C Anal. Non Lin\'{e}aire},
  31(1):81--101, 2014.

\end{thebibliography}
\bibliographystyle{abbrv}

\end{document}